\documentclass[12pt,a4paper]{article}

\usepackage[utf8]{inputenc} 		
\usepackage[english]{babel}		
\usepackage{amsmath,amssymb,amsthm}	

\def\restriction#1#2{\mathchoice
              {\setbox1\hbox{${\displaystyle #1}_{\scriptstyle #2}$}
              \restrictionaux{#1}{#2}}
              {\setbox1\hbox{${\textstyle #1}_{\scriptstyle #2}$}
              \restrictionaux{#1}{#2}}
              {\setbox1\hbox{${\scriptstyle #1}_{\scriptscriptstyle #2}$}
              \restrictionaux{#1}{#2}}
              {\setbox1\hbox{${\scriptscriptstyle #1}_{\scriptscriptstyle #2}$}
              \restrictionaux{#1}{#2}}}
\def\restrictionaux#1#2{{#1\,\smash{\vrule height .6\ht1 depth .6\dp1}}_{\,#2}} 
\newcommand{\M}{\widetilde{\mathcal{M}}}
\newcommand{\MM}{\overline{\mathcal{M}}}
\newcommand{\z}{\boldsymbol{\mathrm{z}}}
\newcommand{\HH}{\mathcal{H}}
\newcommand{\g}{\widetilde{g}}
\newcommand{\kk}{\widetilde{k}}
\newcommand{\hh}{\widetilde{h}}
\usepackage{rotating}
\usepackage[]{stackengine}
\usepackage{dsfont}
\usepackage{bm}
\DeclareMathAlphabet{\mathmybb}{U}{bbold}{m}{n}
\newcommand*{\OO}{\mathmybb{\Omega}}

\newtheorem{theorem}{Theorem}[section]
\newtheorem{lemma}{Lemma}[section]
\newtheorem{proposition}{Proposition}[section]

\newtheorem{remark}{Remark}[section]

\usepackage{cancel}
\usepackage{tikz}
\usetikzlibrary{shapes,snakes}
\usetikzlibrary{arrows}
\usepackage{pgfplots}
\usepackage{caption}
\usepackage{subcaption}
\usetikzlibrary{decorations.pathmorphing}
\usetikzlibrary{decorations.markings}
\usetikzlibrary{arrows.meta,bending}
\usepackage{graphicx}			
\usetikzlibrary{patterns}
\usepackage{hyperref}
\usepackage{xcolor}			
\makeatletter
\def\@fnsymbol#1{\ensuremath{\ifcase#1\or \dagger\or \ddagger\or
		\mathsection\or \mathparagraph\or \|\or **\or \dagger\dagger
		\or \ddagger\ddagger \else\@ctrerr\fi}}
\makeatother
\definecolor{darkgreen}{rgb}{-0.2,0.75,-0.4}
\definecolor{royalblue}{rgb}{0.284,0.464,1}
\hypersetup{colorlinks,%
	citecolor=darkgreen,%
	filecolor=red,%
	linkcolor=cyan,%
	urlcolor=red}
\usepackage[left=2.0cm,right=2.0cm,top=2.0cm,bottom=2.0cm]{geometry}
\usepackage{fancyhdr}
\usepackage{appendix}
\usepackage{mathrsfs}
\date{}

\title{\textbf{Scattering Theory for the Charged Klein-Gordon Equation\\
in the Exterior De Sitter-Reissner-Nordstr\"{o}m Spacetime}}
\author{\textbf{Nicolas BESSET}\thanks{Universit\'{e} Grenoble Alpes, CS 40700, France. Institut Fourier - UMR 5582. E-mail: nicolas.besset@univ-grenoble-alpes.fr}}
\begin{document}
\maketitle
\tikzstyle{int}=[draw, fill=blue!20, minimum size=2em]
\tikzstyle{init} = [pin edge={to-,thin,black}]
\begin{center}
	\textbf{Abstract}
\end{center}
\begin{center}
	\small
	We show asymptotic completeness for the charged Klein-Gordon equation in the exterior De Sitter-Reissner-Nordstr\"{o}m spacetime when the product of the charge of the black hole with the charge of Klein-Gordon field is small enough. We then interpret scattering as asymptotic transports along principal null geodesics in a Kaluza-Klein extension of the original spacetime.
\end{center}
%
%
%
%
%
%
%
%
%
%
%
%
%
%
\maketitle						
%
%
%
%
%
%
%
%
%
%
%
%
%
%
%
%
%
%
%
%
%
%
%
%
\section{Introduction}
\label{Introduction}
Time dependent scattering theory describes large time scale interactions between a physical system (particles, waves) and its environment. The fundamental result that one may wish to establish is then the so-called \textit{asymptotic completeness} which compares dynamics one is interested in to a simpler and well-known one, a "free" dynamics. Many works have considered the case when the Hamiltonian associated to the system is self-adjoint with respect to a Hilbert space structure. The spectrum then provides rich information on the dynamics at large times: absolutely continuous spectrum is associated to scattering properties whereas pure point spectrum describes confinement. On the other hand, the asymptotic behaviour of scattering data can help to understand the classical equation which in turn is a first step in understanding the quantization of the field. For example, the wave operator $\Box_{g}=|g|^{-1/2}\partial_{\mu}g^{\mu\nu}|g|^{1/2}\partial_{\nu}$ can be viewed (up to normalizations) as a quantization of the null geodesic flow \textit{i.e.} the flow of the Hamilton vector field associated to the dual metric function $(x,\xi)\mapsto\xi_{\mu}g^{\mu\nu}(x)\xi_{\nu}$ where $(x,\xi)$ are local coordinates in the cotangent bundle; we talk about the first quantization of the field. We can convince ourselves that both are intrinsically linked together using Duistermaat-H\"{o}rmander theorem \cite{DH} which states that singularities of solutions of $\Box_{g}u=0$ propagate along null geodesics. Other interesting non-classical properties on the underlying spacetime can be obtained using scattering properties of the field: we can mention the Hawking effect, rigorously studied by Bachelot in its series of papers \cite{Ba97}, \cite{Ba99} and \cite{Ba00} for a spherically symmetric gravitational collapse, and by H\"{a}fner \cite{Ha09} in the charged and rotating case. Scattering theory is thus an open door to a possible quantum field theory in black hole spacetimes background. In this sight, we will be interested in this document in the charged Klein-Gordon equation which is a further step to a quantization for massive and charged fields.

In the case when the naturally conserved energy of solutions of the field equations is not positive along the flow of the dynamics, it is not possible to realize the Hamiltonian as self-adjoint operator on the underlying Hilbert space. The generator can have real and complex eigenfrequencies and the energy $\|\cdot\|_{\dot{\mathcal{E}}}$ can grow in time at a polynomial or exponential rate: this is \textit{superradiance}. This happens in the euclidean case, when the scalar field interacts with a strong electromagnetic potential. In this situation however, the Hamiltonian is self-adjoint on a Krein space, see the work of G\'{e}rard \cite{Ge12}. In a quite general setting, boundary value of the resolvent for self-adjoint operators on Krein spaces as well as propagation estimates for the Klein–Gordon equation can be obtained, see the works of Georgescu-G\'{e}rard-H\"{a}fner \cite{GGH13} and \cite{GGH15}. Another difficulty can occur however when the coupling term has two different limits at different ends of the considered manifold. If no positive energy is continuous with respect to the energy $\|\cdot\|_{\dot{\mathcal{E}}}$, then it is no longer possible to use a self-adjoint realization with respect to a Krein structure, see \textit{e.g.} \cite{GGH17}. This situation can already be encountered in the case of the one-dimensional charged Klein-Gordon equation, when coupling the Klein-Gordon field with a step-like electrostatic potential (see the work of Bachelot \cite{Ba}). It can also occur for Klein-Gordon equations on black hole type spacetimes. We can mention here the Klein-Gordon equation on the (De Sitter-)Kerr metric where the coupling between the field with the black hole occurs via the rotation. Asymptotic completeness has been obtained in absence of cosmological constant and mass by Dafermos-Rodnianski-Shlapentokh-Rothman \cite{DaRoSR18} using geometric techniques, and for positive cosmological constant and mass and bounded angular momenta of the field by Georgescu-G\'{e}rard-H\"{a}fner \cite{GGH17} using spectral methods. It turns out that the charged Klein-Gordon equation on the (De Sitter-)Reissner-Nordstr\"{o}m spacetime is also superradiant. This time the coupling comes from the equation itself and not from the geometry.

In the present work, we study the charged Klein-Gordon equation on the exterior De Sitter-Reissner-Nordstr\"{o}m spacetime. We will assume a weak coupling, that is, the product of the charge of the scalar field with the charge of the black hole will be supposed small enough. Although some technical issues existing in the De Sitter-Kerr spacetime vanish in our context, the loss of a natural geometric background prevents us to use a rich variety of geometric tools. As a result, we are left with the task to artificially encode the charge from the equation into the geometry if we wish to give a geometric interpretation of the scattering as asymptotic transport along some special null geodesics. The idea of encoding the electromagnetic interaction in an extra dimension was originally formulated by Kaluza in his paper \cite{Ka21} where he tried to unify gravitation and electromagnetism. He proposed that electromagnetism could be encoded in a fifth dimension and formulated his theory using the \textit{cylinder condition}, stating that no component of the new five-dimensional metric depends on the fifth dimension (actually, this condition makes equations easier to handle and avoid extra degrees of freedom). With the then outbreaks of quantum mechanics, Klein interpreted the cylinder condition as a microscopic curling of the electromagnetic field along the extra dimension (see \cite{Kl26}).

Using the symbol of the charged Klein-Gordon operator, we build in this paper an extended metric with an extra dimension for which the original operator can be seen as a wave operator. This is what we have called the \textit{neutralization procedure}. The extra dimension only needs to be a compact manifold, so that we can content ourselves with the unit circle. This choice of simplicity is actually very similar to the one Kaluza made one century ago. The extended manifold is then nothing but a Kaluza-Klein extension of the original one. This procedure works for a null charge product but fails without the mass term. An interesting fact is that the neutralization affects the original Maxwell tensor which becomes the sum of another Maxwell tensor with an effective charge (which takes into account the field's charge and mass) and a perfect fluid tensor acting in the plane generated by the time coordinate and the added one. The original black hole is also modified and turns into a \textit{black ring}. Black rings are similar to black holes as they prevent causal future-pointing curves to escape to infinity, the difference lying in the topology of the event horizon (which is $\mathbb{S}^{1}\times\mathbb{S}^{2}$ for black rings). They are solutions of Einstein equations in dimension 5 or greater, the first example of such object having been discovered by Emparan-Reall (see \cite{EmRe02}), see also a more recent review of black-hole solutions of higher-dimensional vacuum gravity and higher-dimensional supergravity theories \cite{EmRe08}. Using Hassan-Sen transformation, Elvang constructed in \cite{El03} from a neutral black ring a charged one minimizing an effective action with an electromagnetic tensor. Our neutralization procedure provides (with much less efforts) another example of a black ring solving the coupled Einstein-Maxwell equations, the rotation encoding the electromagnetic interaction.

Having constructed the extended spacetime, we can use the scattering theory in \cite{GGH17} as well as the absence of real resonances proved in \cite{Be18} to obtain asymptotic completeness at fixed non-zero angular momentum for the wave equation (which corresponds to the asymptotic completeness for the initial charge Klein-Gordon equation). Making use of the geometric tools in the Kaluza-Klein extension of the spacetime, we can reinterpret scattering as solving an abstract Goursat problem in energy spaces on the horizons by inverting trace operators. These energies are naturally related to the full dynamics associated to the charge Klein-Gordon operator as they use the limits of the electrostatic potential. However, this interpretation can not be "projected" onto the original spacetime due to the absence of appropriate geodesics; existence and unicity of the abstract Goursat problem is suspected to be false in $1+3$ dimensions.
%
%
%
%
\paragraph{Plan of the paper.}The paper is organized as follows:
\begin{itemize}
	\item \textbf{Section \ref{The extended spacetime}: The extended spacetime.} We construct and present some important properties of the extended spacetime in which the charged Klein-Gordon equation becomes a wave equation.\\[-8mm]
	\item \textbf{Section \ref{Global geometry of the extended spacetime}: Global geometry of the extended spacetime.} We build the Killing horizons with the principal null geodesics as well as the conformal boundary at infinity of the extension.\\[-8mm]
	\item \textbf{Section \ref{Analytic scattering theory}: Analytic scattering theory.} It is devoted to the statement and the proof of the analytic scattering results.\\[-8mm]
	\item \textbf{Section \ref{Geometric interpretation}: Geometric interpretation.} In this last section, we interpret the wave operators as transports along the principal null geodesics. Asymptotic completeness is used to define traces on the horizons and solve an abstract Goursat problem.
\end{itemize}
%
%
%
%
%
\paragraph{Statement of the main scattering results.}We give here simplified versions of the main scattering results of this paper.

In the exterior De Sitter-Reissner-Nordstr\"om spacetime $(\mathcal{M},g)$ (introduced in Subsection \ref{The extended spacetime}), the charged Klein-Gordon equation reads (up to a multiplication by a smooth function)
\begin{align}
\label{Intro 1}
	\left(\partial_{t}^{2}-2\mathrm{i}k\partial_{t}+h\right)u&=0.
\end{align}
Here $t\in\mathbb{R}$ is the time coordinate, $h$ is self-adjoint and $k\equiv k(s)$ is a multiplication operator by a smooth real function which linearly depends on $s$, the product of the charge of the black hole with the Klein-Gordon field charge, and $h$ is a second order partial differential operator such that $h_{0}:=h+k^{2}\geq 0$. The crucial point here is that $h$ contains the mass term $m>0$. Putting $m=0$ would yield to the wilder charged wave operator for which many of the results presented in this paper are not true.

The Hamiltonian form of equation \eqref{Intro 1} is
\begin{align}
\label{Intro 2}
	-\mathrm{i}\partial_{t}v&=\dot{H}v,&&\dot{H}v=\begin{pmatrix}
	0&\mathds{1}\\
	h&2k
	\end{pmatrix}.
\end{align}
Let $\dot{\mathcal{E}}$ be the completion of smooth compactly supported functions with respect to the homogeneous norm
\begin{align}
	\label{Intro 3}
	\|(u_0,u_1)\|_{\dot{\mathcal{E}}}^{2}&=\langle h_{0}u_{0},u_{0}\rangle_{\HH}+\|u_{1}-ku_{0}\|_{\HH}^{2}&&(u_0,u_1)\in\dot{\mathcal{E}}.
\end{align}
Here $\HH$ is a standard $L^{2}$ space. Then $\dot{H}$ generates a continuous one-parameter group on $\dot{\mathcal{E}}$, denoted by $(\mathrm{e}^{\mathrm{i}t\dot{H}})_{t\in\mathbb{R}}$. The operator $\mathrm{e}^{\mathrm{i}t\dot{H}}$ is called \textit{propagator}, \textit{evolution} or \textit{dynamics} (at the time $t$) of $\dot{H}$. The norm \eqref{Intro 3} is not the natural energy $\|\cdot\|$ associated to solutions of \eqref{Intro 2}, which is: 
\begin{align}
\label{Intro 4}
	\|(u_0,u_1)\|^{2}&=\langle hu_{0},u_{0}\rangle_{\HH}+\|u_{1}\|_{\HH}^{2}&&(u_0,u_1)\in\dot{\mathcal{E}}.
\end{align}
It turns out that \eqref{Intro 4} is not in general positive because of negative contributions from $k$ as the radial coordinate approaches horizons. Contrary to De Sitter-Kerr spacetime, there is a lack of geometric tools to explain this phenomenon; Section \ref{The extended spacetime} and more particularly Remark \ref{Dyadorings} provide them. The norm \eqref{Intro 3} is obviously positive but may grow with time. For this reason, we say that the charged Klein-Gordon equation is \textit{superradiant} in this context.

The neutralization procedure of Subsection \ref{The neutralization procedure} turns the charged Klein-Gordon operator \eqref{Intro 1} into the wave operator of an abstract Lorentzian manifold $(\M,\g)$ of dimension $1+4$. The extra dimension is labeled $z$, and we get the original charged operator by restriction on $\ker(\mathrm{i}\partial_{z}+\z)$ for $\z=1$. We will however consider general $\z\in\mathbb{Z}\setminus\{0\}$ in the sequel and denote with the superscript $^{\z}$ the restriction of operators and spaces to $\ker(\mathrm{i}\partial_{z}+\z)$. The first important result is the uniform boundedness of the evolution (see Theorem \ref{Uniform boundedness of the evolution}):
\begin{theorem}
	Let $\z\in\mathbb{Z}\setminus\{0\}$. There exists $s_{0}\equiv s_{0}(\z)>0$ such that for all $s\in\left]-s_{0},s_{0}\right[$, there exists a constant $C\equiv C(\z,s_{0})>0$ such that
	\begin{align*}
		\big\|\mathrm{e}^{\mathrm{i}t\dot{H}^{\z}}u\big\|_{\dot{\mathcal{E}}^{\z}}&\leq C\|u\|_{\dot{\mathcal{E}}^{\z}}\qquad\qquad\forall t\in\mathbb{R},\ \forall u\in\dot{\mathcal{E}}^{\z}.
	\end{align*}
\end{theorem}
Notice that the above result is not uniform in $\z$. In the extended spacetime, we construct \textit{incoming/outgoing principal null geodesics} $\gamma_{\textup{in/out}}$ in a similar way as for (De Sitter-)Kerr spacetime. They generate Killing horizons in $(\M,\g)$ and are $\g$-orthogonal to the extra dimension $z$ (see Section \ref{Global geometry of the extended spacetime}). To these geodesics we associate Hamiltonians $\dot{H}^{\z}_{\mathscr{H}\!/\!\mathscr{I}}$ acting on some energy spaces $\dot{\mathcal{E}}^{\z}_{\mathscr{H}\!/\!\mathscr{I}}$. They are related to transports towards $\mathscr{H}\!/\!\mathscr{I}$. Let $i_{-\!/\!+}$ be two smooth functions supported respectively near $\mathscr{H}\!/\!\mathscr{I}$. We can then show the existence and the completeness of wave operators (see Theorem \ref{Asymptotic completeness, geometric profiles}):
\begin{theorem}
	Let $\z\in\mathbb{Z}\setminus\{0\}$. There exists $s_{0}>0$ such that for all $s\in\left]-s_{0},s_{0}\right[$, the following holds:
	\begin{enumerate}
		\item There exists a dense subspace $\mathcal{D}^{\textup{fin},\z}_{\mathscr{H}\!/\!\mathscr{I}}\subset\dot{\mathcal{E}}^{\z}_{\mathscr{H}\!/\!\mathscr{I}}$ such that for all $u\in\mathcal{D}^{\textup{fin},\z}_{\mathscr{H}\!/\!\mathscr{I}}$, the limits
		\begin{align*}
			\boldsymbol{W}_{\mathscr{H}\!/\!\mathscr{I}}^{f}u&=\lim_{t\to+\infty}\mathrm{e}^{\mathrm{i}t\dot{H}^{\z}}i_{-\!/\!+}\mathrm{e}^{-\mathrm{i}t\dot{H}^{\z}_{\mathscr{H}\!/\!\mathscr{I}}}u,\\
			\boldsymbol{W}_{\mathscr{H}\!/\!\mathscr{I}}^{p}u&=\lim_{t\to-\infty}\mathrm{e}^{\mathrm{i}t\dot{H}^{\z}}i_{-\!/\!+}\mathrm{e}^{-\mathrm{i}t\dot{H}^{\z}_{\mathscr{H}\!/\!\mathscr{I}}}u
		\end{align*}
		exist in $\dot{\mathcal{E}}^{\z}$. The operators $\boldsymbol{W}_{\mathscr{H}\!/\!\mathscr{I}}^{f/p}$ extend to bounded operators $\boldsymbol{W}_{\mathscr{H}\!/\!\mathscr{I}}^{f/p}\in\mathcal{B}(\dot{\mathcal{E}}^{\z}_{\mathscr{H}\!/\!\mathscr{I}},\dot{\mathcal{E}}^{\z})$.
		\item The inverse future/past wave operators
		\begin{align*}
			\boldsymbol{\Omega}_{\mathscr{H}\!/\!\mathscr{I}}^{f}&=\textup{s}-\lim_{t\to+\infty}\mathrm{e}^{\mathrm{i}t\dot{H}^{\z}_{\mathscr{H}\!/\!\mathscr{I}}}i_{-\!/\!+}\mathrm{e}^{-\mathrm{i}t\dot{H}^{\z}},\\
			\boldsymbol{\Omega}_{\mathscr{H}\!/\!\mathscr{I}}^{p}&=\textup{s}-\lim_{t\to-\infty}\mathrm{e}^{\mathrm{i}t\dot{H}^{\z}_{\mathscr{H}\!/\!\mathscr{I}}}i_{-\!/\!+}\mathrm{e}^{-\mathrm{i}t\dot{H}^{\z}}
		\end{align*}
		exist in $\mathcal{B}(\dot{\mathcal{E}}^{\z},\dot{\mathcal{E}}^{\z}_{\mathscr{H}\!/\!\mathscr{I}})$.
	\end{enumerate}
\end{theorem}
Consider eventually the following abstract \textit{Goursat problem}. Let $\dot{\mathscr{E}}_{\pm}^{\z}$ be the energy spaces on $\mathscr{H}^{\pm}\times\mathscr{I}^{\pm}$ obtained from $\dot{\mathcal{E}}^{\z}_{\mathscr{H}\!/\!\mathscr{I}}$ following the flow of the principal null geodesics. Given $(\xi^{\pm},\zeta^{\pm})\in\dot{\mathscr{E}}_{\pm}^{\z}$, the problem consists in finding a solution $\phi$ to the wave equation in $(\M,\g)$ (and thus to \eqref{Intro 1} with mass $m\z$) such that
\begin{align}
\label{Boundary data intro}
	\phi_{\vert\mathscr{H}^{\pm}}&=\xi^{\pm},\qquad\qquad\phi_{\vert\mathscr{I}^{\pm}}=\zeta^{\pm}.
\end{align}
We then have the following result (see Theorem \ref{Abstract Goursat problem}):
\begin{theorem}
	Let $\z\in\mathbb{Z}\setminus\{0\}$. There exists $s_{0}>0$ such that for all $s\in\left]-s_{0},s_{0}\right[$ the following property: there exist homeomorphisms
	\begin{align*}
		\mathbb{T}^{\pm}&:\dot{\mathcal{E}}^{\z}\longrightarrow\dot{\mathscr{E}}^{\z}_{\pm}
	\end{align*}
	solving the Goursat problem \eqref{Boundary data intro} in the energy spaces, that is, for all $(\xi^{\pm},\zeta^{\pm})\in\dot{\mathscr{E}}^{\z}_{\pm}$, there exists an unique $\phi\in\mathcal{C}^{0}(\mathbb{R}_{t};\dot{\mathcal{E}}^{\z})$ solving the wave equation on $(\M,\g)$ with initial data $\phi(0)=(\phi_{0},\phi_{1})$ such that
	\begin{align*}
		(\xi^{\pm},\zeta^{\pm})&=\mathbb{T}^{\pm}(\phi_{0},\phi_{1}).
	\end{align*}
	The operators $\mathbb{T}^{\pm}$ are the extensions to the energy spaces of traces on the horizons.
\end{theorem}
We emphasize here that the above geometric interpretation of the scattering only holds in the Kaluza-Kein extension as it uses transport along geodesics which do not project onto the original spacetime (because of the absence of the charge $q$ and mass $m$ in the original metric).
%
%
%
%
\paragraph{Notations and conventions.}The set $\left\{z\in\mathbb{C}\mid\Im z>0\right\}$ will be denoted by $\mathbb{C}^{+}$. For any complex number $\lambda\in\mathbb{C}$, $D(\lambda,R)$ will be the disc centered at $\lambda\in\mathbb{C}$ of radius $R>0$. To emphasize some important dependences, the symbol $\equiv$ will be used: for example, $a\equiv a(b)$ means "$a$ depends on $b$". We will write $u\lesssim v$ to mean $u\leq Cv$ for some constant $C>0$ independant of $u$ and $v$.

The notation $\mathcal{C}^k_{\mathrm{c}}$ will be used to denote the space of compactly supported $\mathcal{C}^k$ functions. All the scalar products $\langle \cdot\,,\cdot\rangle$ will be antilinear in the first component and linear in the second component. The identity operator (acting on some vector space of functions) will be written $\mathds{1}$. For any function $f$, the support of $f$ will be denoted by $\mathrm{Supp\,}f$. For an operator $A$, we will denote by $\mathscr{D}\left(A\right)$ its domain.

When using the standard spherical coordinates $(\theta,\varphi)\in\left]0,\pi\right[\times\left]0,2\pi\right[$ on $\mathbb{S}^{2}$, we will always ignore the singularities $\{\theta=0\}$, $\{\theta=\pi\}$. We refer to \cite[Lemma 2.2.2]{ON} to properly fix it.
%
%
%
%
%
%
%
%
%
%
%
\section{The extended spacetime}
\label{The extended spacetime}
This Section introduces the basic notions and objects used throughout this paper. We recall in Subsection \ref{The charged Klein-Gordon equation in the exterior De Sitter-Reissner-Nordstrom spacetime} the De Sitter-Reissner-Nordstr\"{o}m metric and the charged Klein-Gordon equation we are interested in. Subsection \ref{The neutralization procedure} presents the neutralization procedure which extends the original spacetime using the charged Klein-Gordon operator. We next show in Subsection \ref{Extended Einstein-Maxwell equations} how the neutralization affects Einstein-Maxwell equations. We then deduce in Subsection \ref{Dominant energy condition} a dominant energy condition that fulfill the energy-momentum tensor associated to $(\M,\g)$ under some assumption on the cosmological constant $\Lambda$.
\subsection{The charged Klein-Gordon equation in the exterior De Sitter-Reissner-Nordstr\"{o}m spacetime}
\label{The charged Klein-Gordon equation in the exterior De Sitter-Reissner-Nordstrom spacetime}
Let
\begin{align*}
F\left(r\right)&:=1-\frac{2M}{r}+\frac{Q^{2}}{2r^{2}}-\frac{\Lambda r^{2}}{3}.
\end{align*}
with $M>0$ the mass of the black hole, $Q\in\mathbb{R}\setminus\{0\}$ its electric charge and $\Lambda>0$ the cosmological constant. $F$ is the \textit{horizon} function. We assume that
\begin{align}
\label{4 roots}
\Delta&:=9M^{2}-4Q^{2}>0,&\max\left\{0,\frac{6\big(M-\sqrt{\Delta}\,\big)}{\big(3M-\sqrt{\Delta}\,\big)^{3}}\right\}&<\Lambda<\frac{6\big(M+\sqrt{\Delta}\,\big)}{\big(3M+\sqrt{\Delta}\,\big)^{3}}
\end{align}
so that $F$ has four distinct zeros $-\infty<r_{n}<0<r_{c}<r_{-}<r_{+}<+\infty$ and is positive for all $r\in\left]r_{-},r_{+}\right[$ (see \cite[Proposition 3.2]{Hi18} with $Q^{2}$ replaced by $Q^{2}/2$ for us). We also assume that $9\Lambda M^2<1$ so that we can use the work of Bony-H\"{a}fner \cite{BoHa08}. In Boyer-Lindquist local coordinates, the \textit{exterior De Sitter-Reissner-Nordstr\"{o}m spacetime} is the Lorentzian manifold $(\mathcal{M},g)$ with
\begin{align*}
\mathcal{M}&=\mathbb{R}_t\times\left]r_-,r_+\right[_r\times\mathbb{S}^2_\omega,\qquad\qquad g=F\left(r\right)\mathrm{d}t^{2}-F\left(r\right)^{-1}\mathrm{d}r^{2}-r^{2}\mathrm{d}\omega^{2} 
\end{align*}
where $\mathrm{d}\omega^2$ is the standard metric on the unit sphere $\mathbb{S}^2$.

Let $A=(Q/r)\mathrm{d}t$. Then $(g,A)$ solves the Einstein-Maxwell field equations
\begin{align}
\label{Stress-energy tensor}
	\mathrm{Ric}_{\mu\nu}-\frac{1}{2}Rg_{\mu\nu}-\Lambda g_{\mu\nu}&=-T_{\mu\nu},&T_{\mu\nu}&=\mathfrak{F}_{\mu\sigma}\mathfrak{F}_{\nu}^{\ \sigma}-\frac{1}{4}g_{\mu\nu}\mathfrak{F}^{\sigma\rho}\mathfrak{F}_{\sigma\rho}
\end{align}
where $\mathrm{Ric}$ is the Ricci tensor, $R$ the scalar curvature and $\mathfrak{F}=\mathrm{d}A$ the electromagnetic tensor. The charged wave operator on $(\mathcal{M},g)$ is
\begin{align*}
	(\nabla_\mu-\mathrm{i}A_{\mu})(\nabla^\mu-\mathrm{i}A^{\mu})&=\frac{1}{F}\big(\partial_{t}-\mathrm{i}sV(r)\big)^{2}-\frac{1}{r^{2}}\partial_{r}r^{2}F\partial_{r}-\frac{1}{r^{2}}\Delta_{\mathbb{S}{^2}},
\end{align*}
where $s:=qQ$ is the charge product and $V(r)=r^{-1}$. In the sequel, we will write
\begin{align*}
	V_{\alpha}&:=r_{\alpha}^{-1}=\lim_{r\to r_{\alpha}}V(r)\qquad\qquad\forall\alpha\in\{c,-,+\}.
\end{align*}
The charged Klein-Gordon operator is then
\begin{align*}
	(\nabla_\mu-\mathrm{i}A_{\mu})(\nabla^\mu-\mathrm{i}A^{\mu})+m^{2}&=\frac{1}{F(r)}\big(\partial_{t}-\mathrm{i}sV(r)\big)^{2}-\frac{1}{r^{2}}\partial_{r}r^{2}F(r)\partial_{r}-\frac{1}{r^{2}}\Delta_{\mathbb{S}{^2}}+m^{2}.
\end{align*}
%
%
%
%
\subsection{The neutralization procedure}
\label{The neutralization procedure}
We introduce a fifth dimension labeled $z\in\mathbb{S}^{1}$ in order to reinterpret the charged Klein-Gordon operator in the De Sitter-Reissner-Nordstr\"{o}m spacetime as a wave operator in a $1+4$ spacetime. Define
\begin{align*}
	L&:=\frac{1}{F(r)}(\partial_{t}-sV(r)\partial_{z})^{2}-\frac{1}{r^{2}}\partial_{r}r^{2}F(r)\partial_{r}-\frac{1}{r^{2}}\Delta_{\mathbb{S}{^2}}-m^{2}\partial_{z}^{2}.
\end{align*}
Diagonalizing $-\mathrm{i}\partial_{z}$ on the unit circle, we can recover $P$ by restriction to the harmonic $-\mathrm{i}\partial_{z}=1$. We construct a new metric $\g$ such that $\Box_{\g}=L$: the \textit{extended metric} in the \textit{extended Boyer-Lindquist coordinates} $(t,z,r,\omega)$ with signature $(+,-,-,-,-)$ is defined as
\begin{align*}
	\g&:=\left(F(r)-\frac{s^{2}V(r)^{2}}{m^{2}}\right)\mathrm{d}t^{2}-\frac{sV(r)}{m^{2}}\left(\mathrm{d}t\mathrm{d}z+\mathrm{d}z\mathrm{d}t\right)-\frac{1}{m^{2}}\mathrm{d}z^{2}-\frac{1}{F(r)}\mathrm{d}r^{2}-r^{2}\mathrm{d}\omega^{2}.
\end{align*}
It is non-degenerate since the determinant is equal to $-r^{4}\sin^{2}\theta/m^{2}<0$ (notice here the crucial hypothesis $m\neq0$). The inverse extended metric is given by
\begin{align*}
	\g^{-1}&=\frac{1}{F(r)}\partial_{t}\otimes\partial_{t}-\frac{sV(r)}{F(r)}\left(\partial_{t}\otimes\partial_{z}+\partial_{z}\otimes\partial_{t}\right)+\left(\frac{s^{2}V(r)^{2}}{F(r)}-m^{2}\right)\partial_{z}\otimes\partial_{z}\\
	&-F(r)\partial_{r}\otimes\partial_{r}-\frac{1}{r^{2}}\partial_{\theta}\otimes\partial_{\theta}-\frac{1}{r^{2}\sin^{2}\theta}\partial_{\varphi}\otimes\partial_{\varphi}.
\end{align*}
Let us define the four blocks
\begin{align*}
	\M_{1}&:=\mathbb{R}_{t}\times\mathbb{S}^{1}_{z}\times\left]0,r_{c}\right[_{r}\times\mathbb{S}^{2}_{\omega},&\M_{2}&:=\mathbb{R}_{t}\times\mathbb{S}^{1}_{z}\times\left]r_{c},r_{-}\right[_{r}\times\mathbb{S}^{2}_{\omega},\\
	\M_{3}&:=\mathbb{R}_{t}\times\mathbb{S}^{1}_{z}\times\left]r_{-},r_{+}\right[_{r}\times\mathbb{S}^{2}_{\omega},&\M_{4}&:=\mathbb{R}_{t}\times\mathbb{S}^{1}_{z}\times\left]r_{+},+\infty\right[_{r}\times\mathbb{S}^{2}_{\omega}.
\end{align*}
The \textit{extended spacetime} is then the $(1+4)$-dimensional manifold $\big(\M_{\textup{ext}},\g\big)$ with
\begin{align*}
	\M_{\textup{ext}}&:=\bigcup_{j=1}^{4}\M_{j}.
\end{align*}
In the sequel, we will consider more carefully the \textit{outer space}
\begin{align*}
	\M&:=\mathbb{R}_{t}\times\mathbb{S}^{1}_{z}\times\left]r_{-},r_{+}\right[_{r}\times\mathbb{S}^{2}_{\omega}
\end{align*}
which we will simply call the extended spacetime when no confusion can occur. In $\M$, we will use the timelike vector field $\nabla t=F(r)^{-1}(\partial_{t}-sV(r)\partial_{z})$ to define an orientation on $\M$: any causal vector field $X\in T\M$ will be said \textit{future-pointing} if and only if
\begin{align*}
	\g(\nabla t,X)&>0.
\end{align*}
We can easily check that the outer space is  a globally hyperbolic spacetime. Besides, we can check that $X_{\alpha}:=\partial_{t}-sV_{\alpha}\partial_{z}$ are timelike Killing vector fields for $\alpha\in\{c,-,+\}$ (they will be useful in Subsection \ref{Surface gravities and Killing horizons}).
\begin{remark}[Dyadorings]
\label{Dyadorings}
	For $s$ small enough, the shifted horizon function
	\begin{align*}
		\mathbf{F}(r)&:=F(r)-\frac{s^2V^2}{m^2}=1-\frac{2M}{r}+\frac{Q^2}{2r^2}\left(1-\frac{2q^2}{m^2}\right)-\frac{\Lambda r^2}{3}
	\end{align*}
	has four roots with two\footnote{It is a simple consequence of the intermediate value theorem since $\mathbf{F}<F$, $F>0$ inside $\left]r_-,r_+\right[$ and $F$ cancels at $r=r_\pm$.} inside $\left]r_-,r_+\right[$; call them $r_1, r_2$ with $r_-<r_1<r_2<r_+$. We can check that $\partial_{t}$ becomes spacelike when $r\in\left]r_{-},r_{1}\right[\cup\left]r_{2},r_{+}\right[$. We then define the \textup{dyadorings}
	\begin{align*}
		\mathcal{D}_{-}&:=\mathbb{R}_{t}\times\mathbb{S}^{1}_{z}\times\left]r_{-},r_{1}\right[_{r}\times\mathbb{S}^{2}_{\omega},\qquad\qquad\mathcal{D}_{+}:=\mathbb{R}_{t}\times\mathbb{S}^{1}_{z}\times\left]r_{2},r_{+}\right[_{r}\times\mathbb{S}^{2}_{\omega}.
	\end{align*}
	We may observe that if $s$ is too large, namely if
	\begin{align*}
	|s|&\geq mrF(r)^{1/2}\qquad\qquad\forall r\in\left]r_{-},r_{+}\right[,
	\end{align*}
	then the dyadorings cover the entire outer space $\M$.
\end{remark}

The wave equation on $(\M,\g)$ reads
\begin{align*}
	\Box_{\g}u&=\frac{1}{F(r)}\left(\big(\partial_{t}-sV(r)\partial_{z}\big)^{2}-\frac{F(r)}{r^{2}}\partial_{r}r^{2}F(r)\partial_{r}-\frac{F(r)}{r^{2}}\Delta_{\mathbb{S}{^2}}-m^{2}F(r)\partial_{z}^{2}\right)u=0
\end{align*}
with $u\in L^{2}(\mathbb{R}_{t}\times\mathbb{S}^{1}_{z}\times\left]r_{-},r_{+}\right[_{r}\times\mathbb{S}^{2}_{\omega},r^{2}F(r)^{-1}\mathrm{d}t\mathrm{d}z\mathrm{d}r\mathrm{d}\omega)$. It will be convenient for Section \ref{Analytic scattering theory} to rewrite this equation as
\begin{align}
\label{Extended wave equation}
	\left(\partial_{t}^{2}-2sV(r)\partial_{z}\partial_{t}+\hat{P}\right)u&=0
\end{align}
where
\begin{align}
\label{P_z}
	\hat{P}&=-\frac{F(r)}{r^{2}}\partial_{r}r^{2}F(r)\partial_{r}-\frac{F(r)}{r^{2}}\Delta_{\mathbb{S}{^2}}-\big(m^{2}F(r)-s^{2}V(r)^{2}\big)\partial_{z}^{2}
\end{align}
acts on $L^{2}(\mathbb{S}^{1}_{z}\times\left]r_{-},r_{+}\right[_{r}\times\mathbb{S}^{2}_{\omega},r^{2}F(r)^{-1}\mathrm{d}z\mathrm{d}r\mathrm{d}\omega)$. Restricting $u$ to $\ker(\mathrm{i}\partial_{z}+\z)$, we get back the original charge Klein-Gordon operator with the modified mass $m\z$ and charge $s\z$.
%
%
%
%
\subsection{Extended Einstein-Maxwell equations}
\label{Extended Einstein-Maxwell equations}
The neutralization procedure has modified the Einstein-Maxwell equations \eqref{Stress-energy tensor} as we will see in this Subsection. 

Denote with a $\,\widetilde{\,}\,$ the quantities corresponding to the extended spacetime $(\M_{\textup{ext}},\g)$ and set\footnote{We can use \textit{any} $\mathcal{C}^{2}$ potential $W\equiv W(r)$ instead of $sV(r)$ in $\g$.} $W(r)=sV(r)$. The Christoffel symbols are given in matrix notations by:
\begin{align*}
\widetilde{\Gamma}^{0}_{\mu\nu}&=\begin{pmatrix}
0&0&\frac{m^{2}F'-WW'}{2m^{2}F}&0&0\\[1mm]
0&0&-\frac{W'}{2m^{2}F}&0&0\\[1mm]
\frac{m^{2}F'-WW'}{2m^{2}F}&-\frac{W'}{2m^{2}F}&0&0&0\\[1mm]
0&0&0&0&0\\[1mm]
0&0&0&0&0
\end{pmatrix},
\end{align*}
\begin{align*}
\widetilde{\Gamma}^{1}_{\mu\nu}=\begin{pmatrix}
0&0&\frac{m^{2}(FW'-F'W)+W^{2}W'}{2m^{2}F}&0&0\\[1mm]
0&0&\frac{WW'}{2m^{2}F}&0&0\\[1mm]
\frac{m^{2}(FW'-F'W)+W^{2}W'}{2m^{2}F}&\frac{WW'}{2m^{2}F}&0&0&0\\[1mm]
0&0&0&0&0\\[1mm]
0&0&0&0&0
\end{pmatrix},
\end{align*}
\begin{align*}
\widetilde{\Gamma}^{2}_{\mu\nu}&=\begin{pmatrix}
\frac{F(m^{2}F'-2WW')}{2m^{2}}&-\frac{FW'}{2m^{2}}&0&0&0\\[1mm]
-\frac{FW'}{2m^{2}}&0&0&0&0\\[1mm]
0&0&-\frac{F'}{2F}&0&0\\[1mm]
0&0&0&-rF&0\\[1mm]
0&0&0&0&-rF\sin^{2}\theta
\end{pmatrix},
\end{align*}
\begin{align*}
\widetilde{\Gamma}^{3}_{\mu\nu}&=\begin{pmatrix}
0&0&0&0&0\\[1mm]
0&0&0&0&0\\[1mm]
0&0&0&\frac{1}{r}&0\\[1mm]
0&0&\frac{1}{r}&0&0\\[1mm]
0&0&0&0&-\sin\theta\cos\theta
\end{pmatrix}, &\widetilde{\Gamma}^{4}_{\mu\nu}=\begin{pmatrix}
0&0&0&0&0\\[1mm]
0&0&0&0&0\\[1mm]
0&0&0&0&\frac{1}{r}\\[1mm]
0&0&0&0&\cot\theta\\[1mm]
0&0&\frac{1}{r}&\cot\theta&0
\end{pmatrix}.
\end{align*}
The non-zero components of the extended Ricci tensor $\widetilde{\mathrm{Ric}}$ are given by
\begin{align*}
	\widetilde{\mathrm{Ric}}_{00}&=\frac{F\left(2F'+rF''\right)}{2r}-\left(\frac{2FWW'}{m^2r}+\frac{FWW''}{m^2}+\frac{FW'^2}{2m^2}+\frac{W^2W'^2}{2m^4}\right),\\
	\widetilde{\mathrm{Ric}}_{01}&=\widetilde{\mathrm{Ric}}_{10}=-\left(\frac{FW'}{m^2r}+\frac{FW''}{2m^2}+\frac{WW'^2}{2m^4}\right),\\
	\widetilde{\mathrm{Ric}}_{11}&=-\frac{W'^{2}}{2m^{4}},\\
	\widetilde{\mathrm{Ric}}_{22}&=-\frac{2F'+rF''}{2rF}+\frac{W'^2}{2m^2F},\\
	\widetilde{\mathrm{Ric}}_{33}&=1-F-rF',\\
	\widetilde{\mathrm{Ric}}_{44}&=\left(1-F-rF'\right)\sin^{2}\theta
\end{align*}
and the scalar curvature is
\begin{align*}
	\widetilde{R}&:=\sum_{i=0}^{4}\sum_{j=0}^{4}\g^{\mu\nu}\widetilde{\mathrm{Ric}}_{\mu\nu}=R_{\mathrm{DSRN}}-\frac{(W')^2}{2m^2}=-4\Lambda-\frac{q^2Q^2}{2m^2r^4}
\end{align*}
where $R_{\mathrm{DSRN}}=-4\Lambda$ is the scalar curvature associated to the De Sitter-Reissner-Nordstr\"{o}m metric $g$.
The singularity at $r=0$ is therefore still present in the extended spacetime; notice however that the scalar curvature has become singular at this point after neutralization.

Let now $\widetilde{A}=\frac{Q}{r}\sqrt{1-\frac{q^2}{2m^2}}\mathrm{d}t$. Tedious but direct computations show that $(\g,\widetilde{A})$ solves the Einstein field equations
\begin{align}
\label{Extended Einstein}
	\widetilde{\mathrm{Ric}}-\frac{1}{2}\widetilde{R}\g-\Lambda\g&=-\widetilde{T}.
\end{align}
The extended stress-energy tensor $\widetilde{T}$ is given by $\widetilde{T}=\widetilde{T}_{\mathrm{Maxwell}}+\widetilde{T}_{\mathrm{fluid}}$ with in matrix notations
\begin{align*}
\widetilde{T}_{\mathrm{Maxwell}}&=\frac{Q^2}{2r^4}\left(1-\frac{q^2}{2m^2}\right)\begin{pmatrix}
-F(r)-\frac{s^{2}V(r)^{2}}{m^{2}}&-\frac{sV(r)}{m^{2}}&0&0&0\\
-\frac{sV(r)}{m^{2}}&-\frac{1}{m^{2}}&0&0&0\\
0&0&\frac{1}{F(r)}&0&0\\
0&0&0&-r^2&0\\
0&0&0&0&-r^2\sin^2\theta
\end{pmatrix},\\[2mm]
\widetilde{T}_{\mathrm{fluid}}&=\left(\Lambda+\frac{Q^{2}}{2r^{4}}\left(1+\frac{q^{2}}{m^{2}}\right)\right)\begin{pmatrix}
\frac{s^{2}V(r)^{2}}{m^{2}}&\frac{sV(r)}{m^{2}}&0&0&0\\
\frac{sV(r)}{m^{2}}&\frac{1}{m^{2}}&0&0&0\\
0&0&0&0&0\\
0&0&0&0&0\\
0&0&0&0&0
\end{pmatrix}.
\end{align*}
$\widetilde{T}_{\mathrm{Maxwell}}$ is nothing but the Maxwell electromagnetic tensor associated to $\g$ and $\widetilde{A}$; $\widetilde{T}_{\mathrm{fluid}}$ describes as for itself a perfect fluid:
\begin{align*}
\widetilde{T}_{\mathrm{fluid}}&=\rho(r)u\otimes u
\end{align*}
with $\rho(r)=\Lambda+\frac{Q^{2}}{2r^{4}}\left(1+\frac{q^{2}}{m^{2}}\right)$ the energy density in the fluid and $u=\frac{1}{m}\big(sV(r)\mathrm{d}t+\mathrm{d}z\big)$ the dual vector field of the fluid's velocity which is given by
\begin{align}
\label{Perfect fluid speed}
v^{\mu}&=\g^{\mu\nu}u_{\nu}=-m\partial_{z}= \frac{1}{m}\big(sV(r)\nabla t+\nabla z\big).
\end{align}
Observe that the energy density of the fluid as measured by an observer at rest is zero:
\begin{align*}
(\widetilde{T}_{\mathrm{fluid}})_{\mu\nu}\nabla^{\mu}t\nabla^{\nu}t&=0.
\end{align*}
Observe also that taking the divergence on both sides of \eqref{Extended Einstein} yields
\begin{align*}
	\mathrm{div}(\rho(r)u\otimes u)&=\frac{\rho(r)F(r)W(r)W'(r)}{2m^{2}}\partial_{r}=:-\nabla P
\end{align*}
with
\begin{align*}
	P(r)&=-\frac{q^{2}Q^{2}}{4m^{2}r^{2}}\left(\Lambda+\frac{Q^{2}(1+\frac{q^{2}}{m^{2}})}{6r^{4}}\right)\leq 0.
\end{align*}
Moreover, we can see that
\begin{align*}
	\mathrm{div}(\rho u)&=0.
\end{align*}
Thus $u$ obeys compressible Euler law for a static fluid with \textit{mass density} $\rho$, \textit{pressure} $P$ and no internal source term:
\begin{align*}
	\begin{cases}
		\partial_{t}(\rho u)+\mathrm{div}(\rho u\otimes u)~=~-\nabla P\\
		\partial_{t}\rho+\mathrm{div}(\rho u)~=~0
	\end{cases}.
\end{align*}
The second equation above is the conservation of the mass.
%
%
%
%
\subsection{Dominant energy condition}
\label{Dominant energy condition}
In all this Subsection, we will assume that $|q|<2m$. Let us rewrite
\begin{align*}
	-\widetilde{T}&=\mathfrak{Q}(r)\Big(F(r)\mathrm{d}t^{2}-F(r)^{-1}\mathrm{d}r^{2}+r^{2}\mathrm{d}\omega^{2}\Big)-\mathfrak{D}(r)\big(sV(r)\mathrm{d}t+\mathrm{d}z\big)^{2}
\end{align*}
where
\begin{align*}
	\mathfrak{Q}(r)&=\frac{Q^{2}}{2r^{4}}\left(1-\frac{q^{2}}{2m^{2}}\right),&\mathfrak{D}(r)&=\left(\frac{\Lambda}{m^{2}}+\frac{3q^{2}Q^{2}}{4m^{4}r^{4}}\right).
\end{align*}
Consider then the following condition:
\begin{align}
\label{Weak condition Lambda}
	m^{2}\mathfrak{D}(r)&\leq\mathfrak{Q}(r)
\end{align}
which is equivalent to $\Lambda\leq\frac{Q^{2}}{2r_{+}^{4}}\left(1-\frac{2q^{2}}{m^{2}}\right)$.
\begin{lemma}
\label{Weak energy condition}
	The condition \eqref{Weak condition Lambda} implies that for all timelike vector field $X$, we have
	\begin{align}
	\label{WEC}
		-\widetilde{T}_{\mu\nu}X^{\mu}X^{\nu}&\geq 0.
	\end{align}
	If $Q\neq 0$, then both the above conditions are equivalent.
\end{lemma}
\begin{proof}
	Let $X=X^{t}\partial_{t}+X^{z}\partial_{z}+X^{r}\partial_{r}+X^{\theta}\partial_{\theta}+X^{\varphi}\partial_{\varphi}$ be a timelike vector field:
	\begin{align*}
		\left(F(r)-\frac{s^{2}V(r)^{2}}{m^{2}}\right)(X^{t})^{2}-2\frac{sV(r)}{m^{2}}X^{t}X^{z}-\frac{1}{m^{2}}(X^{z})^{2}-F(r)^{-1}(X^{r})^{2}-r^{2}(X^{\theta})^{2}-r^{2}\sin^{2}\theta(X^{\varphi})^{2}&>0.
	\end{align*}
	\indent Assume first \eqref{Weak condition Lambda}. Then
	\begin{align*}
		-\widetilde{T}_{\mu\nu}X^{\mu}X^{\nu}&=\mathfrak{Q}(r)\left(\g(X,X)+\left(\frac{1}{m^{2}}-\frac{\mathfrak{D}(r)}{\mathfrak{Q}(r)}\right)\left(s^{2}V(r)^{2}(X^{t})^{2}+2sV(r)X^{t}X^{z}+(X^{z})^{2}\right)\right.\\
		&\qquad\quad\,\left.\textcolor{white}{\frac{A}{Q}}+2r^{2}(X^{\theta})^{2}+2r^{2}\sin^{2}\theta(X^{\varphi})^{2}\right)\\
		&\geq\mathfrak{Q}(r)\left(\frac{1}{m^{2}}-\frac{\mathfrak{D}(r)}{\mathfrak{Q}(r)}\right)\left(s^{2}V(r)^{2}(X^{t})^{2}+2sV(r)X^{t}X^{z}+(X^{z})^{2}\right).
	\end{align*}
	Furthermore,
	\begin{align*}
		\left(\frac{1}{m^{2}}-\frac{\mathfrak{D}(r)}{\mathfrak{Q}(r)}\right)\begin{pmatrix}
		s^{2}V(r)^{2}&sV(r)\\
		sV(r)&1
		\end{pmatrix}&\geq 0
	\end{align*}
	as quadratic form since the spectrum of the matrix on the left-hand side is $\big\{0,1+s^{2}V(r)^{2}\big\}$ (the eigenvalue 0 is associated to $\nabla t$). It follows that $-\widetilde{T}_{\mu\nu}X^{\mu}X^{\nu}\geq 0$.
	
	Assume now that \eqref{Weak condition Lambda} is not verified and $Q\neq 0$ (so that $\mathfrak{Q}(r)\neq 0$). Put
	\begin{align*}
		\alpha(r)&:=\frac{1}{m^{2}}-\frac{\mathfrak{D}(r)}{\mathfrak{Q}(r)}<0.
	\end{align*}
	Let $\delta,\varepsilon\in\mathcal{C}^{\infty}(\left]r_{-},r_{+}\right[,\left]0,+\infty\right[)$ with $\varepsilon$ supported far away from the dyadorings and set $X:=\delta(r)\big(\partial_{t}-sV(r)(1+\varepsilon(r))\partial_{z}\big)$. Taking $\varepsilon$ sufficiently small, we get
	\begin{align*}
		\g(X,X)&=\delta(r)\left(F(r)-\frac{s^{2}V(r)^{2}\varepsilon(r)^{2}}{m^{2}}\right)>0.
	\end{align*}
	Taking now $\delta$ large enough, we find
	\begin{align*}
		-\widetilde{T}_{\mu\nu}X^{\mu}X^{\nu}&=\mathfrak{Q}(r)\Big(\g(X,X)+\alpha(r)\varepsilon(r)^{2}\delta(r)^{2}s^{2}V(r)^{2}\Big)\\
		&=\mathfrak{Q}(r)\delta(r)\left(F(r)-\frac{s^{2}V(r)^{2}\varepsilon(r)^{2}}{m^{2}}+\alpha(r)\varepsilon(r)^{2}\delta(r)s^{2}V(r)^{2}\right)\\
		&<0
	\end{align*}
	for $r\in\mathrm{Supp\,}\varepsilon$. This completes the proof.
\end{proof}
\begin{proposition}[Dominant energy condition]
	\label{Dominant energy condition proposition}
	Assume $Q\neq 0$. Then the condition \eqref{Weak condition Lambda} is verified if and only if
	\begin{align}
	\label{DEC}
		\begin{cases}
			-\widetilde{T}_{\mu\nu}X^{\mu}X^{\nu}\geq 0\qquad\text{ for all timelike vector field $X$},\\
			\text{If $X$ is a future-pointing causal vector field, then so is for $-\widetilde{T}^{\mu}_{\ \,\nu}X^{\nu}\geq 0$}
		\end{cases}.
	\end{align}
\end{proposition}
\begin{proof}
	Lemma \ref{Weak energy condition} shows that \eqref{Weak condition Lambda} is equivalent to the first condition in \eqref{DEC} provided that $Q\neq 0$. Next, a direct computation shows that
	\begin{align*}
		-\widetilde{T}^{\mu}_{\ \,\nu}=-\g^{\mu\sigma}\widetilde{T}_{\sigma\nu}&=\mathfrak{Q}(r)\partial_{t}\otimes\mathrm{d}t+\big(m^{2}\mathfrak{D}(r)-\mathfrak{Q}(r)\big)sV(r)\partial_{z}\otimes\mathrm{d}t+m^{2}\mathfrak{D}(r)\partial_{z}\otimes\mathrm{d}z\nonumber\\
		&+\mathfrak{Q}(r)\partial_{r}\otimes\mathrm{d}r-\mathfrak{Q}(r)\partial_{\theta}\otimes\mathrm{d}\theta-\mathfrak{Q}(r)\partial_{\varphi}\otimes\mathrm{d}\varphi.
	\end{align*}
	Let $X=X^{t}\partial_{t}+X^{z}\partial_{z}+X^{r}\partial_{r}+X^{\theta}\partial_{\theta}+X^{\varphi}\partial_{\varphi}$ be a future-pointing causal vector field:
	\begin{align*}
		\g(\nabla t,X)&=X^{t}>0,\qquad\g(X,X)\geq 0.
	\end{align*}
	We compute:
	\begin{align*}
		-\widetilde{T}^{\mu}_{\ \,\nu}X^{\nu}&=\mathfrak{Q}(r)X^{t}\partial_{t}+\big((m^{2}\mathfrak{D}(r)-\mathfrak{Q}(r))sV(r)X^{t}+m^{2}\mathfrak{D}(r)X^{z}\big)\partial_{z}\\
		&+\mathfrak{Q}(r)X^{r}\partial_{r}-\mathfrak{Q}(r)X^{\theta}\partial_{\theta}-\mathfrak{Q}(r)X^{\varphi}\partial_{\varphi}.
	\end{align*}
	Then the vector field $-\widetilde{T}^{\mu}_{\ \,\nu}X^{\nu}$ is future-pointing if and only if
	\begin{align*}
		\g(\nabla_{t},-\widetilde{T}^{\mu}_{\ \,\nu}X^{\nu})&=-\widetilde{T}^{t}_{\nu}X^{\nu}=\mathfrak{Q}(r)X^{t}>0
	\end{align*}
	which is always true when $Q\neq 0$. Furthermore:
	\begin{align*}
		\mathfrak{Q}(r)^{-2}\g(\widetilde{T}^{\mu}_{\ \,\nu}X^{\nu},\widetilde{T}^{\mu}_{\ \,\nu}X^{\nu})&=\left(F(r)-\frac{s^{2}V(r)^{2}}{m^{2}}\left(\frac{m^{2}\mathfrak{D}(r)}{\mathfrak{Q}(r)}\right)^{2}\right)(X^{t})^{2}-2\frac{sV(r)}{m^{2}}\left(\frac{m^{2}\mathfrak{D}(r)}{\mathfrak{Q}(r)}\right)^{2}X^{t}X^{z}\\
		&-\frac{1}{m^{2}}\left(\frac{m^{2}\mathfrak{D}(r)}{\mathfrak{Q}(r)}\right)^{2}(X^{z})^{2}-F(r)^{-1}(X^{r})^{2}\\
		&=\tilde{g}(X,X)\\
		&+\left(1-\left(\frac{m^{2}\mathfrak{D}(r)}{\mathfrak{Q}(r)}\right)^{2}\right)\left(\frac{s^{2}V(r)^{2}}{m 2}(X^{t})^{2}+2\frac{sV(r)}{m^{2}}X^{t}X^{z}+\frac{1}{m^{2}}(X^{z})^{2}\right).
	\end{align*}
	Using $\g(X,X)\geq 0$, we get:
	\begin{align*}
		\mathfrak{Q}(r)^{-2}\g(\widetilde{T}^{\mu}_{\ \,\nu}X^{\nu},\widetilde{T}^{\mu}_{\ \,\nu}X^{\nu})&\geq\left(1-\left(\frac{m^{2}\mathfrak{D}(r)}{\mathfrak{Q}(r)}\right)^{2}\right)\left(\frac{s^{2}V(r)^{2}}{m 2}(X^{t})^{2}+2\frac{sV(r)}{m^{2}}X^{t}X^{z}+\frac{1}{m^{2}}(X^{z})^{2}\right)
	\end{align*}
	and this quantity is nonnegative if and only if $m^{2}\mathfrak{D}(r)\leq\mathfrak{Q}(r)$ as in the proof of Lemma \ref{Weak energy condition} (the same vector field as therein shows the necessity of the condition). This completes the proof.
\end{proof}
\begin{remark}
	The condition \eqref{Weak condition Lambda} is called the \textup{dominant energy condition}. It will allow us to define constant surface gravities in Subsection \ref{Surface gravities and Killing horizons}. In the De Sitter-Reissner-Nordström case, Proposition \ref{Dominant energy condition proposition} is always true. Indeed, denoting by $T$ the corresponding stress-energy tensor as well as $\mathfrak{Q}(r)=\frac{Q^{2}}{2r^{4}}$, we have
	\begin{align*}
		T_{\mu\nu}X^{\mu}X^{\nu}&=\mathfrak{Q}(r)\left(F(r)(X^{t})^{2}-F(r)^{-1}(X^{r})^{2}+r^{2}(X^{\theta})^{2}+r^{2}\sin^{2}\theta(X^{\varphi})^{2}\right),\\
		\g(\partial_{t},T^{\mu}_{\ \,\nu}X^{\nu})&=\mathfrak{Q}(r)X^{t},\\
		\g(T^{\mu}_{\ \,\nu}X^{\nu},T^{\mu}_{\ \,\nu}X^{\nu})&=\mathfrak{Q}(r)X^{t}g(X,X)
	\end{align*}
	for any vector field $X$.
\end{remark}
\begin{remark}
	We may notice that condition \eqref{4 roots} implies a lower bound to $\Lambda$ if $|Q|\geq M$. In this situation and for a small charge $q$, the dominant energy condition \eqref{Weak condition Lambda} is not satisfied.
\end{remark}
%
%
%
%
%
%
\section{Global geometry of the extended spacetime}
\label{Global geometry of the extended spacetime}
This Section is devoted to the study of the geometry of the extended spacetime. The goal is two-fold: first, the study of the geometric objects is interesting in its own right as it describes in details a new solution to Einstein-Maxwell equations in 5 dimensions. On the other hand, the construction of some geometric objects (such as principal null geodesics and horizons) are prerequisites to formulate the scattering theory developed in Section \ref{Analytic scattering theory}. We will henceforth assume the dominant energy condition \eqref{Weak condition Lambda}.

Let us outline here the plan of this Section. Subsection \ref{Principal null geodesics} introduces first special null geodesics and local coordinates used to build horizons. In Subsection \ref{Surface gravities and Killing horizons}, we define surface gravities. In Subsection \ref{Crossing rings}, we add crossing rings to complete the construction of the horizons. Finally, we define the conformal infinity in Subsection \ref{Black rings} and show that the extended spacetime contains black rings.
%
%
%
%
\subsection{Principal null geodesics}
\label{Principal null geodesics}
We introduce in this Subsection a family of null geodesics which send data from the blocks $\M_j$ to the horizons located at the roots of $F$. They will be used for the geometric interpretation of the scattering in Section \ref{Geometric interpretation}. We follow the standard procedure (see for example the construction in \cite[Section 2.5]{ON} in the Kerr spacetime).

Let $\gamma:\mathbb{R}\ni\mu\mapsto(t(\mu),z(\mu),r(\mu),\omega(\mu))\in\M$ be a null geodesic. Denoting by $\dot{\ }$ the derivative with respect to $\mu$, we have
\begin{align}
\label{gamma_in/out beginning}
	\left(F(r)-\frac{s^{2}V(r)^{2}}{m^{2}}\right)\dot{t}^{2}-\frac{2sV(r)}{m^{2}}\dot{t}\dot{z}-\frac{1}{m^{2}}\dot{z}^{2}-\frac{1}{F(r)}\dot{r}^{2}&=0
\end{align}
if $\dot{\omega}=r^{2}\dot{\theta}^{2}+r^{2}\sin^{2}\theta\dot{\varphi}=0$. Since $\partial_{z}$ is Killing\footnote{This is a reformulation of the cylinder condition in the Kaluza-Klein theory.}, there exists a constant $\mathfrak{Z}\in\mathbb{R}$ such that
\begin{align}
\label{Link t. and z.}
	\g(\dot{\gamma},\partial_{z})&=sV(r)\dot{t}+\dot{z}=-m^{2}\mathfrak{Z}.
\end{align}
Plugging \eqref{Link t. and z.} into \eqref{gamma_in/out beginning} then yields
\begin{align}
\label{Regger wheeler origin}
	F(r)\dot{t}^{2}-\frac{1}{F(r)}\dot{r}^{2}&=m^{2}\mathfrak{Z}^{2}.
\end{align}
We build our geodesics so that $\mathfrak{Z}=0$. This is not only a convenient choice that makes computations easier and explicit: the geometric interpretation is that
\begin{align*}
	\g(\dot{\gamma},\partial_z)&=0
\end{align*}
that is, the principal null geodesics will be $\g$-orthogonal to the velocity vector field of the perfect fluid appearing in the stress-energy tensor after neutralization (see \eqref{Perfect fluid speed}). Solving \eqref{Link t. and z.} and \eqref{Regger wheeler origin} for $\mathfrak{Z}=0$, we get
\begin{align*}
	\frac{\mathrm{d}t}{\mathrm{d}r}&=\pm F(r)^{-1},\qquad\qquad\frac{\mathrm{d}z}{\mathrm{d}r}=\mp sV(r)F(r)^{-1}.
\end{align*}
In particular, we can parametrize our geodesic by $\mu=\pm r$.

 
The \textit{incoming principal null geodesic} $\gamma_{\mathrm{in}}(\mu)=(t(\mu),z(\mu),r(\mu),\omega(\mu))$ is the null geodesic parametrized by $\mu=-r$ and defined by
\begin{align*}
	\begin{cases}
		\dot{t}(\mu)~=~F(r)^{-1}\\
		\dot{z}(\mu)~=~-sV(r)F(r)^{-1}\\
		\dot{r}(\mu)~=~-1\\
		\dot{\omega}(\mu)~=~0
	\end{cases}.
\end{align*}
We then define the \textit{extended-star} coordinates\footnote{The map $(t,z,r,\omega)\mapsto(t^{\star},z^{\star},r,\omega)$ is one-to-one in each block $\M_{2}$, $\M_{3}$ and $\M_{4}$ since $T$ and $Z$ are (their derivative are respectively $F^{-1}$ and $-(sVF)^{-1}$ which have constant signs) and its Jacobian determinant is $1$; it therefore defines a coordinates chart.} $(t^{\star},z^{\star},r,\omega)$ with
\begin{align}
\label{t^*}
	t^{\star}&:=t+T(r),\qquad T(r):=\int_{\mathfrak{r}}^{r}\frac{\mathrm{d}\rho}{F(\rho)}=\sum_{\alpha\in\{n,c,-,+\}}\frac{1}{2\kappa_{\alpha}}\ln\left|\frac{r-r_{\alpha}}{\mathfrak{r}-r_{\alpha}}\right|,\\
\label{z^*}
	z^{\star}&:=z+Z(r),\qquad Z(r):=-\int_{\mathfrak{r}}^{r}\frac{sV(\rho)\mathrm{d}\rho}{F(\rho)}=-\sum_{\alpha\in\{n,c,-,+\}}\frac{s}{2r_{\alpha}\kappa_{\alpha}}\ln\left|\frac{\mathfrak{r}}{r}\frac{r-r_{\alpha}}{\mathfrak{r}-r_{\alpha}}\right|
\end{align}
for some $\mathfrak{r}\in\M_{\textup{ext}}$. Here we have used the factorization
\begin{align}
\label{Factorization F}
	F(r)&=\frac{\Lambda}{3r^{2}}(r-r_{n})(r-r_{c})(r-r_{-})(r_{+}-r)
\end{align}
as well as the \textit{surface gravities} $\kappa_{n}:=F'(r_{n})/2>0$, $\kappa_{c}:=F'(r_{c})/2<0$, $\kappa_{-}:=F'(r_{-})/2>0$ and $\kappa_{+}:=F'(r_{+})/2<0$ (we will give a more precise meaning of them in Subsection \ref{Surface gravities and Killing horizons}). The function $T$ is the so-called \textit{Regge-Wheeler} (or \textit{tortoise}) coordinate, and will be denoted by $x$ in Section \ref{Analytic scattering theory} and Section \ref{Geometric interpretation}. The expression of the extended metric in these coordinates is given by
\begin{align*}
	\g_{\star}&=\left(F(r)-\frac{s^{2}V(r)^{2}}{m^{2}}\right)(\mathrm{d}t^{\star})^{2}-\frac{sV(r)}{m^{2}}\left(\mathrm{d}t^{\star}\mathrm{d}z^{\star}+\mathrm{d}z^{\star}\mathrm{d}t^{\star}\right)-\frac{1}{m^{2}}(\mathrm{d}z^{\star})^{2}-\left(\mathrm{d}t^{\star}\mathrm{d}r+\mathrm{d}r\mathrm{d}t^{\star}\right)-r^{2}\mathrm{d}\omega^{2}
\end{align*}
with inverse
\begin{align*}
	(\g_{\star})^{-1}&=-m^{2}\partial_{z^{\star}}^{\otimes 2}-\left(\partial_{t^{\star}}\otimes\partial_{r}+\partial_{r}\otimes\partial_{t^{\star}}\right)+sV(r)\left(\partial_{z^{\star}}\otimes\partial_{r}+\partial_{r}\otimes\partial_{z^{\star}}\right)\nonumber\\
	&-F(r)\partial_{r}^{\otimes 2}-\frac{1}{r^{2}}\partial_{\theta}^{\otimes 2}-\frac{1}{r^{2}\sin^{2}\theta}\partial_{\varphi}^{\otimes 2}.
\end{align*}
Observe that, by construction, we have $\dot{t}^{\star}=\dot{z}^{\star}=0$: this shows that $t\to\pm\infty$ as $r\to(r_{\pm})^{\mp}$ (but $t^{\star}$ remains smooth because $\dot{t}^{\star}=0$). The same conclusion can be drawn for $z$.

We similarly define the \textit{outgoing principal null geodesic} $\gamma_{\mathrm{out}}(\mu)=(t(\mu),z(\mu),r(\mu),\omega(\mu))$ as the null geodesic parametrized by $\mu=r$ and defined by
\begin{align*}
	\begin{cases}
		\dot{t}(\mu)~=~F(r)^{-1}\\
		\dot{z}(\mu)~=~-sV(r)F(r)^{-1}\\
		\dot{r}(\mu)~=~1\\
		\dot{\omega}(\mu)~=~0
	\end{cases}.
\end{align*}
We then define the \textit{star-extended} coordinates $(^{\star}t,{}^{\star}z,r,\omega)$ with
\begin{align*}
	^{\star}t&:=t-T(r),\qquad ^{\star}z:=z-Z(r)
\end{align*}
with $T$ and $Z$ as in \eqref{t^*} and \eqref{z^*}. The expression of the extended metric in these coordinates is given by
\begin{align*}
	_{\star}\g&=\left(F(r)-\frac{s^{2}V(r)^{2}}{m^{2}}\right)(\mathrm{d}^{\star}t)^{2}-\frac{sV(r)}{m^{2}}\left(\mathrm{d}^{\star}t\mathrm{d}^{\star}z+\mathrm{d}^{\star}z\mathrm{d}^{\star}t\right)-\frac{1}{m^{2}}(\mathrm{d}^{\star}z)^{2}+\left(\mathrm{d}^{\star}t\mathrm{d}r+\mathrm{d}r\mathrm{d}^{\star}t\right)-r^{2}\mathrm{d}\omega^{2}
\end{align*}
with inverse
\begin{align*}
	(_{\star}\g)^{-1}&=-m^{2}\partial_{z^{\star}}^{\otimes 2}+\left(\partial_{t^{\star}}\otimes\partial_{r}+\partial_{r}\otimes\partial_{t^{\star}}\right)-sV(r)\left(\partial_{z^{\star}}\otimes\partial_{r}+\partial_{r}\otimes\partial_{z^{\star}}\right)\nonumber\\
	&-F(r)\partial_{r}^{\otimes 2}-\frac{1}{r^{2}}\partial_{\theta}^{\otimes 2}-\frac{1}{r^{2}\sin^{2}\theta}\partial_{\varphi}^{\otimes 2}.
\end{align*}
We have $^{\star}\dot{t}=\,\!^{\star}\dot{z}=0$ which entails $t\to\mp\infty$ and $z\to\mp\infty$ as $r\to(r_{\pm})^{\mp}$.%

%
%
%
%
%
%
\subsection{Surface gravities and Killing horizons}
\label{Surface gravities and Killing horizons}
\textit{Surface gravities} are accelerations felt in the incoming direction locally near a hypersurface by an unit test mass due to the gravitational force (measured infinitely far away from the hypersurface). We compute in this Subsection the surface gravities $\kappa_{\alpha}$ at $r=r_\alpha$ for all $\alpha\in\{c,-,+\}$.

Consider the following normalization of the velocity vector field $\nabla t$ associated to a static observer:
\begin{align}
\label{Static observer}
	F(r)\nabla t&=\partial_{t}-sV(r)\partial_{z}.
\end{align}
For De Sitter Kerr black holes, $\frac{\nabla t}{\sqrt{\nabla^{i}t\nabla_{i}t}}$ is the velocity vector field which follows the rotation of the black hole; it tends to $\partial_{t}-\frac{a}{r_{\pm}^{2}+a^{2}}\partial_{\varphi}$ as $r\to r_{\pm}$ which provides two Killing vector fields in the ergoregions near the horizons. The equivalent in our setting of the rotations at the horizons $\frac{a}{r_{\pm}^{2}+a^{2}}$ are the terms $sV_{\alpha}$, $\alpha\in\{c,-,+\}$. We thus consider the Killing vector fields
\begin{align}
\label{Killing generators}
	X_{c}&:=\partial_{t^{\star}}-sV_{c}\partial_{z^{\star}},\qquad\qquad X_{-}:=\partial_{t^{\star}}-sV_{-}\partial_{z^{\star}},\qquad\qquad X_{+}:=\partial_{t^{\star}}-sV_{+}\partial_{z^{\star}}
\end{align}
which are null at $r=r_{\alpha}$, $\alpha\in\{c,-,+\}$:
\begin{align*}
	\g_{\star}(X_\alpha,X_\alpha)&=F-\frac{s^{2}(V-V_\alpha)^{2}}{m^{2}}.
\end{align*}
\begin{remark}
	We used the extended-starr coordinates to define surface gravities. It is perfectly fine doing it with the star-extended coordinates, the only difference being the opposite sign we have to put in formula \eqref{Newton eq surface gravities} below.
\end{remark}

The physical interpretation of the surface gravities leads to the following Newtonian forces equilibrium\footnote{Using that orthogonal null vector fields are collinear, we only know in general that $\kappa_{\alpha}$ is at least a smooth function $\{r=r_{\alpha}\}\to\mathbb{R}$.}:
\begin{align}
\label{Newton eq surface gravities}
	\restriction{\left(\nabla_{X_{\alpha}}X_{\alpha}\right)}{\{r=r_{\alpha}\}}&=-\kappa_{\alpha}\restriction{X_{\alpha}}{\{r=r_{\alpha}\}}\qquad\kappa_{\alpha}\in\mathbb{R}.
\end{align}
The constants $\kappa_{\alpha}$ are well defined since the surface gravities are constant on $\{r=r_{\alpha}\}$ by the dominant energy condition \eqref{Weak condition Lambda} (see \textit{e.g.} \cite[equation (12.5.31)]{Wa}). It is easy to see that $\nabla^{\mu}\big((X_{\alpha})^\nu(X_{\alpha})_\nu\big)=-2\kappa_{\alpha}(X_{\alpha})^{\mu}$ on the corresponding horizon (\textit{cf.} \cite[equation (12.5.2)]{Wa}). We then compute:
\begin{align*}
	\restriction{\nabla^\mu\big((X_{\alpha})^\nu(X_{\alpha})_\nu\big)}{\{r=r_{\alpha}\}}&=\restriction{\left((\g_{\star})^{\mu r}\partial_{r}\!\left(F-\frac{s^{2}(V-V_{\alpha})^{2}}{m^{2}}\right)\right)}{\{r=r_{\alpha}\}}.
\end{align*}
Since $(\g_{\star})^{tr}=-1$ and $(\g_{\star})^{zr}=sV(r)$ are the only non-zero coefficients of the form $(\g_{\star})^{\mu r}$, we get in all cases:
\begin{align*}
	\kappa_{\alpha}&=\frac{F'(r_{\alpha})}{2}=\frac{(3r_{\alpha}-3M-2\Lambda r_{\alpha}^2)}{3r_{\alpha}^2}.
\end{align*}
Observe that $\kappa_{\alpha}$ is nothing but the De Sitter-Reissner-Nordstr\"{o}m surface gravity at the corresponding horizon (we show in the next Subsection that $\{r=r_{\alpha}\}$ are still Killing horizons in the extended spacetime). Using the signs of $F'$ at $r=r_{\alpha}$, we get:
\begin{align*}
	\kappa_{c}&<0,\qquad\qquad\kappa_{-}>0,\qquad\qquad\kappa_{+}<0.
\end{align*}

The surface gravities $\kappa_{\alpha}$ provide the rate of convergence of the Boyer-Lindquist coordinate $r$ to $r_\alpha$ in terms of the Regge-Wheeler coordinate $T(r)$ introduced in equation \eqref{t^*}: indeed, using the factorization \eqref{Factorization F}, we find
\begin{align}
\label{Decay r-r_- with kappa_pm}
	|r-r_\alpha|&=|\mathfrak{r}-r_{\alpha}|\left(\prod_{\substack{\beta\in\{n,c,-,+\}\\\beta\neq\alpha}}\left|\frac{\mathfrak{r}-r_{\beta}}{r-r_{\beta}}\right|^{\frac{\kappa_{\alpha}}{\kappa_{\beta}}}\right)\mathrm{e}^{2\kappa_{\alpha}T(r)}=\mathcal{O}_{r\to r_\pm}\big(\mathrm{e}^{2\kappa_{\alpha}T(r)}\big).
\end{align}

Let now $r_{0}>0$ and define the hypersurface $\Sigma_{r_{0}}:=\mathbb{R}_t\times\mathbb{S}^1_z\times\{r=r_0\}\times\mathbb{S}^2_\omega$. Using then the inverse metric expression in the extended-star coordinates, we compute
\begin{align}
\label{Normal to Sigma_pm}
	n_{\Sigma_{r_{0}}}:=\nabla r&=-\partial_{t^{\star}}+sV(r)\partial_{z^{\star}}-F(r)\left(\frac{\partial}{\partial r}\right)_{E{\star}}
\end{align}
where $\left(\frac{\partial}{\partial r}\right)_{E{\star}}$ is the vector field $\partial_{r}$ in the extended-star coordinates. Since
\begin{align*}
\g_{\star}\left(n_{\Sigma_{r_{0}}},-\left(\frac{\partial}{\partial r}\right)_{E{\star}}\right)&=1,
\end{align*}
$n_{\Sigma_{r_{0}}}$ is a incoming normal vector field to $\Sigma_{r_{0}}$. Besides,
\begin{align*}
	\g_{\star}(n_{\Sigma_{r_{0}}},n_{\Sigma_{r_{0}}})&=-F(r)
\end{align*}
so that the incoming normal $n_{\Sigma_{r_{\alpha}}}$ is also tangent to $\Sigma_{r_{\alpha}}$ and this hypersurface is null (we could have also seen this by observing that $\det(\restriction{\g_{\star}}{\Sigma_{r_{\alpha}}})=0$). As
\begin{align*}
	\det(\g_{\star})&=\frac{r^{4}\sin^{2}\theta}{m^{2}}=\det(_{\star}\g)
\end{align*}
does not vanish at $r=r_{\alpha}$, $\Sigma_{r_{\alpha}}$ is not degenerate. Since $\nabla r\to X_{\alpha}$ as $r\to r_{\alpha}$, $\alpha\in\{c,-,+\}$, this hypersurface $\{r=r_{\alpha}\}$ is a \textit{Killing horizon}, that is a non-degenerate null hypersurface generated by a Killing vector field. We then define:
\begin{align*}
&\mathscr{H}_{c}^{+}:=\mathbb{R}_{t^{\star}}\times\mathbb{S}^{1}_{z^{\star}}\times\{r_{c}\}_{r}\times\mathbb{S}^{2}_{\omega}\qquad\qquad(\textit{future Cauchy horizon}),\\
&\mathscr{H}_{c}^{-}:=\mathbb{R}_{^{\star}t}\times\mathbb{S}^{1}_{^{\star}z}\times\{r_{c}\}_{r}\times\mathbb{S}^{2}_{\omega}\qquad\qquad(\textit{past Cauchy horizon}),\\[2.5mm]
&\mathscr{H}^{+}:=\mathbb{R}_{t^{\star}}\times\mathbb{S}^{1}_{z^{\star}}\times\{r_{-}\}_{r}\times\mathbb{S}^{2}_{\omega}\qquad\qquad(\textit{future event horizon}),\\
&\mathscr{H}^{-}:=\mathbb{R}_{^{\star}t}\times\mathbb{S}^{1}_{^{\star}z}\times\{r_{-}\}_{r}\times\mathbb{S}^{2}_{\omega}\qquad\qquad(\textit{past event horizon}),\\[2.5mm]
&\mathscr{I}^{+}:=\mathbb{R}_{t^{\star}}\times\mathbb{S}^{1}_{z^{\star}}\times\{r_{+}\}_{r}\times\mathbb{S}^{2}_{\omega}\qquad\qquad(\textit{future cosmological horizon}),\\
&\mathscr{I}^{-}:=\mathbb{R}_{^{\star}t}\times\mathbb{S}^{1}_{^{\star}z}\times\{r_{+}\}_{r}\times\mathbb{S}^{2}_{\omega}\qquad\qquad(\textit{past cosmological horizon}).
\end{align*}
Observe that the construction of the horizons is not complete so far: we need to add two 3-surfaces where $\mathscr{H}^{+}$ and $\mathscr{H}^{-}$ on the one hand, $\mathscr{I}^{+}$ and $\mathscr{I}^{-}$ on the other hand, meet. This will be done in Subsection \ref{Crossing rings}.
%
%
%
%
\subsection{Crossing rings}
\label{Crossing rings}
In the previous Subsections, we have constructed Killing horizons. We now build sets of codimension 2 which will complete the construction of the hypersurfaces $\{r=r_\alpha\}$ for all $\alpha\in\{c,-,+\}$. As only the event and cosmological horizons will be concerned in the next Sections, we will only detail computations related to them.
 
We start with the event horizon. Following (8.19) and (8.20) in \cite{HaNi04}, we define the \textit{Kruskal-Boyer-Lindquist coordinates}
\begin{align*}
	u&:=\mathrm{e}^{-\kappa_{-}\!{}^{\star}t},\quad v:=\mathrm{e}^{\kappa_{-}t^{\star}},\quad z^{\sharp}:=z-\restriction{\left(\frac{\dot{z}^{\star}}{\dot{t}^{\star}}\right)}{r=r_-}t=z+sV_{-}t.
\end{align*}
The variable $z^{\sharp}$ must be understood as an element of $\mathbb{S}^{1}$ (\textit{i.e.} it is defined modulo $2\pi$), the rotation direction on the circle being imposed by the sign of the charge product $s$. It is introduced to follow the "rotation" of the event horizon (by analogy with the Kerr case); concretely, it cancels the "rotation" term $\frac{s^{2}V(r)^{2}}{m^{2}}$ in front of $\mathrm{d}t^{2}$ in the extended metric as $r\to r_{-}$. The extended metric $\g$ now reads in these coordinates
\begin{align*}
\g&=\frac{s^2\big(V(r)-V_{-}\big)^{2}}{4m^{2}\kappa_{-}^{2}u^{2}v^{2}}\Big(uv\big(\mathrm{d}u\mathrm{d}v+\mathrm{d}v\mathrm{d}u\big)-u^{2}\mathrm{d}v^{2}-v^{2}\mathrm{d}u^{2}\Big)\\
&-\frac{s\big(V(r)-V_{-}\big)}{2m^{2}\kappa_{-}uv}\Big(\big(u\mathrm{d}v-v\mathrm{d}u\big)\mathrm{d}z^{\sharp}+\mathrm{d}z^{\sharp}\big(u\mathrm{d}v-v\mathrm{d}u\big)\Big)-\frac{1}{m^{2}}(\mathrm{d}z^{\sharp})^{2}-r^{2}\mathrm{d}\omega^{2}.
\end{align*}
We have to ensure that the metric is well defined and non degenerate at $r=r_{-}$ (that is when $u=v=0$). Set
\begin{align*}
	G_{-}(u,v,r)&:=\frac{r-r_{-}}{uv}=(r-r_-)\mathrm{e}^{-2\kappa_{-}T(r)}.
\end{align*}
The function $G_{-}$ is analytic and non vanishing near $r=r_{-}$ because of \eqref{Decay r-r_- with kappa_pm}. It follows that the extended metric
\begin{align*}
	\g&=-\frac{s^2G_{-}(u,v,r)^{2}}{4m^{2}r^{2}r_{-}^{2}\kappa_{-}^{2}}\big(u^{2}\mathrm{d}v^{2}+v^{2}\mathrm{d}u^{2}\big)-\frac{s^2G_{-}(u,v,r)(r-r_{-})}{4m^{2}r^{2}r_{-}^{2}\kappa_{-}^{2}}\big(\mathrm{d}u\mathrm{d}v+\mathrm{d}v\mathrm{d}u\big)\\
	&+\frac{sG_{-}(u,v,r)}{2m^{2}rr_{-}\kappa_{-}}\Big(\big(u\mathrm{d}v-v\mathrm{d}u\big)\mathrm{d}z^{\sharp}+\mathrm{d}z^{\sharp}\big(u\mathrm{d}v-v\mathrm{d}u\big)\Big)-\frac{1}{m^{2}}(\mathrm{d}z^{\sharp})^{2}-r^{2}\mathrm{d}\omega^{2}
\end{align*}
extends smoothly on a neighbourhood of $\big\{u=v=0\big\}$. Call the set
\begin{align*}
	\mathscr{R}^{\mathscr{H}}_{c}:=\{u=0\}\times\{v=0\}\times\mathbb{S}^{1}_{z^{\sharp}}\times\mathbb{S}^{2}_{\omega}
\end{align*}
the \textit{crossing ring at the event horizon} (this is the equivalent of the crossing sphere in the usual case of Kerr-Newman black holes).

We now turn to the construction of the complete cosmological horizon. Computations are so similar that we will omit most of them. Define the new coordinates
\begin{align*}
	\tilde{u}&:=\mathrm{e}^{-\kappa_{+}\!{}^{\star}t},\quad \tilde{v}:=\mathrm{e}^{\kappa_{+}t^{\star}},\quad \,\!^{\sharp}\!z:=z-\restriction{\left(\frac{^{\star}\dot{z}}{^{\star}\dot{t}}\right)}{r=r_+}t=z+sV_{+}t.
\end{align*}
Recall that $\kappa_{+}<0$. The extended metric $\g$ reads in these coordinates
\begin{align*}
	\g&=-\frac{s^2G_{+}(\tilde{u},\tilde{v},r)^{2}}{4m^{2}r^{2}r_{+}^{2}\kappa_{+}^{2}}\big(\tilde{u}^{2}\mathrm{d}\tilde{v}^{2}+\tilde{v}^{2}\mathrm{d}\tilde{u}^{2}\big)+\frac{s^2G_{+}(\tilde{u},\tilde{v},r)(r-r_{+})}{4m^{2}r^{2}r_{+}^{2}\kappa_{+}^{2}}\big(\mathrm{d}\tilde{u}\mathrm{d}\tilde{v}+\mathrm{d}\tilde{v}\mathrm{d}\tilde{u}\big)\\
	&-\frac{sG_{+}(\tilde{u},\tilde{v},r)^{2}}{2m^{2}rr_{+}\kappa_{+}}\Big(\big(\tilde{u}\mathrm{d}\tilde{v}-\tilde{v}\mathrm{d}\tilde{u}\big)\mathrm{d}^{\sharp}\!z+\mathrm{d}^{\sharp}\!z\big(\tilde{u}\mathrm{d}\tilde{v}-\tilde{v}\mathrm{d}\tilde{u}\big)\Big)-\frac{1}{m^{2}}(\mathrm{d}^{\sharp}\!z)^{2}-r^{2}\mathrm{d}\omega^{2}
\end{align*}
with
\begin{align*}
	G_{+}(\tilde{u},\tilde{v},r)&:=\frac{r_{+}-r}{\tilde{u}\tilde{v}}=(r_+-r)\mathrm{e}^{-2\kappa_{+}T(r)}.
\end{align*}
which is analytic and non vanishing near $r=r_{+}$ because of \eqref{Decay r-r_- with kappa_pm}. Then $\g$ extends smoothly on a neighbourhood of $\big\{\tilde{u}=\tilde{v}=0\big\}$. Call the set
\begin{align*}
	\mathscr{R}^{\mathscr{I}}_{c}:=\{\tilde{u}=0\}\times\{\tilde{v}=0\}\times\mathbb{S}^{1}_{\,\!^{\sharp}\!z}\times\mathbb{S}^{2}_{\omega}
\end{align*}
the \textit{crossing ring at the cosmological horizon}.

Now the construction is complete. The event horizon is defined as
\begin{align*}
	\mathscr{H}&:=\mathscr{H}^{-}\cup\mathscr{R}^{\mathscr{H}}_{c}\cup\mathscr{H}^{+}\\
	&=\big([0,+\infty[_{u}\times\{0\}_{v}\times\mathbb{S}^{1}_{z^{\sharp}}\times\mathbb{S}^{2}_{\omega}\big)\cup\big(\{0\}_{u}\times[0,+\infty[_{v}\times\mathbb{S}^{1}_{z^{\sharp}}\times\mathbb{S}^{2}_{\omega}\big)
\end{align*}
and the cosmological horizon is
\begin{align*}
	\mathscr{I}&:=\mathscr{I}^{-}\cup\mathscr{R}^{\mathscr{I}}_{c}\cup\mathscr{I}^{+}\\
	&=\big(\{0\}_{\tilde{u}}\times[0,+\infty[_{\tilde{v}}\times\mathbb{S}^{1}_{\,\!^{\sharp}\!z}\times\mathbb{S}^{2}_{\omega}\big)\cup\big([0,+\infty[_{\tilde{u}}\times\{0\}_{\tilde{v}}\times\mathbb{S}^{1}_{\,\!^{\sharp}\!z}\times\mathbb{S}^{2}_{\omega}\big).
\end{align*}
The outer space can be extended itself as the \textit{global outer space}
\begin{align*}
	\MM&:=\M\cup\mathscr{H}\cup\mathscr{I}.
\end{align*}
As discussed at the beginning of this Section, the compactification shown in Figure \ref{Block I uncomplete} is not really complete as some points still lie at infinity whatever the coordinate system is. They are the future timelike infinity $i^{+}$ and the past timelike infinity $i^{-}$.

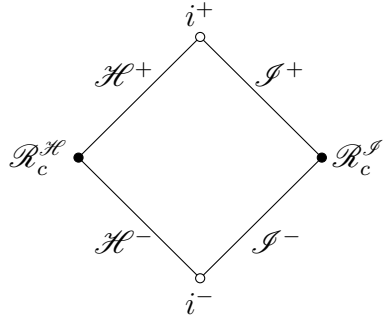
\begin{figure}[!h]
	\centering
	\captionsetup{justification=centering,margin=2cm}
		\begin{tikzpicture}[scale=0.8]
		\draw[](-2,0)--(0,2)--(2,0)--(0,-2)--cycle;
		\filldraw[fill=black](-2,0)circle[x radius=0.075,y radius=0.075]node[left]{$\mathscr{R}^{\mathscr{H}}_{c}$};
		\filldraw[fill=black](2,0)circle[x radius=0.075,y radius=0.075]node[right]{$\mathscr{R}^{\mathscr{I}}_{c}$};
		\filldraw[fill=white](0,2)circle[x radius=0.075,y radius=0.075]node[above]{$i^{+}$};
		\filldraw[fill=white](0,-2)circle[x radius=0.075,y radius=0.075]node[below]{$i^{-}$};
		\draw[](-1.25,1)node[left, above]{$\mathscr{H}^{+}$};
		\draw[](1.30,1)node[right, above]{$\mathscr{I}^{+}$};
		\draw[](-1.25,-1)node[left, below]{$\mathscr{H}^{-}$};
		\draw[](1.25,-1)node[right, below]{$\mathscr{I}^{-}$};
		\end{tikzpicture}
		\caption{\label{Block I uncomplete}The global outer space $\MM$ (each point in the diagram is a copy of $\mathbb{S}^{1}\times\mathbb{S}^{2}$). Only the points $i^{+}$ and $i^{-}$ are at infinity.}
\end{figure}
%
%
%
%
%
\subsection{Black rings}
\label{Black rings}
This Subsection is devoted to the construction of a global structure of $\M_{\textup{ext}}$. The procedure is similar to the standard one for De Sitter-Reissner-Nordstr\"om spacetime and will reveal the black rings structure of $\M_{\textup{ext}}$.

\paragraph{Conformal infinity.} We start by including the infinity to the region $\{r\geq r_+\}$. Recall here that $V(r)=1/r$. Let us set the \textit{conformal factor} $\Omega:=1/r$. In the coordinates $(t^{\star},z^{\star},R,\omega)$ with $R:=1/r\in\left[0,1/r_{+}\right]$, we define the conformal extended-star metric
\begin{align*}
	\hat{g}_{\star}&:=\Omega^{2}\g_{\star}\\
	&=R^{2}\left(F(1/R)-\frac{s^{2}R^{2}}{m^{2}}\right)(\mathrm{d}t^{\star})^{2}-\frac{sR^{3}}{m^{2}}\left(\mathrm{d}t^{\star}\mathrm{d}z^{\star}+\mathrm{d}z^{\star}\mathrm{d}t^{\star}\right)-\frac{R^{2}}{m^{2}}(\mathrm{d}z^{\star})^{2}+\left(\mathrm{d}t^{\star}\mathrm{d}R+\mathrm{d}R\mathrm{d}t^{\star}\right)-\mathrm{d}\omega^{2}
\end{align*}
and similarly, in the coordinates $(^{\star}t,{}^{\star}z,R,\omega)$, the conformal star-extended metric is
\begin{align*}
	_{\star}\hat{g}&:=\Omega^{2}\,\!_{\star}\g\\
	&=R^{2}\left(F(1/R)-\frac{s^{2}R^{2}}{m^{2}}\right)(\mathrm{d}^{\star}t)^{2}-\frac{sR^{3}}{m^{2}}\left(\mathrm{d}^{\star}t\mathrm{d}^{\star}z+\mathrm{d}^{\star}z\mathrm{d}^{\star}t\right)-\frac{R^{2}}{m^{2}}(\mathrm{d}^{\star}z)^{2}-\left(\mathrm{d}^{\star}t\mathrm{d}R+\mathrm{d}R\mathrm{d}^{\star}t\right)-\mathrm{d}\omega^{2}.
\end{align*}
We can check that $\det\hat{g}_{\star}=\det\,\!_{\star}\hat{g}=\frac{R^{2}}{m^{2}}$. Hence, the following hypersurfaces are null:
\begin{align*}
	\mathscr{I}^{+}_{\infty}&:=\mathbb{R}_{^{\star}t}\times\mathbb{S}^{1}_{^{\star}z}\times\{0\}_{R}\times\mathbb{S}^{2}_{\omega}&&(\textit{future null infinity}),\\
	\mathscr{I}^{-}_{\infty}&:=\mathbb{R}_{t^{\star}}\times\mathbb{S}^{1}_{z^{\star}}\times\{0\}_{R}\times\mathbb{S}^{2}_{\omega}&&(\textit{past null infinity}),\\
	\mathscr{I}_{\infty}&:=\mathscr{I}^{+}_{\infty}\cup\mathscr{I}^{-}_{\infty}&&(\textit{null infinity}).
\end{align*}
They are the sets of "end points" at infinity of the principal null geodesics; observe that they do not intersect (the spacelike infinity still lies at infinity as for the De Sitter-Kerr-Newman family). The conformal metrics then restrict to
\begin{align*}
	\restriction{\hat{g}_{\star}}{\mathscr{I}^{-}_{\infty}}&=-\frac{\Lambda}{3}(\mathrm{d}t^{\star})^{2}-\mathrm{d}\omega^{2}=\restriction{_{\star}\hat{g}}{\mathscr{I}^{+}_{\infty}}.
\end{align*}
\paragraph{Time orientation.} We now define a global time function in $\M_{\textup{ext}}$; recall that a \textit{time function} $\tau$ is a $\mathcal{C}^{1}$ function such that $\nabla\tau$ is timelike. The extended spacetime $\M_{\textup{ext}}$ is not connected as we remove the hypersurfaces $\Sigma_{r_{\alpha}}$ for $\alpha\in\{c,-,+\}$. As a result, there is no canonical way of defining a time-orientation on it. Since
\begin{align*}
	\nabla t&=F(r)^{-1}\big(\partial_{t}-sV(r)\partial_{z}\big),&&\g(\nabla t,\nabla t)=F(r)^{-1}>0\quad\text{in $\M_{1}$ and $\M_{3}$},\\
	\nabla r&=-F(r)\partial_{r},&&\g(\nabla r,\nabla r)=-F(r)>0\quad\text{in $\M_{2}$ and $\M_{4}$},
\end{align*}
we see that $\pm t$ is a time orientation in the blocks $\M_{1}$ and $\M_{3}$ whereas $\pm r$ is a time orientation in the blocks $\M_{2}$ and $\M_{4}$. From now on, $\M_{j}$ denotes the corresponding block endowed with the time orientation $t$ if $j\in\{1,3\}$ or $r$ if $j\in\{2,4\}$, and $\M_{j}'$ denotes the same block but with the time orientation $-t$ if $j\in\{1,3\}$ or $-r$ if $j\in\{2,4\}$. We still write $\M_{\textup{ext}}:=\bigcup_{j=1}^{4}\M_{j}$ and set $\M_{\textup{ext}}':=\bigcup_{j=1}^{4}\M_{j}'$.

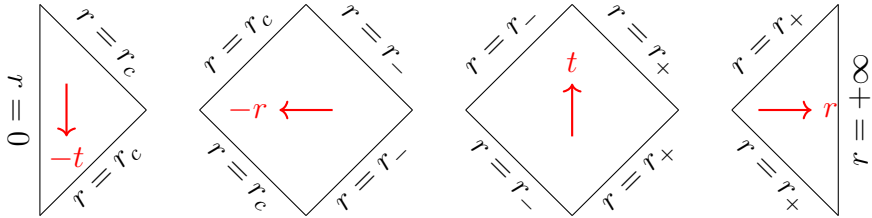
\begin{figure}[!h]
	\centering
	\captionsetup{justification=centering,margin=1.6cm}
	\begin{tikzpicture}[scale=0.7]
	\draw[](-2,-2)--(0,0)node[midway,below,sloped]{$r=r_c$};
	\draw[](0,0)--(-2,2)node[midway,above,sloped]{$r=r_c$};
	\draw[](-2,2)--(-2,-2)node[midway,below,sloped]{$r=0$};
	\draw[red, thick, ->](-1.5,0.5)--(-1.5,-0.5)node[below]{$-t$};
	\draw[](1,0)--(3,-2)node[midway,below,sloped]{$r=r_c$};
	\draw[](3,-2)--(5,0)node[midway,below,sloped]{$r=r_-$};
	\draw[](5,0)--(3,2)node[midway,above,sloped]{$r=r_-$};
	\draw[](3,2)--(1,0)node[midway,above,sloped]{$r=r_c$};
	\draw[red, thick, ->](3.5,0)--(2.5,0)node[left]{$-r$};
	\draw[](6,0)--(8,-2)node[midway,below,sloped]{$r=r_-$};
	\draw[](8,-2)--(10,0)node[midway,below,sloped]{$r=r_+$};
	\draw[](10,0)--(8,2)node[midway,above,sloped]{$r=r_+$};
	\draw[](8,2)--(6,0)node[midway,above,sloped]{$r=r_-$};
	\draw[red, thick, ->](8,-0.5)--(8,0.5)node[above]{$t$};
	\draw[](11,0)--(13,-2)node[midway,below,sloped]{$r=r_+$};
	\draw[](13,-2)--(13,2)node[midway,below,sloped]{$r=+\infty$};
	\draw[](13,2)--(11,0)node[midway,above,sloped]{$r=r_+$};
	\draw[red, thick, ->](11.5,0)--(12.5,0)node[right]{$r$};
	\end{tikzpicture}
	\caption{\label{4 Blocks}The four blocks $\M_{j}$ and their respective time orientation.}
\end{figure}
\paragraph{Carter-Penrose diagram of the extended spacetime.} We now construct a global structure $\MM_{\textup{ext}}$ which respects the time orientation of each block $\M_{j}$ and $\M_{j}'$.

Let $\MM_{\textup{DSRN}}$ be the maximal analytic extension of the De Sitter-Reissner-Nordstr\"{o}m spacetime (see \textit{e.g.} Subsection 1.2 and particularly the paragraph 1.2.5 in \cite{Mo17}). This extension essentially consists in defining appropriate coordinates near the positive\footnote{$r=0$ being a genuine singularity at which the metric can not be $\mathcal{C}^{2}$, we can not build an analytic extension near the negative root.} roots of $F$ in which the metric is analytic, so that all the work boils down to build it considering the quotient of the original spacetime by the action of the rotations group on $\mathbb{S}^{2}$. $\MM_{\textup{DSRN}}$ also satisfies a local inextendibility property: there is no open non-empty subset whose closure is non-compact and can be embedded in an analytic manifold with a relatively compact image (see \cite[Subsection 1.2]{Mo17}). We then define
\begin{align*}
	\MM_{\textup{ext}}&:=\MM_{\textup{DSRN}}\times\mathbb{S}^{1}
\end{align*}
which is nothing but the orbit of $\MM_{\textup{DSRN}}$ under the action  of the rotation on $\mathbb{S}^{1}$. 
It contains infinitely many isometric copies of $\M_{\textup{ext}}\sqcup\mathscr{H}_{c}\sqcup\mathscr{H}\sqcup\mathscr{I}\sqcup\mathscr{I}_{\infty}$ and $\M_{\textup{ext}}'\sqcup\mathscr{H}_{c}\sqcup\mathscr{H}\sqcup\mathscr{I}\sqcup\mathscr{I}_{\infty}$.

\begin{figure}[!h]
	\centering
	\captionsetup{justification=centering,margin=1.6cm}
	\begin{tikzpicture}[scale=0.6]
	%
	%
	\filldraw[gray!20](-4,-2)--(0,2)--(4,2)--(-4,-6)--cycle;
	\draw[dashed](4,6)--(2,8);
	\draw[dashed](8,6)--(10,8);
	\draw[thick](4,6)--(6,8)node[midway,above,sloped]{\small$r=r_-$};
	\draw[thick](6,8)--(8,6)node[midway,above,sloped]{\small$r=r_-$};
	\draw[](6,4)--(8,6)node[midway,above,sloped]{\small$r=r_c$};
	\draw[thick, blue](8,6)--(8,2)node[midway,above,sloped]{\small$r=0$};
	\draw[](6,4)--(4,6)node[midway,above,sloped]{\small$r=r_c$};
	\draw[thick, blue](4,6)--(4,2)node[midway,below,sloped]{\small$r=0$};
	\draw[](4,2)--(6,4)node[midway,above,sloped]{\small$r=r_c$};
	\draw[](6,4)--(8,2)node[midway,above,sloped]{\small$r=r_c$};
	\draw[](2,0)--(4,-2)node[midway,above,sloped]{\small$r=r_+$};
	\draw[thick](4,-2)--(6,0)node[midway,below,sloped]{\small$r=r_-$};
	\draw[thick](6,0)--(4,2)node[midway,above,sloped]{\small$r=r_-$};
	\draw[](4,2)--(2,0)node[midway,below,sloped]{\small$r=r_+$};
	\draw[thick](6,0)--(8,-2)node[midway,below,sloped]{\small$r=r_-$};
	\draw[](8,-2)--(10,0)node[midway,above,sloped]{\small$r=r_+$};
	\draw[](10,0)--(8,2)node[midway,below,sloped]{\small$r=r_+$};
	\draw[thick](8,2)--(6,0)node[midway,above,sloped]{\small$r=r_-$};
	\draw[](4,-2)--(6,-4)node[midway,below,sloped]{\small$r=r_c$};
	\draw[](6,-4)--(8,-2)node[midway,below,sloped]{\small$r=r_c$};
	\draw[](6,-4)--(4,-6)node[midway,below,sloped]{\small$r=r_c$};
	\draw[thick, blue](4,-6)--(4,-2)node[midway,above,sloped]{\small$r=0$};
	\draw[](6,-4)--(8,-6)node[midway,below,sloped]{\small$r=r_c$};
	\draw[thick, blue](8,-6)--(8,-2)node[midway,below,sloped]{\small$r=0$};
	\draw[thick](4,-6)--(6,-8)node[midway,below,sloped]{\small$r=r_-$};
	\draw[thick](6,-8)--(8,-6)node[midway,below,sloped]{\small$r=r_-$};
	\draw[dashed](4,-6)--(2,-8);
	\draw[dashed](8,-6)--(10,-8);
	\draw[](2,0)--(0,2)node[midway,below,sloped]{\small$r=r_+$};
	\draw[thick, red](0,2)--(4,2)node[midway,above,sloped]{\small$r=+\infty$};
	\draw[thick](0,2)--(-2,0)node[midway,above,sloped]{\small$r=r_-$};
	\draw[thick](-2,0)--(0,-2)node[midway,below,sloped]{\small$r=r_-$};
	\draw[thick](0,2)--(-2,4)node[midway,above,sloped]{\small$r=r_c$};
	\draw[thick](-4,2)--(-2,4)node[midway,above,sloped]{\small$r=r_c$};
	\draw[thick](-4,2)--(-2,0)node[midway,above,sloped]{\small$r=r_-$};
	\draw[thick](-2,4)--(0,6)node[midway,above,sloped]{\small$r=r_c$};
	\draw[thick](-2,4)--(-4,6)node[midway,above,sloped]{\small$r=r_c$};
	\draw[thick, blue](0,6)--(0,2)node[midway,above,sloped]{\small$r=0$};
	\draw[thick, blue](-4,6)--(-4,2)node[midway,below,sloped]{\small$r=0$};
	\draw[thick](0,6)--(-2,8)node[midway,above,sloped]{\small$r=r_-$};
	\draw[thick](-4,6)--(-2,8)node[midway,above,sloped]{\small$r=r_-$};
	\draw[dashed](0,6)--(2,8);
	\draw[dashed](-4,6)--(-6,8);
	\draw[thick](-4,2)--(-6,0)node[midway,below,sloped]{\small$r=r_+$};
	\draw[thick](-4,-2)--(-6,0)node[midway,above,sloped]{\small$r=r_+$};
	\draw[thick, red](-4,2)--(-8,2)node[midway,above,sloped]{\small$r=+\infty$};
	\draw[thick, red](-4,-2)--(-8,-2)node[midway,below,sloped]{\small$r=+\infty$};
	\draw[thick](-8,2)--(-10,0)node[midway,above,sloped]{\small$r=r_-$};
	\draw[thick](-8,-2)--(-10,0)node[midway,below,sloped]{\small$r=r_-$};
	\draw[thick](-8,-2)--(-6,0)node[midway,above,sloped]{\small$r=r_+$};
	\draw[thick](-8,2)--(-6,0)node[midway,below,sloped]{\small$r=r_+$};
	\draw[dashed](-8,-2)--(-10,-4);
	\draw[dashed](-8,2)--(-10,4);
	\draw[](2,0)--(0,-2)node[midway,above,sloped]{\small$r=r_+$};
	\draw[thick, red](0,-2)--(4,-2)node[midway,below,sloped]{\small$r=+\infty$};
	\draw[](10,0)--(12,2)node[midway,below,sloped]{\small$r=r_+$};
	\draw[thick, red](12,2)--(8,2)node[midway,above,sloped]{\small$r=+\infty$};
	\draw[](10,0)--(12,-2)node[midway,above,sloped]{\small$r=r_+$};
	\draw[thick, red](12,-2)--(8,-2)node[midway,below,sloped]{\small$r=+\infty$};
	\draw[thick](12,-2)--(14,0)node[midway,below,sloped]{\small$r=r_-$};
	\draw[thick](14,0)--(12,2)node[midway,above,sloped]{\small$r=r_-$};
	\draw[dashed](12,2)--(14,4);
	\draw[dashed](12,-2)--(14,-4);
	\draw[thick](-2,0)--(-4,-2)node[midway,below,sloped]{\small$r=r_-$};
	\draw[thick](-4,-2)--(-2,-4)node[midway,below,sloped]{\small$r=r_c$};
	\draw[thick](-2,-4)--(0,-2)node[midway,below,sloped]{\small$r=r_c$};
	\draw[thick](-2,-4)--(0,-6)node[midway,below,sloped]{\small$r=r_c$};
	\draw[thick, blue](0,-6)--(0,-2)node[midway,below,sloped]{\small$r=0$};
	\draw[thick, blue](-4,-6)--(-4,-2)node[midway,above,sloped]{\small$r=0$};
	\draw[thick](-2,-4)--(-4,-6)node[midway,below,sloped]{\small$r=r_c$};
	\draw[thick](-4,-6)--(-2,-8)node[midway,below,sloped]{\small$r=r_-$};
	\draw[thick](0,-6)--(-2,-8)node[midway,below,sloped]{\small$r=r_-$};
	\draw[thick, red](0,6)--(4,6)node[midway,below,sloped]{\small$r=+\infty$};
	\draw[thick, red](0,-6)--(4,-6)node[midway,above,sloped]{\small$r=+\infty$};
	\draw[dashed](-4,-6)--(-6,-8);
	\draw[dashed](0,-6)--(2,-8);
	%
	%
	%
	%
	\draw[](-3.4,-4.25)node[]{$\M_{1}$};
	\draw[](-2,-2.0)node[]{$\M_{2}$};
	\draw[](-0.25,0)node[]{$\M_{3}$};
	\draw[](2,1.25)node[]{$\M_{4}$};
	\draw[](7.4,4.5)node[]{$\M_{1}'$};
	\draw[](6,2.5)node[]{$\M_{2}'$};
	\draw[](4.45,0)node[]{$\M_{3}'$};
	\draw[](2,-1.25)node[]{$\M_{4}'$};
	\end{tikzpicture}
	\caption{\label{Global geometry figure}The global geometry of the extended spacetime. Time orientation goes from the bottom to the top of the figure.}
\end{figure}
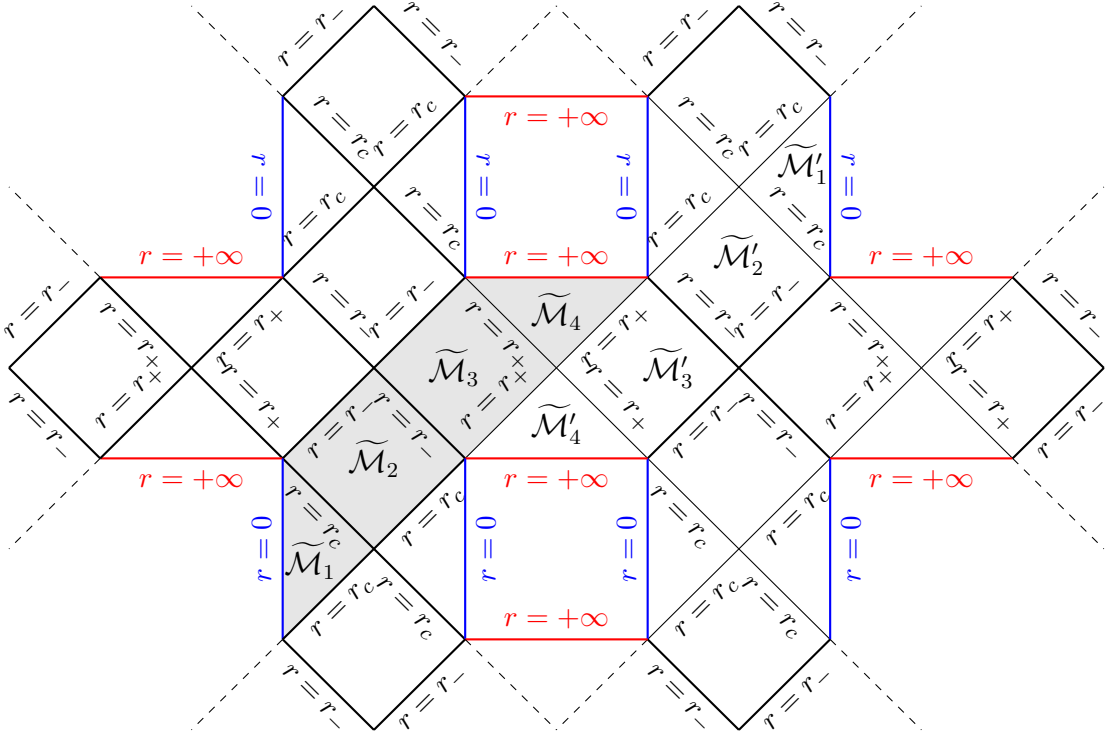
\paragraph{Black rings, white rings and worm rings} The neutralization procedure presented in Subsection \ref{The neutralization procedure} has turned the original De Sitter-Reissner-Nordstr\"om black hole into a stranger object, a \textit{black ring}. It is the equivalent of a black hole except that the topology of the horizon is $\mathbb{S}^{1}\times\mathbb{S}^{2}$. As the we will see, the maximal extension $\MM_{\textup{ext}}$ contains infinitely many of them so we have to distinguish the different copies of $\M_{\textup{ext}}$ from each others.

A \textit{piece of universe $\mathscr{P}$} is any copy of $\M_{\textup{ext}}\sqcup\mathscr{H}_{c}\sqcup\mathscr{H}\sqcup\mathscr{I}\sqcup\mathscr{I}_{\infty}$. We exclude in this definition the reverse time-oriented blocks $\M_{j}'$ as the physical block (that containing the Earth) is any copy of the outer space $\M_{3}$. Adapting the definition of Wald \cite{Wa} (see Subsection 12.1 therein which deals with the Kerr case), we will say that:
\begin{itemize}
	\item a non-empty closed subset $\mathcal{B}\subset\MM_{\textup{ext}}$ is a \textit{black ring} if $\partial\mathcal{B}=\mathbb{R}\times\mathbb{S}^{1}\times\mathbb{S}^{2}$ and if for all inextendible causal future-pointing geodesic $\gamma:\mathbb{R}\to\MM_{\textup{ext}}$ starting in $\M_{3}$ and entering $\mathcal{B}$ at some $\mu_{0}$, then
	\begin{align*}
		\bigcup_{\mu\geq\mu_{0}}\{\gamma(\mu)\}\cap\M_{4}\cap\mathscr{P}&=\emptyset;
	\end{align*}
	\item a non-empty subset $\mathcal{W}\subset\MM_{\textup{ext}}$ is a \textit{white ring} if it is a black ring for the reverse time orientation;
	\item a non-empty subset $\mathcal{S}\subset\MM_{\textup{ext}}$ is a \textit{worm ring} $\mathcal{S}$ if there exist a black ring $\mathcal{B}$ and a white ring $\mathcal{W}$ such that $\mathcal{S}=\mathcal{B}\cap\mathcal{W}$.
\end{itemize}
A black ring thus prevents any inextendible causal future-pointing geodesic entering inside it to escape at infinity in the same piece of universe: it can escape only in another copy of the universe. Observe that $\partial\mathcal{B}$ and $\partial\mathcal{W}$ are necessarily null.

We can similarly define the more standard notions of black/white/worm holes by requiring the topology $\mathbb{R}\times\mathbb{S}^{3}$ for the boundary. It turns that the extended spacetime contains only black/white/worm rings.
\begin{lemma}
	\begin{enumerate}
		\item All the copies of $\M_{2}'\sqcup\mathscr{H}_{c}\sqcup\mathscr{H}$ are black rings.
		\item All the copies of $\M_{2}\sqcup\mathscr{H}_{c}\sqcup\mathscr{H}$ are white rings.
		\item All the copies of $\M_{1}\sqcup\mathscr{H}_{c}\sqcup\mathscr{H}$ and $\M_{1}'\sqcup\mathscr{H}_{c}\sqcup\mathscr{H}$ are worm rings.
		\item There is no other black/white/worm ring in $\MM_{\textup{ext}}$. Furthermore, $\MM_{\textup{ext}}$ contains no black/white/worm hole.
	\end{enumerate}
\end{lemma}
\begin{proof}
	\begin{enumerate}
		\item Let $\gamma=(t,z,r,\omega)$ be an inextendible causal future-pointing geodesic which starts in $\left]r_-,r_+\right[$ and such that $r(\mu_0)\in\left]r_c,r_-\right[$ for some $\mu_0\in I$ (such a curve exists, take \textit{e.g.} the incoming principal null geodesics $\gamma_{\textup{in}}$). Since $r(\mu)<r_-$ for some $\mu<\mu_0$, there exists a proper time $\underline{\mu}\in I$, $\underline{\mu}\leq \mu_0$, such that $\dot{r}(\underline{\mu})<0$. Using that $-r$ is a time orientation in $\M_{2}'$, we have
		\begin{align*}
			\dot{r}&=\nabla_{\dot{\gamma}}r=\g_{\mu\nu}\dot{\gamma}^{\mu}\nabla^{\nu}r\leq 0.
		\end{align*}
		As a result, $\dot{r}\leq0$ along the flow of $\gamma$ with equality as long as $\gamma$ lies on $\{r=r_-\}$. Thus $\gamma$ will stay in the block $\M_{2}$ or will cross $\{r=r_{c}\}$, entailing that it can escape to $r=+\infty$ only in another piece of universe.
		\item [2.] and $3.$ follow from $1.$
		\item [4.] Removing all the black/white/worm rings from $\MM_{\textup{ext}}$, it only remains copies of $\M_{3}$, $\M_{3}'$, $\M_{4}$ and $\M_{4}'$ (with the cosmological horizons at $r=r_{+}$). Starting in any of these blocks, the outgoing principal null geodesics $\gamma_{\textup{out}}$ can escape to $r=+\infty$, meaning that no black hole (and thus no white and worm hole) lies there.
	\end{enumerate}
\end{proof}
\begin{remark}
	All the black and white rings contained in the global extended spacetime $\MM_{\textup{ext}}$ are delimited by horizons. In particular, the timelike singularity $\{r=0\}$ lying in all the worm ring copies is always hidden by horizons. The weak Cosmic Censorship is therefore respected.
\end{remark}
%
%
%
%
%
%
%
%
%
%
%
%
%
\section{Analytic scattering theory}
\label{Analytic scattering theory}
This section is devoted to the scattering theory associated to the extended wave equation \eqref{Extended wave equation}. We will show existence and completeness of wave operators associated to several comparison dynamics for fixed momenta $\mathrm{i}\partial_{z}=\z\in\mathbb{Z}\neq\{0\}$ of the scalar field. The case $\z=1$ corresponds to the charged Klein-Gordon equation. Most of the results below follows from the work of Georgescu-G\'{e}rard-H\"{a}fner \cite{GGH17}.

In Subsection \ref{Hamiltonian formulation of the extended wave equation}, we introduce the Hamiltonian formalism associated to equation \eqref{Extended wave equation}. The different comparison dynamics used for the scattering are introduced in Subsection \ref{Comparison dynamics}. The structure of the energy spaces for the comparison dynamics related to transport is analyzed in Subsection \ref{Structure of the energy spaces for the comparison dynamics} in order to prepare the proof of the scattering results. The analytic scattering results are then presented in Subsection \ref{Analyic scattering results} and proved in Subsection \ref{Proof of the analytic results}.
%
%
%
%
\subsection{Hamiltonian formulation of the extended wave equation}
\label{Hamiltonian formulation of the extended wave equation}
We introduce some notations following \cite[Section 2.1]{GGH17}.
%
%
%
%
\subsubsection{The full dynamics}
\label{The full dynamics}
Let us introduce the full dynamics associated to \eqref{Extended wave equation}. First of all, observe that if $u$ solves \eqref{Extended wave equation} then $v:=\mathrm{e}^{-\mathrm{i}sV_{+}t\partial_{z}}u$ satisfies
\begin{align}
\label{Extended wave}
\left(\partial_{t}^{2}-2s\tilde{V}(r)\partial_{z}\partial_{t}+\tilde{P}\right)u&=0,&\tilde{V}(r):=V(r)-V_{+}.
\end{align}
We can therefore work with the potential\footnote{Working with this potential is actually equivalent to use the gauge $A_{\mu}=\left(\frac{Q}{r}-\frac{Q}{r_{+}}\right)\mathrm{d}t$.} $\tilde{V}$ in this Section and use the results in \cite{GGH17}. In order not to overload notations, we will still denote $\tilde{V}$ by $V$; to keep track of the potential $V$, we will often write $V_+$ in the sequel even though it is the constant 0. Note here that we can again compare $sV_{+}$ to the Kerr rotation $\Omega_{+}=\frac{a}{r_+^{2}+a^{2}}$: the unitary transform performed above is the equivalent operation in our setting to the variable change $\varphi\mapsto\varphi-\Omega_{+}t$ at the very beginning of Section 13 in \cite{GGH17}.

We introduce the Regge-Wheeler coordinate
\begin{align*}
\frac{\mathrm{d}x}{\mathrm{d}r}&:=T'(r)=\frac{1}{F(r)}
\end{align*}
so that $x\equiv x(r)$ is equal to the function $T$ up to an additive constant; in particular, $x\in\mathbb{R}$ and $x\to\pm\infty$ if and only if $r\to r_{\pm}$. In the sequel, we will denote by $'$ the derivative with respect to the variable $r$.

Next, let us define $\mathcal{H}:=L^{2}\left(\mathbb{S}^{1}_{z}\times\left]r_{-},r_{+}\right[_{r}\times\mathbb{S}^{2}_{\omega},T'(r)\mathrm{d}z\mathrm{d}r\mathrm{d}\omega\right)=L^{2}\left(\mathbb{S}^{1}_{z}\times\mathbb{R}_{x}\times\mathbb{S}^{2}_{\omega},\mathrm{d}z\mathrm{d}x\mathrm{d}\omega\right)$ and
\begin{align*}
h&:=r\hat{P}r^{-1}\nonumber\\
&=-r^{-1}\partial_{x}r^{2}\partial_{x}r^{-1}-\frac{F(r)}{r^{2}}\Delta_{\mathbb{S}{^2}}-\big(m^{2}F(r)-s^{2}V(r)^{2}\big)\partial_{z}^{2}\\
&=-\partial_{x}^{2}+F(r)\left(-\frac{1}{r^{2}}\Delta_{\mathbb{S}^{2}}+\frac{F'(r)}{r}-m^{2}\partial_{z}^{2}\right)+s^{2}V(r)^{2}\partial_{z}^{2}.
\end{align*}
with $\hat{P}$ given by \eqref{P_z}. Here the variable $r$ has to be understood as a function $r(x)$ of the Regge-Wheeler coordinate. The operator $\partial_{z}$ plays a similar role as $\partial_{\varphi}$ for the De Sitter-Kerr case (\textit{cf.} \cite[equation (13.3)]{GGH17}). Observe here that $h$ is not positive in the dyadorings $\mathcal{D}_{\pm}$ (because $sV-m^{2}F<0$ near $r=r_\pm$). This problem is very similar to the failure of $\partial_{t}$ to be timelike in the extended spacetime (or in Kerr spacetime). Considering the timelike (but not Killing) vector field $\nabla t=F(r)^{-1}(\partial_{t}-sV(r)\partial_{z})$ instead, we add the extra "rotating" term $-sV(r)\partial_{z}$ which cancels the negative parts of $\g(\partial_{t},\partial_{t})$ near $r=r_{\pm}$.

Introduce now cut-offs $i_\pm,j_\pm\in\mathcal{C}^{\infty}(\left]r_-,r_+\right[,\mathbb{R})$ satisfying
\begin{align*}
&\mathrm{Supp\,i_-}=\left]-\infty,1\right],\qquad\mathrm{Supp\,i_+}=\left[-1,+\infty\right[,\\
&i_-^2+i_+^2=1,\qquad i_\pm j_\pm=j_\pm,\qquad i_\pm j_\mp=0.
\end{align*}
They will be useful to separate incoming and outgoing parts of solutions for the scattering. We next introduce the following operators:
\begin{align*}
k&:=-\mathrm{i}sV(r)\partial_{z},\\
k_{\pm}&:=-\mathrm{i}sV_{\pm}\partial_{z},\\
\kk_{\pm}&:=k\mp j_{\mp}^{2}k_{-},\\
h_{0}:=h+k^{2}&=-r^{-1}\partial_{x}r^{2}\partial_{x}r^{-1}-\frac{F(r)}{r^{2}}\Delta_{\mathbb{S}{^2}}-m^{2}F(r)\partial_{z}^{2}\\
&=-\partial_{x}^{2}+F(r)\left(-\frac{1}{r^{2}}\Delta_{\mathbb{S}^{2}}+\frac{F'(r)}{r}-m^{2}\partial_{z}^{2}\right),\\
\hh_{\pm}&:=h_{0}-(\kk-k_{\pm})^{2}.
\end{align*}
Observe that $h_{0}\geq0$ which witnesses of the hyperbolic nature of the equation \eqref{Extended wave equation} (\textit{cf.} \cite[Remark 2.2]{GGH17}). Observe also that if $s$ is sufficiently small (that we will always assume in the sequel), then $\hh_{\pm}\geq0$ to since $\kk-k_{\pm}$ as the same exponential decay at infinity as $F$, \textit{cf.} \eqref{Decay r-r_- with kappa_pm}.

Using the spherical symmetries of the problem, that is, using the diagonalizations $-\Delta_{\mathbb{S}^{2}}=\ell(\ell+1)$ and $-\mathrm{i}\partial_{z}=\z$ with $(\ell,\z)\in\mathbb{N}\times\mathbb{Z}$, we define
\begin{align*}
&h_{0}^{\ell,\z}:=-r^{-1}\partial_{x}r^{2}\partial_{x}r^{-1}-\frac{F(r)}{r^{2}}\ell(\ell+1)+m^{2}F(r)\z^{2}=-\partial_{x}^{2}+F(r)\left(\frac{\ell(\ell+1)}{r^{2}}+\frac{F'(r)}{r}+m^{2}\z^{2}\right),\\[-12mm]
\end{align*}
\begin{align*}
\mathscr{D}(h_{0}^{\ell,\z})&:=\Big\{u\in\HH\,\big\vert\,h_{0}^{\ell,\z}u\in\HH\Big\}.
\end{align*}
In their work \cite{BaMo93}, Bachelot and Motet-Bachelot show in Proposition II.1 that $-\partial_{x}^{2}+\mathcal{V}(x)$ has no 0 eigenvalue if the potential $\mathcal{V}(x)$ has some polynomial decay at infinity; this is the case for $h_{0}^{\ell,\z}$ which is then an elliptic second order differential operator, thus self-adjoint on $\HH$. We realize $h_{0}$ as the direct sum of the harmonic operators $h_{0}^{\ell,\z}$ \textit{i.e.}
\begin{align*}
\mathscr{D}(h_{0})&:=\bigg\{u=\sum_{\ell,\z}u_{\ell,\z}\ \bigg\vert\ u_{\ell,\z}\in\mathscr{D}(h_{0}^{\ell,\z}),\;\sum_{\ell,\z}\big\|h_{0}^{\ell,\z}u_{\ell,\z}\big\|_{\HH}^{2}<+\infty\bigg\}
\end{align*}
which is in turn elliptic and self-adjoint. Finally, as $k\in\mathcal{B}(\HH)$, we can also realize $h$ as a self-adjoint operator on the domain $\mathscr{D}(h)=\mathscr{D}(h_{0})$. In the sequel, we will use the elliptic self-adjoint realizations $(\hh_{\pm},\mathscr{D}(\hh_{\pm}))$ defined as above for $h_{0}$.
%
%
%
%
\subsubsection{Energy spaces for the full dynamics}
\label{Energy spaces for the full dynamics}
We now turn to the definition of the energy spaces associated to the Hamiltonian $\dot{H}$ following \cite{GGH17} (see the paragraphs 3.4.1 and 3.4.2 therein).

We define the \textit{inhomogeneous energy spaces}
\begin{align*}
\mathcal{E}&:=\langle h_{0}\rangle^{-1/2}\HH\times\HH
\end{align*}
equipped with the norm\footnote{We use here that $h_{0}\geq 0$.}
\begin{align*}
\|(u_{0},u_{1})\|_{\mathcal{E}}^{2}&:=\langle(1+h_{0})u_{0},u_{0}\rangle_{\HH}+\|u_{1}-ku_{0}\|_{\HH}^{2}.
\end{align*}
Using that $h_{0}\geq 0$ has no kernel in $\HH$, we can also define the \textit{homogeneous energy space} $\dot{\mathcal{E}}$ as the completion of $\mathcal{E}$ for the norm\footnote{We use here that $h_{0}>0$.}
\begin{align*}
\|(u_{0},u_{1})\|_{\dot{\mathcal{E}}}^{2}&:=\langle h_{0}u_{0},u_{0}\rangle_{\HH}+\|u_{1}-ku_{0}\|_{\HH}^{2}.
\end{align*}
Observe that, as explained in paragraph 3.4.3 of \cite{GGH17}, this energy is not conserved in general. The natural conserved energies $\langle\cdot\!\mid\!\cdot\rangle_{\ell}$ defined for all $\ell\in\mathbb{R}$ by
\begin{align*}
\langle(u_{0},u_{1})\mid(u_{0},u_{1})\rangle_{\ell}&:=\langle(h_{0}-(k-\ell)^{2})u_{0},u_{0}\rangle_{\HH}+\|u_{1}-\ell u_{0}\|_{\HH}^{2}
\end{align*}
are not in general positive (because of the existence of the dyadorings). In contrast, the (positive) energy $\|\cdot\|_{\dot{\mathcal{E}}}$ is not conserved along the flow of $\partial_{t}$ and may grow in time: this is \textit{superradiance}. From the geometric point of view, superradiance occurs because of the existence of the dyadorings and we are using the timelike vector field $\nabla t$ of \eqref{Static observer} instead of $\partial_{t}$ to get a positive quantity near the dyadorings; the cost to pay for this is the non conservation of the energy since $\nabla t$ is not Killing.

Define next the \textit{asymptotic energy spaces}\footnote{Now we use that $\hh_{\pm}>0$.}
\begin{align*}
\dot{\widetilde{\mathcal{E}}}_{\pm}&:=\Phi(\kk_{\pm})\left(\hh_{\pm}^{-1/2}\HH\times\HH\right)
\end{align*}
where
\begin{align*}
\Phi(\kk_{\pm})&:=\begin{pmatrix}
\mathds{1}&0\\\kk_{\pm}&\mathds{1}
\end{pmatrix}.
\end{align*}
The spaces $\dot{\widetilde{\mathcal{E}}}_{\pm}$ are equipped with the norms
\begin{align*}
\|(u_{0},u_{1})\|_{\dot{\widetilde{\mathcal{E}}}_{\pm}}^{2}&:=\langle \hh_{\pm}u_{0},u_{0}\rangle_{\HH}+\|u_{1}-\kk_{\pm}u_{0}\|_{\HH}^{2}.
\end{align*}
As discussed above Lemma 3.13 in \cite{GGH17}, the operators $\Phi(\kk_{\pm}):\hh_{\pm}^{-1/2}\HH\times\HH\to\dot{\widetilde{\mathcal{E}}}_{\pm}$ are isomorphisms with inverses $\Phi(-\kk_{\pm})$.

The Hamiltonian form of \eqref{Extended wave} is given by $-\mathrm{i}\partial_{t}u=\dot{H}u$ with
\begin{align*}
\dot{H}&=\begin{pmatrix}
0&\mathds{1}\\
h&2k
\end{pmatrix},&\mathscr{D}(\dot{H})&=\big\{u\in\dot{\mathcal{E}}\mid\dot{H}u\in\dot{\mathcal{E}}\big\}
\end{align*}
is the \textit{energy Klein-Gordon operator}. We will also need to use the \textit{asymptotic Hamiltonians}
\begin{align*}
\dot{\widetilde{H}}_{\pm}&=\begin{pmatrix}
0&\mathds{1}\\
\hh_{\pm}&2\kk_{\pm}
\end{pmatrix},&\mathscr{D}(\dot{\widetilde{H}}_{\pm})&:=\Phi(k_{\pm})\left(\hh_{\pm}^{-1/2}\HH\cap \hh_{\pm}^{-1}\HH\times\langle \hh_{\pm}\rangle^{-1/2}\HH\right).
\end{align*}
Given now $\z\in\mathbb{Z}$, we define on $\ker(\mathrm{i}\partial_{z}+\z)$ the restricted operators $h^{\z}$, $h_{0}^{\z}$, $k^{\z}$, $\dot{H}^{\z}$, etc. as well as the spaces $\HH^{\z}$, $\mathcal{E}^{\z}$, $\dot{\mathcal{E}}^{\z}$ in the obvious way. The operator $h_{0}^{\z}$ is nothing but a Klein-Gordon operator (we are in the same situation as in \cite{Be18} if $\z\neq0$). Furthermore, if $\z\neq0$, then $\langle h_{0}\rangle^{-1/2}\HH^{\z}=h_{0}^{-1/2}\HH^{\z}=(h_{0}^{\z})^{-1/2}\HH^{\z}$.

\cite[Lemma 2.2]{Ha02} shows\footnote{$\mathcal{H}^{2}$ must be replaced by its completion with respect to the homogeneous norm as pointed out by the author of this paper in \cite{Ha03}, above Lemma 2.1.1.} that $(\dot{\widetilde{H}}\,\!^{\,\z}_{\pm},\mathscr{D}(\dot{\widetilde{H}}\,\!^{\,\z}_{\pm}))$ are self-adjoint. The difference with $(\dot{H}^{\z},\mathscr{D}(\dot{H}^{\z}))$ is that for the asymptotic operators, we have $\big\|\kk_{\pm}^{\z}u\big\|_{\HH^{\z}}\lesssim\big\langle\hh^{\z}_{\pm}u,u\big\rangle_{\HH^{\z}}$ for all $u\in\mathscr{D}\big((\hh_{\pm}^{\z})^{1/2}\big)$ (as required by (2.22) in \cite{Ha02}); such an estimate is false with $k^{\z}$ and $h^{\z}$.

\cite[Lemma 3.19]{GGH17} shows that $\dot{H}$ is the generator of a continuous group $(\mathrm{e}^{\mathrm{i}t\dot{H}})_{t\in\mathbb{R}}$ on $\dot{\mathcal{E}}$. If $u=(u_{0},u_{1})$ solves $-\mathrm{i}\partial_{t}u=\dot{H}u$, then $u_{0}$ is a solution of \eqref{Extended wave equation} and conversely, if $u$ solves \eqref{Extended wave equation} then $v=(u,-\mathrm{i}\partial_{t}u)$ satisfies $-\mathrm{i}\partial_{t}v=\dot{H}v$.
%
%
%
%
\subsection{Comparison dynamics}
\label{Comparison dynamics}
We present in this Subsection the comparison dynamics we will use for the scattering.
%
%
%
%
\subsubsection{Introduction of the dynamics}
\label{Introduction of the dynamics}
\paragraph{Separable comparison dynamics.} The first dynamics we will compare $(\mathrm{e}^{\mathrm{i}t\dot{H}^{\z}})_{t\in\mathbb{R}}$ to are the following \textit{separable comparison dynamics}. Define the operators
\begin{align*}
h_{\pm\infty}&:=-\partial_{x}^{2}-\frac{F(r)}{r^{2}}\Delta_{\mathbb{S}{^2}}-\big(m^{2}F(r)-s^{2}V_{\pm}^{2}\big)\partial_{z}^{2},\qquad\qquad k_{\pm\infty}:=-\mathrm{i}sV_{\pm}\partial_{z}.
\end{align*}
The associated second order equation reads
\begin{align}
\label{Second order pminfty}
\big(\partial_{t}^{2}-2\mathrm{i}k_{\pm\infty}\partial_{t}+h_{\pm\infty}\big)u&=0.
\end{align}
The corresponding Hamiltonian is given by
\begin{align*}
\dot{H}_{\pm\infty}&=\begin{pmatrix}
0&\mathds{1}\\
h_{\pm\infty}&2k_{\pm\infty}
\end{pmatrix},\qquad\qquad\mathscr{D}(\dot{H}_{\pm\infty})=\big\{u\in\dot{\mathcal{E}}_{\pm\infty}\mid\dot{H}_{\pm\infty}u\in\dot{\mathcal{E}}_{\pm\infty}\big\}
\end{align*}
where the spaces $\big(\dot{\mathcal{E}}_{\pm\infty},\|\cdot\|_{\dot{\mathcal{E}}_{\pm\infty}}\big)$ are defined in the canonical way \textit{i.e.} $\dot{\mathcal{E}}_{\pm\infty}$ is the completion of smooth compactly supported functions with respect to the norm
\begin{align*}
\|(u_{0},u_{1})\|_{\dot{\mathcal{E}}_{\pm\infty}}^{2}&:=\big\langle\big(h_{\pm\infty}+k_{\pm\infty}^{2}\big)u_{0},u_{0}\big\rangle_{\HH}+\|u_{1}-k_{\pm\infty}u_{0}\|_{\HH}^{2}.
\end{align*}
Denote by $h^{\z}_{\pm\infty}, k^{\z}_{\pm\infty}, \dot{H}^{\z}_{\pm\infty}$ and $\dot{\mathcal{E}}^{\z}_{\pm\infty}$ the restrictions on $\ker(\mathrm{i}\partial_{z}+\z)$ for $\z\in\mathbb{Z}$ of the above operators and spaces. If $u\in\dot{\mathcal{E}}^{\z}$ solves \eqref{Second order pminfty}, then
\begin{align*}
\frac{1}{2}\frac{\mathrm{d}}{\mathrm{d}t}\|(u,-\mathrm{i}\partial_{t}u)\|_{\dot{\mathcal{E}}_{\pm\infty}^{\z}}^{2}&=\Re\Big(\big\langle\big(h^{\z}_{\pm\infty}+(k_{\pm\infty}^{\z})^{2}\big)u,\partial_{t}u\big\rangle_{\HH}+\big\langle\partial_{t}^{2}u-\mathrm{i}k_{\pm\infty}^{\z}\partial_{t}u,\partial_{t}u-\mathrm{i}k_{\pm\infty}^{\z}u\big\rangle_{\HH}\Big)\\
&=\Re\Big(\big\langle\big(h_{\pm\infty}+(k_{\pm\infty}^{\z})^{2}\big)u,\partial_{t}u\big\rangle_{\HH}+\big\langle\mathrm{i}k_{\pm\infty}^{\z}\partial_{t}u-h_{\pm\infty}u,\partial_{t}u-\mathrm{i}k_{\pm\infty}^{\z}u\big\rangle_{\HH}\Big)\\
&=\Re\Big(\big\langle\mathrm{i}k_{\pm\infty}^{\z}\partial_{t}u,\partial_{t}u\big\rangle_{\HH}+\big\langle h_{\pm\infty}u,\mathrm{i}k_{\pm\infty}^{\z}u\big\rangle_{\HH}\Big)\\
&=0
\end{align*}
because
\begin{align*}
\mathrm{sign}(s\z)\times\big\langle h_{\pm\infty}u,\mathrm{i}k_{\pm\infty}^{\z}u\big\rangle_{\HH}&=\mathrm{i}\big\||k_{\pm\infty}^{\z}|^{1/2}\partial_{x}u\big\|_{\HH}^{2}+\mathrm{i}\big\|r^{-1}F(r)^{1/2}|k_{\pm\infty}^{\z}|^{1/2}\nabla_{\mathbb{S}^{2}}u\big\|_{\HH}^{2}\\
&+\mathrm{i}\big\|mF(r)^{1/2}|k_{\pm\infty}^{\z}|^{1/2}u\big\|_{\HH}^{2}-\mathrm{i}\big\||k_{\pm\infty}^{\z}|^{3/2}\partial_{z}u\big\|_{\HH}^{2}.
\end{align*}
Notice here that the above conservation law strongly relies on the fact that $[\partial_{x},k_{\pm\infty}]=0$ (in comparison, $[\partial_{x},k]\neq 0$). It results that the associated dynamics $(\mathrm{e}^{\mathrm{i}t\dot{H}^{\z}_{\pm\infty}})_{t\in\mathbb{R}}$ are unitary on $\HH^{\z}\times\HH^{\z}$ and the infinitesimal generators $(\dot{H}^{\z}_{\pm\infty},\mathscr{D}(\dot{H}^{\z}_{\pm\infty}))$ are self-adjoint by Stone's theorem.

Theorem \ref{Asymptotic completeness, separable comparison dynamics} states that the operators $\mathrm{e}^{-\mathrm{i}t\dot{H}^{\z}}i_{\pm}\mathrm{e}^{\mathrm{i}t\dot{H}^{\z}_{\pm\infty}}$ and $\mathrm{e}^{-\mathrm{i}t\dot{H}^{\z}_{\pm\infty}}i_{\pm}\mathrm{e}^{\mathrm{i}t\dot{H}^{\z}}$ have strong limits in $\mathcal{B}(\dot{\mathcal{E}}^{\z}_{\pm\infty},\dot{\mathcal{E}}^{\z})$ and $\mathcal{B}(\dot{\mathcal{E}}^{\z},\dot{\mathcal{E}}^{\z}_{\pm\infty})$ as $|t|\to+\infty$.
\paragraph{Asymptotic profiles.} We next introduce the \textit{asymptotic profiles}: they consist in the simplest possible asymptotic comparison dynamics obtained by formally taking the limit $x\to\pm\infty$ in $h$ and $k$. Set
\begin{align*}
h_{-\!/\!+}&:=-\partial_{x}^{2}+s^{2}V_{-\!/\!+}^{2}\partial_{z}^{2},\qquad\qquad k_{-\!/\!+}:=-\mathrm{i}sV_{-\!/\!+}\partial_{z}.
\end{align*}
The associated second order equation reads
\begin{align}
\label{Assymptotic equation}
\big(\partial_{t}^{2}-2\mathrm{i}k_{-\!/\!+}\partial_{t}+h_{-\!/\!+}\big)u&=0.
\end{align}
Notice the following factorization :
\begin{align*}
\partial_{t}^{2}-2sV_{-\!/\!+}\partial_{z}\partial_{t}+h_{-\!/\!+}&=(\partial_{t}-\partial_{x}-sV_{-\!/\!+}\partial_{z})(\partial_{t}+\partial_{x}-sV_{-\!/\!+}\partial_{z}).
\end{align*}
We call \textit{incoming} respectively \textit{outgoing} solutions of \eqref{Assymptotic equation} are the solutions of $(\partial_{t}-\partial_{x}-sV_{-\!/\!+}\partial_{z})u=0$ respectively $(\partial_{t}+\partial_{x}-sV_{-\!/\!+}\partial_{z})u=0$. Define the corresponding Hamiltonians
\begin{align*}
\dot{H}_{-\!/\!+}&=\begin{pmatrix}
0&\mathds{1}\\
h_{-\!/\!+}&2k_{-\!/\!+}
\end{pmatrix},\qquad\qquad\mathscr{D}(\dot{H}_{-\!/\!+})=\big\{u\in\dot{\mathcal{E}}_{-\!/\!+}\mid\dot{H}_{-\!/\!+}u\in\dot{\mathcal{E}}_{-\!/\!+}\big\}
\end{align*}
with the canonical energy spaces $\big(\dot{\mathcal{E}}_{-\!/\!+},\|\cdot\|_{\dot{\mathcal{E}}_{-\!/\!+}}\big)$ and the homogeneous norms
\begin{align*}
\|(u_{0},u_{1})\|_{\dot{\mathcal{E}}_{-\!/\!+}}^{2}&:=\big\langle\big(h_{-\!/\!+}+k_{-\!/\!+}^{2}\big)u_{0},u_{0}\big\rangle_{\HH}+\|u_{1}-k_{-\!/\!+}u_{0}\|_{\HH}^{2}.
\end{align*}
Put $h^{\z}_{-\!/\!+}, k^{\z}_{-\!/\!+}, \dot{H}^{\z}_{-\!/\!+}$ and $\dot{\mathcal{E}}^{\z}_{-\!/\!+}$ for the restrictions of the above operators and spaces on $\ker(\mathrm{i}\partial_{z}+\z)$. Then $(\mathrm{e}^{\mathrm{i}t\dot{H}^{\z}_{-\!/\!+}})_{t\in\mathbb{R}}$ is unitary and $(\dot{H}^{\z}_{-\!/\!+},\mathscr{D}(\dot{H}^{\z}_{-\!/\!+}))$ self-adjoint (the argument is the same as for the separable comparison dynamics above).

Theorem \ref{Asymptotic completeness, asymptotic profiles} below states that there exists a dense subspace $\mathcal{D}^{\textup{fin},\z}_{-\!/\!+}\subset\dot{\mathcal{E}}^{\z}_{-\!/\!+}$ such that $\mathrm{e}^{-\mathrm{i}t\dot{H}^{\z}}i_{-\!/\!+}^{2}\mathrm{e}^{\mathrm{i}t\dot{H}^{\z}_{-\!/\!+}}$ and $\mathrm{e}^{-\mathrm{i}t\dot{H}^{\z}_{-\!/\!+}}i_{-\!/\!+}^{2}\mathrm{e}^{\mathrm{i}t\dot{H}^{\z}}$ have strong limits in $\mathcal{B}(\mathcal{D}^{\textup{fin},\z}_{-\!/\!+},\dot{\mathcal{E}}^{\z})$ and $\mathcal{B}(\dot{\mathcal{E}}^{\z},\dot{\mathcal{E}}^{\z}_{-\!/\!+})$ as $|t|\to+\infty$.
\paragraph{Geometric profiles.}Finally, let us introduce two last profiles we will refer to later on as the \textit{geometric profiles}. We want them to describe a transport along principal null geodesics $\gamma_{\textup{in/out}}$ introduced in Subsection \ref{Principal null geodesics}.

The generators of time-parametrized\footnote{We can parametrize the principal null geodesics by the time variable $t$ as $\frac{\mathrm{d}t}{\mathrm{d}r}=\pm F(r)^{-1}$, \textit{cf.} Subsection \nolinebreak\ref{Principal null geodesics}).} principal null geodesics are given by
\begin{align*}
\boldsymbol{v}_{\textup{in}}&=\partial_{t}+L_{\mathscr{H}},\qquad\qquad L_{\mathscr{H}}:=-\partial_{x}-sV(r)\partial_{z},\\
\boldsymbol{v}_{\textup{out}}&=\partial_{t}+L_{\mathscr{I}},\qquad\qquad L_{\mathscr{I}}:=\partial_{x}-sV(r)\partial_{z}.
\end{align*}
The natural equation one may want to consider to describe transport along $\gamma_{\textup{in/out}}$ would be
\begin{align*}
\frac{1}{2}\big(\boldsymbol{v}_{\textup{in}}\boldsymbol{v}_{\textup{out}}+\boldsymbol{v}_{\textup{out}}\boldsymbol{v}_{\textup{in}}\big)&=0
\end{align*}
However, $[L_{\mathscr{H}},L_{\mathscr{I}}]\neq 0$ (because $[\partial_{x},V]\neq 0$) and no conserved positive energy can be associated to this equation.

The idea is to use instead new dynamics with natural conserved energies whose incoming and outgoing parts are $\boldsymbol{v}_{\textup{in}}$ and $\boldsymbol{v}_{\textup{out}}$. Let us define
\begin{align*}
L_{+}&:=-L_{\mathscr{H}}-2sV_{-}\partial_{z}=\partial_{x}+s(V(r)-2V_{-})\partial_{z},\\
L_{-}&:=-L_{\mathscr{I}}-2sV_{+}\partial_{z}=-\partial_{x}+s(V(r)-2V_{+})\partial_{z}
\end{align*}
and set
\begin{align*}
h_{\mathscr{H}\!/\!\mathscr{I}}&:=-\big(L_{\mathscr{H}\!/\!\mathscr{I}}+sV_{-\!/\!+}\partial_{z}\big)^{2}=-\big(\partial_{x}+\!\!/\!\!-s(V(r)-V_{-\!/\!+})\partial_{z}\big)^{2},\qquad\qquad k_{\mathscr{H}\!/\!\mathscr{I}}:=-\mathrm{i}sV_{-\!/\!+}\partial_{z}.
\end{align*}
The associated second order equation reads
\begin{align}
\label{Assymptotic equation geo}
\big((\partial_{t}-\mathrm{i}k_{\mathscr{H}\!/\!\mathscr{I}})^{2}+h_{\mathscr{H}\!/\!\mathscr{I}}\big)u&=\partial_{t}^{2}u-2\mathrm{i}k_{\mathscr{H}\!/\!\mathscr{I}}\partial_{t}u+L_{\mathscr{H}\!/\!\mathscr{I}}L_{+\!/\!-}u=0
\end{align}
so that we have the factorizations
\begin{align*}
(\partial_{t}-\mathrm{i}k_{\mathscr{H}})^{2}+h_{\mathscr{H}}&=\big(\partial_{t}+L_{\mathscr{H}}\big)\big(\partial_{t}+L_{+}\big),\\
(\partial_{t}-\mathrm{i}k_{\mathscr{I}})^{2}+h_{\mathscr{I}}&=\big(\partial_{t}+L_{\mathscr{I}}\big)\big(\partial_{t}+L_{-}\big)
\end{align*}
with this time $[L_{\mathscr{H}},L_{+}]=[L_{\mathscr{I}},L_{-}]=0$. The \textit{incoming part} $\big(\partial_{t}+L_{\mathscr{H}}\big)$ describes a transport towards $\mathscr{H}$ along $\gamma_{\textup{in}}$ whereas the \textit{outgoing part} $\big(\partial_{t}+L_{\mathscr{I}}\big)$ describes a transport towards $\mathscr{I}$ along $\gamma_{\textup{out}}$. The artificial parts $\big(\partial_{t}+L_{+\!/\!-}\big)$ describe transports towards $\mathscr{I}/\mathscr{H}$ along modified principal null geodesics; they will disappear when we will send data on the horizons later on.

The corresponding Hamiltonians
\begin{align*}
\dot{H}_{\mathscr{H}\!/\!\mathscr{I}}&=\begin{pmatrix}
0&\mathds{1}\\
h_{\mathscr{H}\!/\!\mathscr{I}}-k_{\mathscr{H}\!/\!\mathscr{I}}^{2}&2k_{\mathscr{H}\!/\!\mathscr{I}}
\end{pmatrix},\qquad\qquad\mathscr{D}(\dot{H}_{\mathscr{H}\!/\!\mathscr{I}})=\big\{u\in\dot{\mathcal{E}}_{\mathscr{H}\!/\!\mathscr{I}}\mid\dot{H}_{\mathscr{H}\!/\!\mathscr{I}}u\in\dot{\mathcal{E}}_{\mathscr{H}\!/\!\mathscr{I}}\big\}
\end{align*}
act on their energy spaces $\big(\dot{\mathcal{E}}_{\mathscr{H}\!/\!\mathscr{I}},\|\cdot\|_{\dot{\mathcal{E}}_{\mathscr{H}\!/\!\mathscr{I}}}\big)$ with
\begin{align*}
\|(u_{0},u_{1})\|_{\dot{\mathcal{E}}_{\mathscr{H}\!/\!\mathscr{I}}}^{2}&:=\big\langle h_{\mathscr{H}\!/\!\mathscr{I}}u_{0},u_{0}\big\rangle_{\HH}+\|u_{1}-k_{\mathscr{H}\!/\!\mathscr{I}}u_{0}\|_{\HH}^{2}.
\end{align*}
Notice here that $h_{\mathscr{H}\!/\!\mathscr{I}}$ plays the role of $h_{0}$ for the full dynamics, so that we have to subtract $k_{\mathscr{H}\!/\!\mathscr{I}}^{2}$ in the Hamiltonians $\dot{H}_{\mathscr{H}\!/\!\mathscr{I}}$; besides, we can check that $h_{\mathscr{H}\!/\!\mathscr{I}}>0$ since $\ker_{\HH^{\z}}(h_{\mathscr{H}\!/\!\mathscr{I}})=\{0\}$ for all $\z\in\mathbb{Z}$. By construction, the energies $\|\cdot\|_{\dot{\mathcal{E}}_{\mathscr{H}\!/\!\mathscr{I}}}$ are conserved and the evolutions $(\mathrm{e}^{\mathrm{i}t\dot{H}_{\mathscr{H}\!/\!\mathscr{I}}})_{t\in\mathbb{R}}$ are unitary on $\HH\times\HH$.

Using as above the notation $^{\z}$ for the restriction on $\ker(\mathrm{i}\partial_{z}+\z)$, Theorem \ref{Asymptotic completeness, geometric profiles} below states that that there exists a dense subspace $\mathcal{D}^{\textup{fin},\z}_{\mathscr{H}\!/\!\mathscr{I}}\subset\dot{\mathcal{E}}^{\z}_{\mathscr{H}\!/\!\mathscr{I}}$ such that $\mathrm{e}^{-\mathrm{i}t\dot{H}^{\z}}i_{-\!/\!+}\mathrm{e}^{\mathrm{i}t\dot{H}^{\z}_{\mathscr{H}\!/\!\mathscr{I}}}$ and $\mathrm{e}^{-\mathrm{i}t\dot{H}^{\z}_{\mathscr{H}\!/\!\mathscr{I}}}-_{-\!/\!+}\mathrm{e}^{\mathrm{i}t\dot{H}^{\z}}$ have strong limits in $\mathcal{B}(\mathcal{D}^{\textup{fin},\z}_{\mathscr{H}\!/\!\mathscr{I}},\dot{\mathcal{E}}^{\z})$ and $\mathcal{B}(\dot{\mathcal{E}}^{\z},\dot{\mathcal{E}}^{\z}_{\mathscr{H}\!/\!\mathscr{I}})$. These limits will be given a geometric sense in Subsection \ref{Traces on the energy spaces}.
\begin{remark}
	\begin{enumerate}
		\item We implicitly used \cite[Lemma 3.19]{GGH17} to define all the above dynamics. We refer to Section 3 therein for a detailed discussion about the energy spaces and the Hamiltonian formalism associated to abstract Klein-Gordon equations.
		\item The self-jointedness of the different Hamiltonians (and thus the unitarity of the associated dynamics) have different origins:\\[-6mm]
		\begin{itemize}
			\item [(i)] For the asymptotic Hamiltonians $\dot{\widetilde{H}}\,\!^{\,\z}_{\pm}$, self-adjointness is due to the decay rate $\kk^{\z}_{\pm}=\mathcal{O}_{x\to\pm\infty}\big(\mathrm{e}^{2\kappa_{\pm}x}\big)$ which ensures that $\kk_{\pm}^{\z}\lesssim\hh_{\pm}^{\z}$.
			\item [(ii)] For the comparison Hamiltonians $\dot{H}^{\z}_{\pm\infty}$, $\dot{H}^{\z}_{-\!/\!+}$ and $\dot{H}^{\z}_{\mathscr{H}\!/\!\mathscr{I}}$, self-adjointness is entailed by the conservation of the associated homogeneous energies, which relies on the commutation of $\partial_{x}$ with $k^{\z}_{\pm\infty}=k^{\z}_{-\!/\!+}=k^{\z}_{\mathscr{H}\!/\!\mathscr{I}}$.
		\end{itemize}
		\item Let $\Sigma_{t'}:=\{t=t'\}\times\mathbb{S}^{1}_{z}\times\mathbb{R}_{x}\times\mathbb{S}^{2}_{\omega}$. In all the above statements, we implicitly identify the slices $\Sigma_{0}$ and $\Sigma_{t}$ for all $t\in\mathbb{R}$ using the curves $\{t=\mathrm{constant}\}$. All the energy spaces $\dot{\mathcal{E}},\dot{\mathcal{E}}_{\pm\infty}$ and $\dot{\mathcal{E}}_{-\!/\!+}$ use coordinates $(z,x,\omega)$ in $\Sigma_{0}$.
	\end{enumerate}
\end{remark}
%
%
%
%
%
\subsubsection{Transport along principal null geodesics}
\label{Transport along principal null geodesics}
The purpose of this paragraph is to explicit the action of the Hamiltonians $\dot{H}_{\mathscr{H}\!/\!\mathscr{I}}$. The same considerations hold for $\dot{H}_{-\!/\!+}$, but we omit details here as this case has been treated in \cite{GGH17}, paragraph 13.3.

Let $\dot{\mathcal{H}}_{\mathscr{H}\!/\!\mathscr{I}}^{1}$ be the completions of $\mathcal{C}^{\infty}_{\mathrm{c}}(\Sigma_{0})$ with respect to the norms
\begin{align*}
	\|u\|_{\dot{\HH}^{1}_{\mathscr{H}\!/\!\mathscr{I}}}&:=\big\|(L_{\mathscr{H}\!/\!\mathscr{I}}+\mathrm{i}k_{\mathscr{H}\!/\!\mathscr{I}})u\big\|_{\HH}=\big\|\big(\partial_{x}+\!\!/\!\!-s(V(r)-V_{-\!/\!+})\partial_{z}\big)u\big\|_{\HH}.
\end{align*}
When $s=0$, we have $h_{\mathscr{H}\!/\!\mathscr{I}}=-\partial_{x}^{2}$ and $\dot{\HH}^{1}_{\mathscr{H}\!/\!\mathscr{I}}$ are nothing but the standard homogeneous Sobolev space $\dot{H}^{1}(\mathbb{R},\mathrm{d}x)$ naturally associated to the one-dimensional wave equation. Consider the following canonical transformations:
\begin{align*}
	\Psi_{\mathscr{H}}&:=\frac{1}{\sqrt{2}}\begin{pmatrix}
	\mathds{1}&\mathds{1}\\
	\mathrm{i}L_{\mathscr{H}}&\mathrm{i}L_{+}
	\end{pmatrix}:\dot{\HH}^{1}_{\mathscr{H}}\times\dot{\HH}^{1}_{\mathscr{H}}\longrightarrow\dot{\mathcal{E}}_{\mathscr{H}},\\
	\Psi_{\mathscr{I}}&:=\frac{1}{\sqrt{2}}\begin{pmatrix}
	\mathds{1}&\mathds{1}\\
	\mathrm{i}L_{-}&\mathrm{i}L_{\mathscr{I}}
	\end{pmatrix}:\dot{\HH}^{1}_{\mathscr{I}}\times\dot{\HH}^{1}_{\mathscr{I}}\longrightarrow\dot{\mathcal{E}}_{\mathscr{I}}.
\end{align*}
While the second column of $\Psi_{\mathscr{H}}$ is artificial, the first one is related to the principal null geodesics $\gamma_{\textup{in}}$: for an incoming/outgoing solution $u^{\textup{in}}$, that is $(\partial_{t}+L_{\mathscr{H}})u^{\textup{in}}=0$, we have $\Psi_{\mathscr{H}}(u^{\textup{in}},0)=(u^{\textup{in}},-\mathrm{i}\partial_{t}u^{\textup{in}})$ so $\Psi_{\mathscr{H}}$ prepare initial data for evolution along $\gamma_{\textup{in}}$. Similarly, the second column of $\Psi_{\mathscr{I}}$ is related to transport along $\gamma_{\textup{out}}$.%
%
%
\begin{lemma}
	\label{Transformation Psi}
	The transformations $\Psi_{\mathscr{H}\!/\!\mathscr{I}}$ are invertible isometries with inverses
	\begin{align*}
	\Psi_{\mathscr{H}}^{-1}&=\sqrt{2}\begin{pmatrix}L_{+}&\mathrm{i}\\-L_{\mathscr{H}}&-\mathrm{i}
	\end{pmatrix}(L_{+}-L_{\mathscr{H}})^{-1},\qquad\Psi_{\mathscr{I}}^{-1}=\sqrt{2}\begin{pmatrix}L_{\mathscr{I}}&\mathrm{i}\\-L_{-}&-\mathrm{i}
	\end{pmatrix}(L_{\mathscr{I}}-L_{-})^{-1}.
	\end{align*}
\end{lemma}
\begin{proof}
	We only treat the $\mathscr{H}$ case. Let $(u_{0},u_{1})\in\dot{\HH}^{1}_{\mathscr{H}}\times\dot{\HH}^{1}_{\mathscr{H}}$. Using the parallelogram law $\|a+b\|^{2}+\|a-b\|^{2}=2\|a\|^{2}+2\|b\|^{2}$, we compute:
	\begin{align*}
	\big\|\Psi_{\mathscr{H}}(u_{0},u_{1})\big\|_{\dot{\mathcal{E}}_{\mathscr{H}}}^{2}&=\frac{1}{2}\big\|(L_{\mathscr{H}}+\mathrm{i}k_{\mathscr{H}})(u_{0}+u_{1})\big\|_{\HH}^{2}+\frac{1}{2}\big\|\mathrm{i}(L_{\mathscr{H}}+\mathrm{i}k_{\mathscr{H}})u_{0}+\mathrm{i}(L_{+}+\mathrm{i}k_{\mathscr{H}})u_{1}\big\|_{\HH}^{2}\\[1.25mm]
	&=\frac{1}{2}\big\|(L_{\mathscr{H}}+\mathrm{i}k_{\mathscr{H}})(u_{0}+u_{1})\big\|_{\HH}^{2}+\frac{1}{2}\big\|(L_{\mathscr{H}}+\mathrm{i}k_{\mathscr{H}})(u_{0}-u_{1})\big\|_{\HH}^{2}\\[1.25mm]
	&=\|(u_{0},u_{1})\|_{\dot{\HH}^{1}_{\mathscr{H}}\times\dot{\HH}^{1}_{\mathscr{H}}}^{2}.
	\end{align*}
	Next, the expression for $\Psi_{\mathscr{H}}^{-1}$ makes sense since $L_{+}-L_{\mathscr{H}}=-2(L_{\mathscr{H}}+\mathrm{i}k_{\mathscr{H}})$ have a trivial kernel in $\HH$ and we have
	\begin{align*}
		\big((L_{\mathscr{H}}+\mathrm{i}k_{\mathscr{H}})^{-1}\phi\big)(z,x,\omega)&=-\int_{0}^{x}\mathrm{e}^{-\mathrm{i}s\z\int_{y}^{x}(V(y')-V_{-})\mathrm{d}y'}\phi(z,y,\omega)\mathrm{d}y\qquad\qquad\forall\phi\in\mathcal{C}^{\infty}_{\mathrm{c}}(\Sigma_{0})\cap\ker(\mathrm{i}\partial_{z}+\z).
	\end{align*}
	To check that $\Psi_{\mathscr{H}}^{-1}$ indeed inverts $\Psi_{\mathscr{H}}$, we use that $[L_{\mathscr{H}},L_{+}]=0$.
\end{proof}
\begin{lemma}
	\label{Link Lemma}
	The transformations $\Psi_{\mathscr{H}\!/\!\mathscr{I}}$ diagonalize the Hamiltonians $\dot{H}_{\mathscr{H}\!/\!\mathscr{I}}$: 
	\begin{align*}
	\Psi_{\mathscr{H}}^{-1}\dot{H}_{\mathscr{H}}\Psi_{\mathscr{H}}&=\dot{\mathbb{H}}_{\mathscr{H}}:=\mathrm{i}\begin{pmatrix}
	L_{\mathscr{H}}&0\\0&L_{+}\end{pmatrix}:\Psi_{\mathscr{H}}^{-1}\mathscr{D}(\dot{H}_{\mathscr{H}})\longrightarrow\dot{\HH}^{1}_{\mathscr{H}}\times\dot{\HH}^{1}_{\mathscr{H}},\\
	\Psi_{\mathscr{I}}^{-1}\dot{H}_{\mathscr{I}}\Psi_{\mathscr{I}}&=\dot{\mathbb{H}}_{\mathscr{I}}:=\mathrm{i}\begin{pmatrix}
	L_{\mathscr{I}}&0\\0&L_{-}\end{pmatrix}:\Psi_{\mathscr{I}}^{-1}\mathscr{D}(\dot{H}_{\mathscr{I}})\longrightarrow\dot{\HH}^{1}_{\mathscr{I}}\times\dot{\HH}^{1}_{\mathscr{I}}.
	\end{align*}
	The dynamics $(\mathrm{e}^{\mathrm{i}t\dot{\mathbb{H}}_{\mathscr{H}\!/\!\mathscr{I}}})_{t\in\mathbb{R}}$ have therefore unitary extensions on $\dot{\HH}^{1}_{\mathscr{H}\!/\!\mathscr{I}}\times\dot{\HH}^{1}_{\mathscr{H}\!/\!\mathscr{I}}$.
\end{lemma}
\begin{proof}
	We only treat the $\mathscr{H}$ case. Recall that $L_{+}=-L_{\mathscr{H}}-2\mathrm{i}k_{\mathscr{H}}$ and $h_{\mathscr{H}}-k_{\mathscr{H}}^{2}=L_{\mathscr{H}}L_{\mathscr{+}}$. We compute:
	\begin{align*}
	\dot{H}_{\mathscr{H}}\Psi_{\mathscr{H}}&=\begin{pmatrix}
	0&\mathds{1}\\
	L_{\mathscr{H}}L_{+}&2k_{\mathscr{H}}
	\end{pmatrix}\begin{pmatrix}
	\mathds{1}&\mathds{1}\\
	\mathrm{i}L_{\mathscr{H}}&\mathrm{i}L_{+}
	\end{pmatrix}=\frac{1}{\sqrt{2}}\begin{pmatrix}
	\mathrm{i}L_{\mathscr{H}}&\mathrm{i}L_{+}\\
	-L_{\mathscr{H}}^{2}&-L_{+}^{2}
	\end{pmatrix},\\[1.5mm]
	\Psi_{\mathscr{H}}\dot{\mathbb{H}}_{\mathscr{H}}&=\begin{pmatrix}
	\mathds{1}&\mathds{1}\\
	\mathrm{i}L_{\mathscr{H}}&\mathrm{i}L_{+}
	\end{pmatrix}\begin{pmatrix}
	\mathrm{i}L_{\mathscr{H}}&0\\
	0&\mathrm{i}L_{+}
	\end{pmatrix}=\frac{1}{\sqrt{2}}\begin{pmatrix}
	\mathrm{i}L_{\mathscr{H}}&\mathrm{i}L_{+}\\
	-L_{\mathscr{H}}^{2}&-L_{+}^{2}
	\end{pmatrix}.
	\end{align*}
	This gives the desired formula.
\end{proof}
%
We can now interpret the dynamics $(\mathrm{e}^{\mathrm{i}t\dot{H}_{\mathscr{H}\!/\!\mathscr{I}}})_{t\in\mathbb{R}}$ as mixed transports towards the horizons. Introduce the time-parametrized curves $\gamma_{+}$ and $\gamma_{-}$ defined in the $(z,r,\omega)$ coordinates by
\begin{align}
\label{Principal curves}
\begin{cases}
\dot{\gamma}_{+}(t):=\big(1,s(V(r(t))-2V_{-}),F(r(t)),0\big),\\
\dot{\gamma}_{-}(t):=\big(1,s(V(r(t))-2V_{+}),-F(r(t)),0\big)\\
\end{cases}.
\end{align}
The curve $\gamma_{+}$ carries data to $\mathscr{I}^{+}$ whereas $\gamma_{-}$ carries data to $\mathscr{H}^{+}$. By Lemma \ref{Link Lemma}, we have for all $(\phi_{0},\phi_{1})\in\mathcal{C}^{\infty}_{\mathrm{c}}(\Sigma_{0})\times\mathcal{C}^{\infty}_{\mathrm{c}}(\Sigma_{0})$ and all $t\in\mathbb{R}$:
\begin{align}
\mathrm{e}^{\mathrm{i}t\dot{\mathbb{H}}_{\mathscr{H}\!/\!\mathscr{I}}}\begin{pmatrix}
\phi_{0}\\\phi_{1}
\end{pmatrix}&=\begin{pmatrix}
\mathrm{e}^{-tL_{\mathscr{H}\!/\!\mathscr{I}}}&0\\0&\mathrm{e}^{-tL_{+\!/\!-}}
\end{pmatrix}\begin{pmatrix}
\phi_{0}\\\phi_{1}
\end{pmatrix}=\begin{pmatrix}
\phi_{0}\circ\gamma_{\textup{in/out}}(t)\\\phi_{1}\circ\gamma_{+\!/\!-}(t)
\end{pmatrix},\nonumber\\[2mm]
\label{Evolution on SCS}
\mathrm{e}^{\mathrm{i}t\dot{H}_{\mathscr{H}\!/\!\mathscr{I}}}\begin{pmatrix}
\phi_{0}\\\phi_{1}
\end{pmatrix}&=\frac{1}{2}\begin{pmatrix}
-L_{+\!/\!-}\widetilde{\phi}_{0}-\mathrm{i}\widetilde{\phi}_{1}\\\mathrm{i}L_{\mathscr{H}\!/\!\mathscr{I}}\big(-L_{+\!/\!-}\widetilde{\phi}_{0}-\mathrm{i}\widetilde{\phi}_{1}\big)
\end{pmatrix}\circ\gamma_{\textup{in/out}}(t)+\frac{1}{2}\begin{pmatrix}
L_{\mathscr{H}\!/\!\mathscr{I}}\widetilde{\phi}_{0}+\mathrm{i}\widetilde{\phi}_{1}\\\mathrm{i}L_{+\!/\!-}\big(L_{\mathscr{H}\!/\!\mathscr{I}}\widetilde{\phi}_{0}+\mathrm{i}\widetilde{\phi}_{1}\big)
\end{pmatrix}\circ\gamma_{+\!/\!-}(t)\nonumber\\
&=\frac{1}{2}\begin{pmatrix}
\phi_{0}-\mathrm{i}\big(\widetilde{\phi}_{1}-k_{\mathscr{H}/\!\mathscr{I}}\widetilde{\phi}_{0}\big)\\\mathrm{i}L_{\mathscr{H}\!/\!\mathscr{I}}\Big[\phi_{0}-\mathrm{i}\big(\widetilde{\phi}_{1}-k_{\mathscr{H}/\!\mathscr{I}}\widetilde{\phi}_{0}\big)\Big]
\end{pmatrix}\circ\gamma_{\textup{in/out}}(t)\nonumber\\
&\quad+\frac{1}{2}\begin{pmatrix}
\phi_{0}+\mathrm{i}\big(\widetilde{\phi}_{1}-k_{\mathscr{H}/\!\mathscr{I}}\widetilde{\phi}_{0}\big)\\\mathrm{i}L_{+\!/\!-}\Big[\phi_{0}+\mathrm{i}\big(\widetilde{\phi}_{1}-k_{\mathscr{H}/\!\mathscr{I}}\widetilde{\phi}_{0}\big)\Big]
\end{pmatrix}\circ\gamma_{+\!/\!-}(t)
\end{align}
where $\widetilde{\phi}_{j}:=(L_{\mathscr{H}\!/\!\mathscr{I}}+\mathrm{i}k_{\mathscr{H}\!/\!\mathscr{I}})^{-1}\phi_{j}$. At $t=0$, we get $(\phi_{0},\phi_{1})$ back on the right-hand side above.
%
%
%
%
\subsection{Structure of the energy spaces for the comparison dynamics}
\label{Structure of the energy spaces for the comparison dynamics}
We analyze in this Subsection the structure of the energy spaces associated to the asymptotic and geometric comparison dynamics introduced in Subsection \ref{Comparison dynamics}. We will obtain explicit representation formulas on dense subspaces in smooth compactly supported functions. This will help us to show existence and completeness of the wave operators in Theorem \ref{Asymptotic completeness, asymptotic profiles} and Theorem \ref{Asymptotic completeness, geometric profiles}. In all this Section, we will restrict ourselves to $\ker(\mathrm{i}\partial_{z}+\z)$ with $\z\in\mathbb{Z}$.
%
%
%
%
\subsubsection{Structure of the energy spaces for the asymptotic profiles}
\label{Structure of the energy spaces for the asymptotic profiles}
The solutions of the initial value problem associated to \eqref{Assymptotic equation}
\begin{align*}
\begin{cases}
\big(\partial_{t}^{2}-2\mathrm{i}k_{-\!/\!+}^{\z}\partial_{t}+h_{-\!/\!+}^{\z}\big)u=0\\
u_{\vert\Sigma_{0}}=u_{0}\in\mathcal{C}^{\infty}_{\mathrm{c}}(\Sigma_{0})\\
(-\mathrm{i}\partial_{t}u)_{\vert\Sigma_{0}}=u_{1}\in\mathcal{C}^{\infty}_{\mathrm{c}}(\Sigma_{0})
\end{cases}
\end{align*}
are given by the Kirchhoff type formula:
\begin{align}
\label{Kirchhoff type formula 1}
u(t,z,x,\omega)&=\frac{\mathrm{e}^{\mathrm{i}s\z tV_{-\!/\!+}}}{2}\left(u_{0}(z,x+t,\omega)+\mathrm{i}\int_{0}^{x+t}\big(u_{1}-sV_{-\!/\!+}\z u_{0}\big)(z,y,\omega)\right)\nonumber\\
&+\frac{\mathrm{e}^{\mathrm{i}s\z tV_{-\!/\!+}}}{2}\left(u_{0}(z,x-t,\omega)+\mathrm{i}\int_{x-t}^{0}\big(u_{1}-sV_{-\!/\!+}\z u_{0}\big)(z,y,\omega)\right).
\end{align}
They are a linear combination of incoming and outgoing solutions of \eqref{Assymptotic equation}.

The simplicity of the asymptotic profiles has a cost: all the angular information has been lost in the construction of $\dot{H}^{\z}_{-\!/\!+}$. As a result, as explained in \cite{GGH17}, there is no chance that the limits
\begin{align*}
\boldsymbol{W}_{l/r}u&=\lim_{|t|\to+\infty}\mathrm{e}^{-\mathrm{i}t\dot{H}^{\z}}i_{-\!/\!+}\mathrm{e}^{\mathrm{i}t\dot{H}^{\z}_{-\!/\!+}}u
\end{align*}
exist for all $u\in\dot{\mathcal{E}}^{\z}_{-\!/\!+}$ since $\dot{H}^{\z}$ is built from $h^{\z}$ which acts on the angular part of $u$. The strategy is to define the limits first on a suitable dense subspace then to extend the corresponding operator by continuity on the whole space $\dot{\mathcal{E}}^{\z}_{-\!/\!+}$. We define the spaces
\begin{align*}
\mathcal{E}^{\textup{fin},\z}_{l/r}&:=\left\{u\in\mathcal{E}^{\z}_{l/r}\ \bigg\vert\ \exists \ell_{0}>0\;;\;u\in\big(L^{2}(\mathbb{S}^{1}\times\mathbb{R},\mathrm{d}z\mathrm{d}x)\otimes\bigoplus_{\ell\leq\ell_{0}}Y_{\ell}\big)\times\big(L^{2}(\mathbb{S}^{1}\times\mathbb{R},\mathrm{d}z\mathrm{d}x)\otimes\bigoplus_{\ell\leq\ell_{0}}Y_{\ell}\big)\right\}
\end{align*}
where $Y_{\ell}$ is the eigenspace associated to the eigenvalue $\ell(\ell+1)$ of the self-adjoint realization $(-\Delta_{\mathbb{S}^{2}},H^{2}(\mathbb{S}^{2},\mathrm{d}\omega))$. In order to exploit \eqref{Kirchhoff type formula 1} and decompose elements of $\mathcal{E}^{\textup{fin},\z}_{l/r}$ into incoming and outgoing solutions of \eqref{Assymptotic equation}, we will use the following spaces:
\begin{align*}
\dot{\mathcal{E}}_{-\!/\!+}^{L,\z}&:=\bigg\{(u_{0},u_{1})\in\dot{\mathcal{E}}_{-\!/\!+}\ \Big\vert\ u_{1}-sV_{-\!/\!+}\z u_{0}\in L^{1}\big(\mathbb{R},\mathrm{d}x\,;(\mathbb{S}^{1}\times\mathbb{S}^{2},\mathrm{d}z\mathrm{d}\omega)\big),\\
&\qquad\qquad\qquad\qquad\quad\int_{\mathbb{R}}\big(u_{1}-sV_{-\!/\!+}\z u_{0}\big)(z,x,\omega)\mathrm{d}x=0\text{ a.e. in $z,\omega$}\bigg\}.
\end{align*}
\cite[Lemma 13.3]{GGH17} shows that\footnote{Observe the different gauge used therein (the initial time-derivative is $\partial_{t}u$). To obtain the incoming/outgoing spaces of \cite{GGH17}, we need to replace $u_{1}$ by $\mathrm{i}u_{1}$ below; this of course does not modify the results.}
\begin{align}
\label{In/out decomposition 1}
\dot{\mathcal{E}}_{-\!/\!+}^{L,\z}&=\dot{\mathcal{E}}_{-\!/\!+}^{\textup{in},\z}\oplus\dot{\mathcal{E}}_{-\!/\!+}^{\textup{out},\z}
\end{align}
with the spaces of incoming and outgoing initial data
\begin{align*}
\dot{\mathcal{E}}_{-\!/\!+}^{\textup{in},\z}&=\Big\{(u_{0},u_{1})\in\dot{\mathcal{E}}_{-\!/\!+}^{L,\z}\ \big\vert\,u_{1}=-\mathrm{i}\partial_{x}u_{0}-sV_{-\!/\!+}\z u_{0}\Big\},\\
\dot{\mathcal{E}}_{-\!/\!+}^{\textup{out},\z}&=\Big\{(u_{0},u_{1})\in\dot{\mathcal{E}}_{-\!/\!+}^{L,\z}\ \big\vert\,u_{1}=\mathrm{i}\partial_{x}u_{0}-sV_{-\!/\!+}\z u_{0}\Big\}.
\end{align*}
Solutions in $\dot{\mathcal{E}}_{-\!/\!+}^{L,\z}$ verify a Huygens principle (\textit{cf.} \cite[Remark 13.4]{GGH17}). The cancellation of the integral in particular removes the resonance at 0 of the wave equation; the non-vanishing term as $t\to\pm\infty$ is the projection of compactly supported data on the resonant state which is nothing but the constant solution $1\not\in\HH$ (see the De Sitter-Schwarzschild case for $s=0$ in \cite{BoHa08}, Proposition 1.2 and Theorem 1.3; here the resonant state is $r\otimes\omega_{0}$ with $\omega_{0}\in Y_{0}$ the fundamental spherical harmonic because we use the conjugated spatial operator $r\widehat{P}r^{-1}$).

Finally, set
\begin{align*}
\mathcal{D}_{l/r}^{\textup{fin},\z}&:=\big(\mathcal{C}^{\infty}_{\mathrm{c}}(\mathbb{S}^{1}\times\mathbb{R}\times\mathbb{S}^{2})\times\mathcal{C}^{\infty}_{\mathrm{c}}(\mathbb{S}^{1}\times\mathbb{R}\times\mathbb{S}^{2})\big)\cap\mathcal{E}_{l/r}^{\textup{fin},\z}\cap\dot{\mathcal{E}}_{-\!/\!+}^{L}.
\end{align*}
Then \cite[Lemma 13.5]{GGH17} shows that $\mathcal{D}_{l/r}^{\textup{fin},\z}$ is dense in $\big(\mathcal{E}_{l/r}^{\textup{fin},\z},\|\cdot\|_{\dot{\mathcal{E}}_{-\!/\!+}^{\z}}\big)$ (and thus in \linebreak$\big(\dot{\mathcal{E}}_{-\!/\!+}^{\z},\|\cdot\|_{\dot{\mathcal{E}}_{-\!/\!+}^{\z}}\big)$). We will show similar results for the geometric profiles in the next paragraph.
%
%
%
%
\subsubsection{Structure of the energy spaces for the geometric profiles}
\label{Structure of the energy spaces for the geometric profiles}
The solutions of the initial value problem associated to \eqref{Assymptotic equation geo}
\begin{align*}
\begin{cases}
\big(\partial_{t}^{2}-2\mathrm{i}k_{\mathscr{H}\!/\!\mathscr{I}}^{\z}\partial_{t}+L_{\mathscr{H}\!/\!\mathscr{I}}^{\z}L_{+\!/\!-}^{\z}\big)u=0\\
u_{\vert\Sigma_{0}}=u_{0}\in\mathcal{C}^{\infty}_{\mathrm{c}}(\Sigma_{0})\\
(-\mathrm{i}\partial_{t}u)_{\vert\Sigma_{0}}=u_{1}\in\mathcal{C}^{\infty}_{\mathrm{c}}(\Sigma_{0})
\end{cases}
\end{align*}
are given by the Kirchhoff type formula (we drop the dependence in $(z,\omega)\in\mathbb{S}^{1}\times\mathbb{S}^{2}$):
\begin{align}
\label{Kirchhoff type formula 2}
u_{\mathscr{H}}(t,x)&=\frac{1}{2}\mathrm{e}^{\mathrm{i}s\z\int_{x}^{x+t}V(x')\mathrm{d}x'}\left(u_{0}(x+t)+\mathrm{i}\int_{0}^{x+t}\mathrm{e}^{-\mathrm{i}s\z\int_{y}^{x+t}(V(y')-V_{-})\mathrm{d}y'}\big(u_{1}-s\z V_{-}u_{0}\big)(y)\mathrm{d}y\right)\nonumber\\
&\quad+\frac{1}{2}\mathrm{e}^{\mathrm{i}s\z\int_{x}^{x-t}(V(x')-2V_{-})\mathrm{d}x'}\left(u_{0}(x-t)-\mathrm{i}\int_{0}^{x-t}\mathrm{e}^{-\mathrm{i}s\z\int_{y}^{x-t}(V(y')-V_{-})\mathrm{d}y'}\big(u_{1}-s\z V_{-}u_{0}\big)(y)\mathrm{d}y\right),
\end{align}
\textcolor{white}{0}\\[-12mm]
\begin{align}
\label{Kirchhoff type formula 2 bis}
u_{\mathscr{I}}(t,x)&=\frac{1}{2}\mathrm{e}^{-\mathrm{i}s\z\int_{x}^{x+t}(V(x')-2V_{+})\mathrm{d}x'}\left(u_{0}(x+t)+\mathrm{i}\int_{0}^{x+t}\mathrm{e}^{\mathrm{i}s\z\int_{y}^{x+t}(V(y')-V_{+})\mathrm{d}y'}\big(u_{1}-s\z V_{+}u_{0}\big)(y)\mathrm{d}y\right)\nonumber\\
&\quad+\frac{1}{2}\mathrm{e}^{-\mathrm{i}s\z\int_{x}^{x-t}V(x')\mathrm{d}x'}\left(u_{0}(x-t)-\mathrm{i}\int_{0}^{x-t}\mathrm{e}^{\mathrm{i}s\z\int_{y}^{x-t}(V(y')-V_{+})\mathrm{d}y'}\big(u_{1}-s\z V_{+}u_{0}\big)(y)\mathrm{d}y\right).
\end{align}
The solution $u_{\mathscr{H}}$ is a linear combination of an incoming solution transported along the principal null geodesic $\gamma_{\textup{in}}$ and and outgoing solution transported along the artificial curve $\gamma_{+}$; the solution $u_{\mathscr{I}}$ is a linear combination of an incoming solution transported along the artificial curve $\gamma_{-}$ and and outgoing solution transported along the principal null geodesic $\gamma_{\textup{out}}$. Formulas \eqref{Kirchhoff type formula 2} and \eqref{Kirchhoff type formula 2 bis} display the first components of $\mathrm{e}^{\mathrm{i}t\dot{H}^{\z}_{\mathscr{H}\!/\!\mathscr{I}}}u$: we can easily check them using \eqref{Evolution on SCS} with (we omit the dependence in $(z,\omega)\in\mathbb{S}^{1}\times\mathbb{S}^{2}$)
\begin{align*}
\big((L_{\mathscr{H}\!/\!\mathscr{I}}^{\z}+\mathrm{i}k_{\mathscr{H}\!/\!\mathscr{I}}^{\z})^{-1}(u_{1}-k^{\z}_{\mathscr{H}\!/\!\mathscr{I}}u_{0})\big)(x)&=-\!/\!\!+\int_{0}^{x}\mathrm{e}^{-\!/\!+\mathrm{i}s\z\int_{y}^{x}(V(y')-V_{-\!/\!+})\mathrm{d}y'}\big(u_{1}-s\z V_{-\!/\!+}u_{0}\big)(y)\mathrm{d}y.
\end{align*}
It will be useful below using the following simplified forms:
\begin{align}
\label{Kirchhoff type formula 2 simplified}
u_{\mathscr{H}}(t,x)&=\frac{\mathrm{e}^{\mathrm{i}s\z tV_{-}}}{2}\left(\sum_{\pm}\mathrm{e}^{\mathrm{i}s\z\int_{x}^{x\pm t}(V(x')-V_{-})\mathrm{d}x'}u_{0}(x\pm t)+\mathrm{i}\int_{x-t}^{x+t}\mathrm{e}^{\mathrm{i}s\z\int_{x}^{y}(V(x')-V_{-})\mathrm{d}x'}\big(u_{1}-s\z V_{-}u_{0}\big)(y)\mathrm{d}y\right),
\end{align}
\textcolor{white}{0}\\[-12mm]
\begin{align}
\label{Kirchhoff type formula 2 bis simplified}
u_{\mathscr{I}}(t,x)&=\frac{\mathrm{e}^{\mathrm{i}s\z tV_{+}}}{2}\left(\sum_{\pm}\mathrm{e}^{-\mathrm{i}s\z\int_{x}^{x\pm t}(V(x')-V_{+})\mathrm{d}x'}u_{0}(x\pm t)+\mathrm{i}\int_{x-t}^{x+t}\mathrm{e}^{-\mathrm{i}s\z\int_{x}^{y}(V(x')-V_{+})\mathrm{d}x'}\big(u_{1}-s\z V_{+}u_{0}\big)(y)\mathrm{d}y\right).
\end{align}

As for the asymptotic profiles of the paragraph \ref{Structure of the energy spaces for the asymptotic profiles}, we need to control the angular directions on $\mathbb{S}^{2}_{\omega}$ if we wish to compare the geometric dynamics with the full one. We thus define
\begin{align*}
\mathcal{E}^{\textup{fin},\z}_{\mathscr{H}\!/\!\mathscr{I}}&:=\left\{u\in\mathcal{E}^{\z}_{\mathscr{H}\!/\!\mathscr{I}}\ \bigg\vert\ \exists \ell_{0}>0\;;\;u\in\big(L^{2}(\mathbb{S}^{1}\times\mathbb{R},\mathrm{d}z\mathrm{d}x)\otimes\bigoplus_{\ell\leq\ell_{0}}Y_{\ell}\big)\times\big(L^{2}(\mathbb{S}^{1}\times\mathbb{R},\mathrm{d}z\mathrm{d}x)\otimes\bigoplus_{\ell\leq\ell_{0}}Y_{\ell}\big)\right\}.
\end{align*}
Formulas \eqref{Kirchhoff type formula 2} and \eqref{Kirchhoff type formula 2 bis} makes \textit{a priori} no sense in $\mathcal{E}^{\textup{fin},\z}_{\mathscr{H}\!/\!\mathscr{I}}$ and the integral terms are not controlled in $\HH$. We thus introduce the following spaces:
\begin{align*}
\dot{\mathcal{E}}_{\mathscr{H}\!/\!\mathscr{I}}^{L,\z}&:=\bigg\{(u_{0},u_{1})\in\dot{\mathcal{E}}_{\mathscr{H}\!/\!\mathscr{I}}\ \Big\vert\ u_{1}-s\z V_{-\!/\!+}u_{0}\in L^{1}\big(\mathbb{R},\mathrm{d}y\,;(\mathbb{S}^{1}\times\mathbb{S}^{2},\mathrm{d}z\mathrm{d}\omega)\big),\\
&\qquad\qquad\qquad\qquad\qquad\!\int_{\mathbb{R}}\mathrm{e}^{+\!/\!-\mathrm{i}s\z\int_{0}^{y}(V(x')-V_{-\!/\!+})\mathrm{d}x'}\big(u_{1}-s\z V_{-\!/\!+}u_{0}\big)(z,y,\omega)\mathrm{d}y=0\text{ a.e. in $z,\omega$}\bigg\}.
\end{align*}
We now establish a result similar to \cite[Lemma 13.3]{GGH17}: it gives a deeper meaning to incoming and outgoing data.
\begin{lemma}
	\label{Decomposition in/out H/I}
	Let
	\begin{align*}
	\dot{\mathcal{E}}_{\mathscr{H}}^{\textup{in},\z}&=\Big\{(u_{0},u_{1})\in\dot{\mathcal{E}}_{\mathscr{H}}^{L,\z}\ \Big\vert\,u_{1}=\mathrm{i}L_{\mathscr{H}}u_{0}\Big\},\qquad\qquad\dot{\mathcal{E}}_{\mathscr{H}}^{\textup{out},\z}=\Big\{(u_{0},u_{1})\in\dot{\mathcal{E}}_{\mathscr{H}}^{L,\z}\ \Big\vert\,u_{1}=\mathrm{i}L_{+}u_{0}\Big\},\\[1.5mm]
	\dot{\mathcal{E}}_{\mathscr{I}}^{\textup{in},\z}&=\Big\{(u_{0},u_{1})\in\dot{\mathcal{E}}_{\mathscr{I}}^{L,\z}\ \Big\vert\,u_{1}=\mathrm{i}L_{-}u_{0}\Big\},\qquad\qquad\dot{\mathcal{E}}_{\mathscr{I}}^{\textup{out},\z}=\Big\{(u_{0},u_{1})\in\dot{\mathcal{E}}_{\mathscr{I}}^{L,\z}\ \Big\vert\,u_{1}=\mathrm{i}L_{\mathscr{I}}u_{0}\Big\}.
	\end{align*}
	The following decompositions into incoming and outgoing solutions of \eqref{Assymptotic equation geo} hold:
	\begin{align*}
	\dot{\mathcal{E}}^{L,\z}_{\mathscr{H}\!/\!\mathscr{I}}&=\dot{\mathcal{E}}_{\mathscr{H}\!/\!\mathscr{I}}^{\textup{in},\z}\oplus\dot{\mathcal{E}}_{\mathscr{H}\!/\!\mathscr{I}}^{\textup{out},\z}.
	\end{align*}
	Furthermore, if $u=u^{\textup{in}}+u^{\textup{out}}\in\dot{\mathcal{E}}_{\mathscr{H}\!/\!\mathscr{I}}^{L,\z}$ is supported in $\mathbb{S}^{1}_{z}\times[R_{1},R_{2}]_{x}\times\mathbb{S}^{2}_{\omega}$ for some $R_{1},R_{2}\in\mathbb{R}$, then
	\begin{align}
	\label{Huygens}
	\mathrm{Supp\,}u^{\textup{in}}&\subset\mathbb{S}^{1}_{z}\times\left]-\infty,R_{2}\right]_{x}\times\mathbb{S}^{2}_{\omega},\qquad\qquad\mathrm{Supp\,}u^{\textup{out}}\subset\mathbb{S}^{1}_{z}\times\left[R_{1},+\infty\right[_{x}\times\mathbb{S}^{2}_{\omega}.
	\end{align}
\end{lemma}
\begin{proof}
	We only show the lemma for $\dot{\mathcal{E}}^{L,\z}_{\mathscr{H}}$. Recall that
	\begin{align*}
		\mathrm{i}L_{\mathscr{H}}u_{0}&=-\mathrm{i}\partial_{x}u_{0}+s\z Vu_{0},\qquad\qquad\mathrm{i}L_{+}u_{0}=\mathrm{i}\partial_{x}u_{0}-s\z(V-2V_{-})u_{0}.
	\end{align*}
	Let $u=(u_{0},u_{1})\in\dot{\mathcal{E}}_{\mathscr{H}}^{L,\z}$ and put $\tilde{u}(x):=\mathrm{e}^{\mathrm{i}s\z\int_{0}^{x}(V(y)-V_{-})\mathrm{d}y}u(x)$. Then $u\in\dot{\mathcal{E}}_{\mathscr{H}}^{\textup{in},\z}$ if and only if
	\begin{align*}
	\tilde{u}_{1}&=-\mathrm{i}\partial_{x}\tilde{u}_{0}+s\z V_{-}\tilde{u}_{0}
	\end{align*}
	whereas $u\in\dot{\mathcal{E}}_{\mathscr{H}}^{\textup{out},\z}$ if and only if
	\begin{align*}
	\tilde{u}_{1}&=\mathrm{i}\partial_{x}\tilde{u}_{0}+s\z V_{-}\tilde{u}_{0}.
	\end{align*}
	These conditions define incoming and outgoing states for the asymptotic profiles (see \cite{GGH17}, above Lemma 13.3 with $u_{1}$ therein being $\mathrm{i}u_{1}$ for us). We then define\footnote{We use formula (13.8) of \cite[Lemma 13.3]{GGH17}.} (omitting the dependence in $(z,\omega)\in\mathbb{S}^{1}\times\mathbb{S}^{2}$):
	\begin{align*}
	\tilde{u}_{0}^{\textup{in}}(x)&=\frac{1}{2}\int_{x}^{+\infty}\big[\!-\!\partial_{y}\tilde{u}_{0}-\mathrm{i}(\tilde{u}_{1}-s\z V_{-}\tilde{u}_{0})\big](y)\mathrm{d}y,\\[2mm]
	\tilde{u}_{1}^{\textup{in}}(x)&=\frac{1}{2}\big[\!-\!\mathrm{i}\partial_{x}\tilde{u}_{0}+(\tilde{u}_{1}-s\z V_{-}\tilde{u}_{0})\big](x)+\frac{s\z V_{-}}{2}\int_{x}^{+\infty}\big[\!-\!\partial_{y}\tilde{u}_{0}-\mathrm{i}(\tilde{u}_{1}-s\z V_{-}\tilde{u}_{0})\big](y)\mathrm{d}y,\\[5mm]
	%
	%
	\tilde{u}_{0}^{\textup{out}}(x)&=\frac{1}{2}\int_{-\infty}^{x}\big[\partial_{y}\tilde{u}_{0}-\mathrm{i}(\tilde{u}_{1}-s\z V_{-}\tilde{u}_{0})\big](y)\mathrm{d}y,\\[2mm]
	\tilde{u}_{1}^{\textup{out}}(x)&=\frac{1}{2}\big[\mathrm{i}\partial_{x}\tilde{u}_{0}+(\tilde{u}_{1}-s\z V_{-}\tilde{u}_{0})\big](x)+\frac{s\z V_{-}}{2}\int_{-\infty}^{x}\big[\partial_{y}\tilde{u}_{0}-\mathrm{i}(\tilde{u}_{1}-s\z V_{-} \tilde{u}_{0})\big](y)\mathrm{d}y.
	\end{align*}
	Direct computations show that
	\begin{align*}
	\tilde{u}_{1}^{\textup{in}}&=-\mathrm{i}\partial_{x}\tilde{u}_{0}^{\textup{in}}+s\z V_{-}\tilde{u}_{0}^{\textup{in}},\quad\quad\tilde{u}_{1}^{\textup{out}}=\mathrm{i}\partial_{x}\tilde{u}_{0}^{\textup{out}}+s\z V_{-}\tilde{u}_{0}^{\textup{out}},\quad\quad(\tilde{u}_{0},\tilde{u}_{1})=(\tilde{u}_{0}^{\textup{in}}+\tilde{u}_{0}^{\textup{out}},\tilde{u}_{1}^{\textup{in}}+\tilde{u}_{1}^{\textup{out}}).
	\end{align*}
	Here we use that
	\begin{align*}
	\int_{-\infty}^{+\infty}(\tilde{u}_{1}-s\z V_{-}\tilde{u}_{0})(y)\mathrm{d}y&=\int_{-\infty}^{+\infty}\mathrm{e}^{\mathrm{i}s\z\int_{0}^{y}(V(y')-V_{-})\mathrm{d}y'}\big(u_{1}-s\z V_{-} u_{0}\big)(y)\mathrm{d}x=0.
	\end{align*}
	It follows that
	\begin{align*}
	u^{\textup{in}/\textup{out}}&:=\mathrm{e}^{-\mathrm{i}s\z\int_{0}^{x}(V(y)-V_{-})\mathrm{d}y}\big(\tilde{u}_{0}^{\textup{in}/\textup{out}},\tilde{u}_{1}^{\textup{in}/\textup{out}}\big)
	\end{align*}
	satisfy $u^{\textup{in}/\textup{out}}\in\dot{\mathcal{E}}_{\mathscr{H}}^{\textup{in}/\textup{out},\z}$ and $u=u^{\textup{in}}+u^{\textup{out}}$.
	
	The support condition \eqref{Huygens} for $\tilde{u}^{\textup{in}/\textup{out}}$ directly reads on the above expressions and is in turn verified for $u^{\textup{in}/\textup{out}}$ as multiplication by $\mathrm{e}^{\mathrm{i}s\z\int_{0}^{x}(V(y)-V_{-})\mathrm{d}y}$ does not modify supports.
	
	It remains to show that $\dot{\mathcal{E}}_{\mathscr{H}}^{\textup{in},\z}\cap\dot{\mathcal{E}}_{\mathscr{H}}^{\textup{out},\z}=\{0\}$. If $u$ lies in the intersection, then $\tilde{u}$ satisfies
	\begin{align*}
	\begin{cases}
	\tilde{u}_{1}~=~-\mathrm{i}\partial_{x}\tilde{u}_{0}+s\z V_{-}\tilde{u}_{0}\\
	\tilde{u}_{1}~=~\mathrm{i}\partial_{x}\tilde{u}_{0}+s\z V_{-}\tilde{u}_{0}
	\end{cases}.
	\end{align*}
	Adding and subtracting both the conditions, we see that $\tilde{u}_{0}$ is constant and $\tilde{u}_{1}=s\z V_{-}\tilde{u}_{0}$ in $\HH$, whence $\tilde{u}_{0}=\tilde{u}_{1}=0$. This entails $u=0$ and we are done.
\end{proof}
\begin{remark}
\label{Psi in out}
	Using the definitions of $\Psi_{\mathscr{H}\!/\!\mathscr{I}}$ in the paragraph \ref{Transport along principal null geodesics}, we easily show that
	\begin{align*}
		\dot{\mathcal{E}}^{\textup{in},\z}_{\mathscr{H}}&=\Psi_{\mathscr{H}}\big(\dot{\mathcal{H}}^{1,\z}_{\mathscr{H}}\times\{0\}\big)\cap\dot{\mathcal{E}}_{\mathscr{H}}^{L,\z},\qquad\qquad\dot{\mathcal{E}}^{\textup{out},\z}_{\mathscr{I}}=\Psi_{\mathscr{I}}\big(\{0\}\times\dot{\mathcal{H}}^{1,\z}_{\mathscr{I}}\big)\cap\dot{\mathcal{E}}_{\mathscr{I}}^{L,\z}.
	\end{align*}
\end{remark}
We then have the following adaptation of \cite[Lemma 13.5]{GGH17} to our framework:
\begin{lemma}
	\label{Density in/out}
	Let
	\begin{align*}
	\mathcal{D}_{\mathscr{H}\!/\!\mathscr{I}}^{\textup{fin},\z}&:=\big(\mathcal{C}^{\infty}_{\mathrm{c}}(\mathbb{S}^{1}\times\mathbb{R}\times\mathbb{S}^{2})\times\mathcal{C}^{\infty}_{\mathrm{c}}(\mathbb{S}^{1}\times\mathbb{R}\times\mathbb{S}^{2})\big)\cap\mathcal{E}_{\mathscr{H}\!/\!\mathscr{I}}^{\textup{fin},\z}\cap\dot{\mathcal{E}}_{\mathscr{H}\!/\!\mathscr{I}}^{L,\z}.
	\end{align*}
	Then $\mathcal{D}_{\mathscr{H}\!/\!\mathscr{I}}^{\textup{fin},\z}$ is dense in $\big(\dot{\mathcal{E}}^{\z}_{\mathscr{H}\!/\!\mathscr{I}},\|\cdot\|_{\dot{\mathcal{E}}^{\z}_{\mathscr{H}\!/\!\mathscr{I}}}\big)$.
\end{lemma}
\begin{proof}
	First of all, $\mathcal{E}_{\mathscr{H}\!/\!\mathscr{I}}^{\textup{fin},\z}$ is dense in $\dot{\mathcal{E}}^{\z}_{\mathscr{H}\!/\!\mathscr{I}}$ by definition of the homogeneous spaces and because the Hilbert direct sum of the eigenspaces $Y_{\ell}$ is $L^{2}(\mathbb{S}^{2},\mathrm{d}\omega)$.
	
	Next, $\big(\mathcal{C}^{\infty}_{\mathrm{c}}(\mathbb{S}^{1}_{z}\times\mathbb{R}_{x}\times\mathbb{S}^{2}_{\omega})\times\mathcal{C}^{\infty}_{\mathrm{c}}(\mathbb{S}^{1}_{z}\times\mathbb{R}_{x}\times\mathbb{S}^{2}_{\omega})\big)\cap\mathcal{E}_{\mathscr{H}\!/\!\mathscr{I}}^{\textup{fin},,\z}$ is dense in $\mathcal{E}_{\mathscr{H}\!/\!\mathscr{I}}^{\textup{fin},,\z}$ by standard regularization arguments.
	
	Finally, let $\psi\in\mathcal{C}^{\infty}_{\mathrm{c}}(\mathbb{S}^{1}_{z}\times\mathbb{R}_{x}\times\mathbb{S}^{2}_{\omega})\times\mathcal{C}^{\infty}_{\mathrm{c}}(\mathbb{S}^{1}_{z}\times\mathbb{R}_{x}\times\mathbb{S}^{2}_{\omega})$ such that $\psi\geq 0$, $\|\psi\|_{L^{1}(\mathbb{S}^{1}\times\mathbb{R}\times\mathbb{S}^{2})}=1$ and put $\psi_{n}(z,x,\omega):=\psi(z,n^{-1}x,\omega)$ for each $n\in\mathbb{N}\setminus\{0\}$. Pick $u\in\big(\mathcal{C}^{\infty}_{\mathrm{c}}(\mathbb{S}^{1}_{z}\times\mathbb{R}_{x}\times\mathbb{S}^{2}_{\omega})\times\mathcal{C}^{\infty}_{\mathrm{c}}(\mathbb{S}^{1}_{z}\times\mathbb{R}_{x}\times\mathbb{S}^{2}_{\omega})\big)\cap\mathcal{E}_{\textup{in}}^{\textup{fin},\z}$ and define
	\begin{align*}
	\begin{cases}
	u_{0}^{n}:=u_{0},\\
	u_{1}^{n}:=\displaystyle u_{1}-n^{-1}\psi_{n}\mathrm{e}^{-\!/\!+\mathrm{i}s\z\int_{0}^{x}(V(x')-V_{-\!/\!+})\mathrm{d}x'}\int_{\mathbb{R}}\mathrm{e}^{+\!/\!-\mathrm{i}s\z\int_{0}^{y}(V(x')-V_{-\!/\!+})\mathrm{d}x'}\big(u_{1}-s\z V_{-\!/\!+}u_{0}\big)\!(z,y,\omega)\mathrm{d}y
	\end{cases}.
	\end{align*}
	Then $u^{n}:=(u_{0}^{n},u_{1}^{n})\in\big(\mathcal{C}^{\infty}_{\mathrm{c}}(\mathbb{S}^{1}_{z}\times\mathbb{R}_{x}\times\mathbb{S}^{2}_{\omega})\times\mathcal{C}^{\infty}_{\mathrm{c}}(\mathbb{S}^{1}_{z}\times\mathbb{R}_{x}\times\mathbb{S}^{2}_{\omega})\big)\cap\mathcal{E}_{\textup{in}}^{\textup{fin},\z}\cap\dot{\mathcal{E}}_{\textup{in}}^{L,\z}$ and we have
	\begin{align*}
	\|u_{1}^{n}-u_{1}\|_{\HH^{z}}&\leq n^{-1/2}\|n^{-1/2}\psi_{n}\|_{\HH^{\z}}\|u_{1}-s\z V_{-\!/\!+}u_{0}\|_{L^{1}(\mathbb{S}^{1}\times\mathbb{R}\times\mathbb{S}^{2})}\\
	&\leq Cn^{-1/2}\|u_{1}-s\z V_{-\!/\!+}u_{0}\|_{L^{1}(\mathbb{S}^{1}\times\mathbb{R}\times\mathbb{S}^{2})}
	\end{align*}
	for some constant $C>0$ independent of $n$. It remains to let $n\to+\infty$ to conclude the proof.
\end{proof}
\begin{remark}[Minimal propagation speed]
	\label{In/out remark}
	Let $u=u^{\textup{in}}+u^{\textup{out}}\in\mathcal{D}_{\mathscr{H}}^{\textup{fin},\z}$ and $v=v^{\textup{in}}+v^{\textup{out}}\in\mathcal{D}_{\mathscr{I}}^{\textup{fin},\z}$ such that $\mathrm{Supp\,}u,\mathrm{Supp\,}v\subset\mathbb{S}^{1}_{z}\times[R_{1},R_{2}]_{x}\times\mathbb{S}^{2}_{\omega}$ for some $R_{1},R_{2}\in\mathbb{R}$. Using the relations
	\begin{align*}
		u^{\textup{in}}_{1}&=-\mathrm{i}\partial_{x}u^{\textup{in}}_{0}+s\z Vu^{\textup{in}}_{0},\qquad\qquad\qquad\quad\ \, u^{\textup{out}}_{1}=\mathrm{i}\partial_{x}u^{\textup{out}}_{0}-s\z(V-2V_{-})u^{\textup{out}}_{0},\\
		v^{\textup{in}}_{1}&=-\mathrm{i}\partial_{x}v^{\textup{in}}_{0}-s\z(V-2V_{+})v^{\textup{in}}_{0},\qquad\qquad v^{\textup{out}}_{1}=\mathrm{i}\partial_{x}v^{\textup{out}}_{0}+s\z Vv^{\textup{out}}_{0},
	\end{align*}
	we can integrate by parts in the Kirchhoff formulas \eqref{Kirchhoff type formula 2 simplified} and \eqref{Kirchhoff type formula 2 bis simplified} to obtain
	\begin{align*}
	\big(\mathrm{e}^{\mathrm{i}t\dot{H}^{\z}_{\mathscr{H}}}u^{\textup{in}}\big)_{0}(z,x,\omega)	&=\mathrm{e}^{\mathrm{i}s\z\int_{x}^{x+t}V(x')\mathrm{d}x'}u^{\textup{in}}_{0}(z,x+t,\omega),\\
	\big(\mathrm{e}^{\mathrm{i}t\dot{H}^{\z}_{\mathscr{H}}}u^{\textup{out}}\big)_{0}(z,x,\omega)	&=\mathrm{e}^{\mathrm{i}s\z\int_{x}^{x-t}(V(x')-2V_{-})\mathrm{d}x'}u^{\textup{out}}_{0}(z,x-t,\omega),\\
	\big(\mathrm{e}^{\mathrm{i}t\dot{H}^{\z}_{\mathscr{I}}}v^{\textup{in}}\big)_{0}(z,x,\omega)	&=\mathrm{e}^{-\mathrm{i}s\z\int_{x}^{x+t}(V(x')-2V_{+})\mathrm{d}x'}v_{0}^{\textup{in}}(z,x+t,\omega),\\
	\big(\mathrm{e}^{\mathrm{i}t\dot{H}^{\z}_{\mathscr{I}}}v^{\textup{out}}\big)_{0}(z,x,\omega)	&=\mathrm{e}^{-\mathrm{i}s\z\int_{x}^{x-t}V(x')\mathrm{d}x'}v_{0}^{\textup{out}}(z,x-t,\omega).
	\end{align*}
	It follows:
	\begin{align*}
	\mathrm{Supp\,}\big(\mathrm{e}^{\mathrm{i}t\dot{H}_{\mathscr{H}}^{\z}}u^{\textup{in}}\big),\mathrm{Supp\,}\big(\mathrm{e}^{\mathrm{i}t\dot{H}_{\mathscr{I}}^{\z}}v^{\textup{in}}\big)&\subset\mathbb{S}^{1}_{z}\times\left]-\infty,R_{2}-t\right]_{x}\times\mathbb{S}^{2}_{\omega},\\
	\mathrm{Supp\,}\big(\mathrm{e}^{\mathrm{i}t\dot{H}_{\mathscr{H}}^{\z}}u^{\textup{out}}\big),\mathrm{Supp\,}\big(\mathrm{e}^{\mathrm{i}t\dot{H}_{\mathscr{I}}^{\z}}v^{\textup{out}}\big)&\subset\mathbb{S}^{1}_{z}\times\left[R_{1}+t,+\infty\right[_{x}\times\mathbb{S}^{2}_{\omega}.
	\end{align*}
	The dynamics $(\mathrm{e}^{\mathrm{i}t\dot{H}_{\mathscr{H}\!/\!\mathscr{I}}^{\z}})_{t\in\mathbb{R}}$ thus verifies the \textup{Huygens principle} on $\mathcal{D}_{\mathscr{H}\!/\!\mathscr{I}}^{\textup{fin},\z}$.
\end{remark}
%
%
%
%
%
\subsection{Analytic scattering results}
\label{Analyic scattering results}
We state in this Subsection the scattering results we will prove in Subsection \ref{Proof of the analytic results}. For all of them, we need to first restrict to $\ker(\mathrm{i}\partial_{z}+\z)$ with $\z\in\mathbb{Z}$ in order to apply the abstract theory of \cite{GGH17} for Klein-Gordon equations with a non zero mass term, then to $|s|$ sufficiently small to use \cite[Theorem 3.8]{Be18}.

The first result concerns the uniform boundedness of the propagator $(\mathrm{e}^{\mathrm{i}t\dot{H}^{\z}})_{t\in\mathbb{R}}$. The proof is given in the paragraph \ref{Proof of Theorem {Uniform boundedness of the evolution}}.
\begin{theorem}[Uniform boundedness of the evolution]
	\label{Uniform boundedness of the evolution}
	Let $\z\in\mathbb{Z}\setminus\{0\}$. There exists $s_{0}>0$ such that for all $s\in\left]-s_{0},s_{0}\right[$, there exists a constant $C\equiv C(\z,s_{0})>0$ such that
	\begin{align*}
	\big\|\mathrm{e}^{\mathrm{i}t\dot{H}^{\z}}u\big\|_{\dot{\mathcal{E}}^{\z}}&\leq C\|u\|_{\dot{\mathcal{E}}^{\z}}\qquad\qquad\forall t\in\mathbb{R},\ \forall u\in\dot{\mathcal{E}}^{\z}.
	\end{align*}
\end{theorem}
The next result concerns the asymptotic completeness between the full dynamics and the separable comparison dynamics. The proof is given in the paragraph \ref{Proof of Theorem {Asymptotic completeness, separable comparison dynamics}}.
\begin{theorem}[Asymptotic completeness, separable comparison dynamics]
	\label{Asymptotic completeness, separable comparison dynamics}
	Let $\z\in\mathbb{Z}\setminus\{0\}$. There exists $s_{0}>0$ such that for all $s\in\left]-s_{0},s_{0}\right[$, the following strong limits
	\begin{align*}
	\boldsymbol{W}_{\pm}^{f}&:=\textup{s}-\lim_{t\to+\infty}\mathrm{e}^{-\mathrm{i}t\dot{H}^{\z}}i_{\pm}\mathrm{e}^{\mathrm{i}t\dot{H}^{\z}_{\pm\infty}},\\
	\boldsymbol{W}_{\pm}^{p}&:=\textup{s}-\lim_{t\to-\infty}\mathrm{e}^{-\mathrm{i}t\dot{H}^{\z}}i_{\pm}\mathrm{e}^{\mathrm{i}t\dot{H}^{\z}_{\pm\infty}},\\
	\boldsymbol{\Omega}_{\pm}^{f}&:=\textup{s}-\lim_{t\to+\infty}\mathrm{e}^{-\mathrm{i}t\dot{H}^{\z}_{\pm\infty}}i_{\pm}\mathrm{e}^{\mathrm{i}t\dot{H}^{\z}},\\
	\boldsymbol{\Omega}_{\pm}^{p}&:=\textup{s}-\lim_{t\to-\infty}\mathrm{e}^{-\mathrm{i}t\dot{H}^{\z}_{\pm\infty}}i_{\pm}\mathrm{e}^{\mathrm{i}t\dot{H}^{\z}}
	\end{align*}
	exist as bounded operators $\boldsymbol{W}_{\pm}^{f/p}\in\mathcal{B}(\dot{\mathcal{E}}^{\z}_{\pm\infty},\dot{\mathcal{E}}^{\z})$ and $\boldsymbol{\Omega}_{\pm}^{f/p}\in\mathcal{B}(\dot{\mathcal{E}}^{\z},\dot{\mathcal{E}}^{\z}_{\pm\infty})$.
	%
\end{theorem}
We now state an asymptotic completeness result with the asymptotic profiles. The proof is given in the paragraph \ref{Proof of Theorem {Asymptotic completeness, asymptotic profiles}}.
\begin{theorem}[Asymptotic completeness, asymptotic profiles]
	\label{Asymptotic completeness, asymptotic profiles}
	Let $\z\in\mathbb{Z}\setminus\{0\}$. There exists $s_{0}>0$ such that for all $s\in\left]-s_{0},s_{0}\right[$, the following holds:
	\begin{enumerate}
		\item For all $u\in\mathcal{D}_{l/r}^{\textup{fin},\z}$, the limits
		\begin{align*}
		\boldsymbol{W}_{l/r}^{f}u&=\lim_{t\to+\infty}\mathrm{e}^{-\mathrm{i}t\dot{H}^{\z}}i_{-\!/\!+}\mathrm{e}^{\mathrm{i}t\dot{H}^{\z}_{-\!/\!+}}u,\\
		\boldsymbol{W}_{l/r}^{p}u&=\lim_{t\to-\infty}\mathrm{e}^{-\mathrm{i}t\dot{H}^{\z}}i_{-\!/\!+}\mathrm{e}^{\mathrm{i}t\dot{H}^{\z}_{-\!/\!+}}u
		\end{align*}
		exist in $\dot{\mathcal{E}}^{\z}$. The operators $\boldsymbol{W}_{l/r}^{f/p}$ extend to bounded operators $\boldsymbol{W}_{l/r}^{f/p}\in\mathcal{B}(\dot{\mathcal{E}}^{\z}_{-\!/\!+},\dot{\mathcal{E}}^{\z})$.
		\item The inverse wave operators
		\begin{align*}
		\boldsymbol{\Omega}_{l/r}^{f}&=\textup{s}-\lim_{t\to+\infty}\mathrm{e}^{-\mathrm{i}t\dot{H}^{\z}_{-\!/\!+}}i_{-\!/\!+}\mathrm{e}^{\mathrm{i}t\dot{H}^{\z}},\\
		\boldsymbol{\Omega}_{l/r}^{p}&=\textup{s}-\lim_{t\to-\infty}\mathrm{e}^{-\mathrm{i}t\dot{H}^{\z}_{-\!/\!+}}i_{-\!/\!+}\mathrm{e}^{\mathrm{i}t\dot{H}^{\z}}
		\end{align*}
		exist in $\mathcal{B}(\dot{\mathcal{E}}^{\z},\dot{\mathcal{E}}^{\z}_{-\!/\!+})$.
	\end{enumerate}
\end{theorem}
Finally, we state a last asymptotic completeness result using the geometric profiles. We will prove the following theorem in the paragraph \ref{Proof of Theorem {Asymptotic completeness, geometric profiles}} and give a geometric interpretation in Subsection \ref{Traces on the energy spaces} (see Remark \ref{Geometric interpretation of full wave operators}).
\begin{theorem}[Asymptotic completeness, geometric profiles]
	\label{Asymptotic completeness, geometric profiles}
	Let $\z\in\mathbb{Z}\setminus\{0\}$. There exists $s_{0}>0$ such that for all $s\in\left]-s_{0},s_{0}\right[$, the following holds:
	\begin{enumerate}
		\item For all $u\in\mathcal{D}^{\textup{fin},\z}_{\mathscr{H}\!/\!\mathscr{I}}$, the limits
		\begin{align*}
		\boldsymbol{W}_{\mathscr{H}\!/\!\mathscr{I}}^{f}u&=\lim_{t\to+\infty}\mathrm{e}^{-\mathrm{i}t\dot{H}^{\z}}i_{-\!/\!+}\mathrm{e}^{\mathrm{i}t\dot{H}^{\z}_{\mathscr{H}\!/\!\mathscr{I}}}u,\\
		\boldsymbol{W}_{\mathscr{H}\!/\!\mathscr{I}}^{p}u&=\lim_{t\to-\infty}\mathrm{e}^{-\mathrm{i}t\dot{H}^{\z}}i_{-\!/\!+}\mathrm{e}^{\mathrm{i}t\dot{H}^{\z}_{\mathscr{H}\!/\!\mathscr{I}}}u
		\end{align*}
		exist in $\dot{\mathcal{E}}^{\z}$. The operators $\boldsymbol{W}_{\mathscr{H}\!/\!\mathscr{I}}^{f/p}$ extend to bounded operators $\boldsymbol{W}_{\mathscr{H}\!/\!\mathscr{I}}^{f/p}\in\mathcal{B}(\dot{\mathcal{E}}^{\z}_{\mathscr{H}\!/\!\mathscr{I}},\dot{\mathcal{E}}^{\z})$.
		\item The inverse future/past wave operators
		\begin{align*}
		\boldsymbol{\Omega}_{\mathscr{H}\!/\!\mathscr{I}}^{f}&=\textup{s}-\lim_{t\to+\infty}\mathrm{e}^{-\mathrm{i}t\dot{H}^{\z}_{\mathscr{H}\!/\!\mathscr{I}}}i_{-\!/\!+}\mathrm{e}^{\mathrm{i}t\dot{H}^{\z}},\\
		\boldsymbol{\Omega}_{\mathscr{H}\!/\!\mathscr{I}}^{p}&=\textup{s}-\lim_{t\to-\infty}\mathrm{e}^{-\mathrm{i}t\dot{H}^{\z}_{\mathscr{H}\!/\!\mathscr{I}}}i_{-\!/\!+}\mathrm{e}^{\mathrm{i}t\dot{H}^{\z}}
		\end{align*}
		exist in $\mathcal{B}(\dot{\mathcal{E}}^{\z},\dot{\mathcal{E}}^{\z}_{\mathscr{H}\!/\!\mathscr{I}})$.
	\end{enumerate}
\end{theorem}
\begin{remark}
	Because of the cut-offs $i_{\pm}$, wave operators and inverse wave operators are not inverse. We will however justify this designation for the geometric profiles in Subsection \ref{The full wave operators}.
\end{remark}
The geometric wave and inverse wave operators satisfy the following properties (the proof is given in the paragraph \ref{Proof of Proposition {{Geo wave op prop}}}):
\begin{proposition}
\label{Geo wave op prop}
Under the hypotheses of Theorem \ref{Asymptotic completeness, geometric profiles}, it holds:
	\begin{align*}
		\boldsymbol{\Omega}_{\mathscr{H}}^{f/p}\dot{\mathcal{E}}^{\z}&\subset\Psi_{\mathscr{H}}\big(\dot{\mathcal{H}}^{1,\z}_{\mathscr{H}}\times\{0\}\big),&&\boldsymbol{\Omega}_{\mathscr{H}}^{f/p}\boldsymbol{W}_{\mathscr{H}}^{f/p}=\mathds{1}_{\Psi_{\mathscr{H}}(\dot{\mathcal{H}}^{1,\z}_{\mathscr{H}}\times\{0\})},\\
		\boldsymbol{\Omega}_{\mathscr{I}}^{f/p}\dot{\mathcal{E}}^{\z}&\subset\Psi_{\mathscr{I}}\big(\{0\}\times\dot{\mathcal{H}}^{1,\z}_{\mathscr{I}}\big),&&\boldsymbol{\Omega}_{\mathscr{I}}^{f/p}\boldsymbol{W}_{\mathscr{I}}^{f/p}=\mathds{1}_{\Psi_{\mathscr{I}}(\{0\}\times\dot{\mathcal{H}}^{1,\z}_{\mathscr{I}})}.
	\end{align*}
\end{proposition}
%
%
%
%
%
\subsection{Proof of the analytic results}
\label{Proof of the analytic results}
This Subsection is devoted to the proofs of the scattering results stated in Subsection \ref{Analyic scattering results}. Theorems \ref{Uniform boundedness of the evolution}, \ref{Asymptotic completeness, separable comparison dynamics} and \ref{Asymptotic completeness, asymptotic profiles} are direct consequences of the results obtained in \cite{GGH17} once some geometric hypotheses are checked. Theorem \ref{Asymptotic completeness, geometric profiles} follows from standard arguments as well as propagation estimates of the full dynamics showed in \cite{GGH17}.
%
%
%
%
\subsubsection{Geometric hypotheses}
\label{Geometric hypotheses}
In this paragraph, we check that the geometric hypotheses (G) of \cite{GGH17} are verified in our setting on $\ker(\mathrm{i}\partial_{z}+\z)$, $\z\in\mathbb{Z}\setminus\{0\}$ ($\z$ must not be zero to conserve a mass term). We use the weight $w(r):=\sqrt{(r-r_{-})(r_{+}-r)}$ defined for all $r\in\left]r_{-},r_{+}\right[$.
\begin{itemize}
	\item [(G1)] The operator $P$ in \cite{GGH17} is\footnote{It is assumed in \cite{GGH17} that the coefficients of $P$ are independent of $r$; we can check however that this restriction can be relaxed to a broader class of operators including the one we use.} $-\Delta_{\mathbb{S}^{2}}-m^{2}r^{2}\partial_{z}^{2}$ for us, and satisfies of course $[P,\partial_{z}]=0$.
	\item [(G2)] Set
	\begin{align*}
	h_{0,s}^{\z}&:=F(r)^{-1/2}h_{0}^{\z}F(r)^{1/2}=-r^{-1}F(r)^{1/2}\partial_{r}r^{2}F(r)\partial_{r}r^{-1}F(r)^{1/2}-r^{-2}F(r)\Delta_{\mathbb{S}^{2}}+m^{2}F(r)\z^{2}.
	\end{align*}
	Then $\alpha_1(r)=\alpha_3(r)=r^{-1}F(r)^{1/2}$, $\alpha_2(r)=rF(r)^{1/2}$ and $\alpha_4(r)=mF(r)^{1/2}\z$. These coefficients are clearly smooth in $r$. Furthermore, since we can write $F(r)=g(r)w(r)^2$ with $g(r)=\frac{\Lambda}{3r^2}(r-r_n)(r-r_c)\gtrsim1$ for all $r\in\left]r_-,r_+\right[$, it comes for all $j\in\{1,2,3,4\}$
	\begin{align*}
	&\alpha_j(r)-w(r)\left(i_-(r)\,\alpha_j^-+i_+(r)\,\alpha_j^+\right)=w(r)\left(g(r)^{1/2}-\alpha_j^{\pm}\right)=\mathcal{O}_{r\to r_\pm}\big(w(r)^{2}\big)
	\end{align*}
	for
	\begin{align*}
	\alpha_1^\pm&=\alpha_3^\pm=\frac{\alpha_{2}^{\pm}}{r_{\pm}^{2}}=\frac{r_{\pm}\alpha_{4}^{\pm}}{m\z}=\frac{1}{r_\pm^2}\sqrt{\frac{\Lambda(r_\pm-r_n)(r_\pm-r_c)}{3}}.
	\end{align*}
	Also, we clearly have $\alpha_j(r)\gtrsim w(r)$. Direct computations show that
	\begin{align*}
	\partial_r^m\partial_{\omega}^n\left(\alpha_j-w\left(i_-\,\alpha_j^-+i_+\,\alpha_j^+\right)\right)\!(r)=\mathcal{O}_{r\to r_\pm}\big(w(r)^{2-2m}\big)
	\end{align*}
	for all $m,n\in\mathbb{N}$.
	\item [(G3)] Set $k_{s,v}^{\z}:=k_{s}^{\z}:=k^{\z}$ and $k_{s,r}^{\z}:=0$, and put $k_{s,v}^{\z,-}:=sV_{-}$ and $k_{s,r}^{\z,+}:=sV_{+}$. We have $V(r)-V_{\pm}=\mathcal{O}_{r\to r_\pm}(|r_+-r_{\pm}|)=\mathcal{O}_{r\to r_{\pm}}\big(w(r)^2\big)$ and $\partial_r^m\partial_{\omega}^nV(r)$ is bounded for any $m,n\in\mathbb{N}$.
	\item [(G4) \& (G6)] Our perturbed operator $h_{0,s}^{\z}$ is the same as in \cite{GGH17} at the beginning of the paragraph 13.1 with $\Delta_{r}=F$, but without the bounded term $\lambda(r^{2}+a^{2})$. We then copy the proof, details are omitted.
	\item [(G5)] Because $\z\neq0$, we clearly have
	\begin{align*}
	h_{0}^{\z}&=-\alpha_1(r)\partial_r w(r)^2r^2g(r)\partial_r\alpha_1(r)-\alpha_1(r)^2\Delta_{\mathbb{S}^{2}}+\alpha_1(r)^2m^2r^2\z^{2}\\
	&=\alpha_1(r)\left(-\partial_r w(r)^2r^2g(r)\partial_r-\Delta_{\mathbb{S}^{2}}+m^2r^2\z^{2}\right)\alpha_1(r)\\
	&\gtrsim\alpha_1(r)\left(-\partial_r w(r)^2\partial_r-\Delta_{\mathbb{S}^{2}}+1\right)\alpha_1(r).
	\end{align*}
	\item [(G7)] We check that $(h_{0}-\kk_{+}^{2},\kk_{+})$ and $(h_{0}-(\kk_{-}-k_{-})^{2},\kk_{-}-k_{-})$ satisfy (G5). Since $\alpha_1(r),k_{\pm}(r)-sV_{\pm}=\mathcal{O}_{r\to r_\pm}(|r_\pm-r|)$, we can write for $|s|<mr_-$:
	\begin{align*}
	\hh_{\pm}&=-\alpha_1(r)\partial_r w(r)^2r^2g(r)\partial_r\alpha_1(r)-\alpha_1(r)^2\Delta_{\mathbb{S}^{2}}+\alpha_1(r)^2m^2r^2\z^{2}-(k_{\pm}(r)-sV_{\pm})^2\\
	&=\alpha_1(r)\left(-\partial_r w(r)^2r^2g(r)\partial_r-\Delta_{\mathbb{S}^{2}}+m^2r^2\z^{2}-\frac{(k_{\pm}(r)-V_{\pm})^2}{\alpha_1(r)^2}\right)\alpha_1(r)\\
	&\gtrsim\alpha_1(r)\left(-\partial_r w(r)^2\partial_r-\Delta_{\mathbb{S}^{2}}+1\right)\alpha_1(r).
	\end{align*}
\end{itemize}
The geometric hypotheses are thus satisfied and we can apply the scattering results of Section 10 in \cite{GGH17}.
%
%
%
%
\subsubsection{Proof of Theorem \ref{Uniform boundedness of the evolution}}
\label{Proof of Theorem {Uniform boundedness of the evolution}}
Let us show Theorem \ref{Uniform boundedness of the evolution}. This result is similar to \cite[Theorem 12.1]{GGH17}.

First, pick $\delta>0$ and set $w(r):=\sqrt{(r-r_{-})(r_{+}-r)}$ defined for all $r\in\left]r_{-},r_{+}\right[$. In \cite{Be18}, Subsections 3.3 and 3.4, it is shown that for a positive mass $m^{2}>0$ (which correspond to any $\z\neq 0$ here), there exists $\varepsilon>0$ such that for all $s\in\mathbb{R}$ sufficiently small, the weighted resolvents
\begin{align*}
w^{\delta}(\dot{\widetilde{H}}\,\!^{\z}_{\pm}-z)^{-1}w^{\delta}&:\dot{\widetilde{\mathcal{E}}}_{\pm}\to\dot{\widetilde{\mathcal{E}}}_{\pm},\qquad\qquad	w^{\delta}(\dot{H}^{\z}-z)^{-1}w^{\delta}:\dot{\mathcal{E}}\to\dot{\mathcal{E}}
\end{align*}
extend form $\mathbb{C}^{+}$ into the strip $\big\{\lambda\in\mathbb{C}\mid\Im\lambda>-\varepsilon\big\}$ as meromorphic operators. The poles are called \textit{resonances}. Then Proposition 3.3 and Theorem 3.8 state that for both the above operators, there is no resonance in a tighter strip $\big\{\lambda\in\mathbb{C}\mid\Im\lambda>-\varepsilon'\big\}$ for some $\varepsilon'\in\left]0,\varepsilon\right[$. Furthermore, \cite[Corollary 3.9]{Be18} shows that the spectrum of $\dot{H}$ is real provided that $s$ is small enough.

We can now follow the proof of \cite[Theorem 12.1]{GGH17} in Subsection 13.2 therein. It follows from \cite[Theorem 7.1]{GGH17} since $\dot{H}^{\z}$ has no eigenvalue and both $\dot{\widetilde{H}}_{\pm}^{\z}$ and $\dot{H}^{\z}$ have no resonance on $\mathbb{R}$. 
%
%
%
%
\subsubsection{Proof of Theorem \ref{Asymptotic completeness, separable comparison dynamics}}
\label{Proof of Theorem {Asymptotic completeness, separable comparison dynamics}}
Let us show Theorem \ref{Asymptotic completeness, separable comparison dynamics}. It is similar to \cite[Theorem 12.2 ]{GGH17} for $n\neq0$ which follows from Theorem 10.5 therein once we have showed the absence of complex pure point spectrum of $\dot{H}$ and also the absence of real resonances of $\dot{H}_{\pm\infty}$. These conditions follow again from \cite[Proposition 13.1]{GGH17} for $s$ small enough and $\z\neq 0$.
%
%
%
%
\subsubsection{Proof of Theorem \ref{Asymptotic completeness, asymptotic profiles}}
\label{Proof of Theorem {Asymptotic completeness, asymptotic profiles}}
Theorem \ref{Asymptotic completeness, asymptotic profiles} is similar to \cite[Theorem 12.3]{GGH17} and the proof therein applies to our setting. As we will show below the same result for the slightly more complicated geometric profiles (\textit{cf.} Subsection \ref{Proof of Theorem {Asymptotic completeness, geometric profiles}}), we omit the details here.
%
%
%
%
\subsubsection{Proof of Theorem \ref{Asymptotic completeness, geometric profiles}}
\label{Proof of Theorem {Asymptotic completeness, geometric profiles}}
The proof of Theorem \ref{Asymptotic completeness, geometric profiles} uses the following fact:
\begin{lemma}
	\label{Injection lemma}
	The operators $i_{-\!/\!+}:\dot{\mathcal{E}}^{\z}\to\dot{\mathcal{E}}^{\z}_{\mathscr{H}\!/\!\mathscr{I}}$ are bounded.
\end{lemma}
\begin{proof}
	We only show the result for $i_{-}$. Let $u=(u_{0},u_{1})\in\dot{\mathcal{E}}$ and write
	\begin{align*}
	\|i_{-}u\|_{\dot{\mathcal{E}}^{\z}_{\mathscr{H}}}^{2}&=\big\|(\partial_{x}-\mathrm{i}(k^{z}-k_{\mathscr{H}}^{z}))i_{-}u_{0}\big\|_{\HH^{\z}}^{2}+\big\|i_{-}(u_{1}-\mathrm{i}k^{\z}_{\mathscr{H}}u_{0})\big\|_{\HH^{\z}}^{2}\\
	&\lesssim\big\|[\partial_{x},i_{-}]u_{0}\big\|_{\HH^{\z}}^{2}+\|\partial_{x}u_{0}\|_{\HH^{\z}}^{2}+\big\|(k^{z}-k_{\mathscr{H}}^{z})i_{-}u_{0}\big\|_{\HH^{\z}}^{2}+\|u_{1}-\mathrm{i}k^{\z}u_{0}\|_{\HH^{\z}}^{2}+\big\|i_{-}(k^{\z}-k^{\z}_{\mathscr{H}})u_{0}\big\|_{\HH^{\z}}^{2}\\
	&\lesssim\|u\|_{\dot{\mathcal{E}}^{\z}}^{2}+\|fu_{0}\|_{\HH^{\z}}^{2}
	\end{align*}
	where $f\in\mathcal{C}^{\infty}(\mathbb{R}_{x},\mathbb{R})$ is exponentially decaying at infinity. Using Hardy type inequality (ii) of \cite[Lemma 9.5]{GGH17}, we get
	\begin{align*}
	\|fu_{0}\|_{\HH^{\z}}&\lesssim\|h_{0}^{1/2}u_{0}\|_{\HH^{\z}}\lesssim\|u\|_{\dot{\mathcal{E}}^{\z}}
	\end{align*}
	and the lemma follows.
\end{proof}
\begin{remark}
	\label{Remark i_+-}
	We may notice that $i_{-\!/\!+}\dot{\mathcal{E}}_{\mathscr{H}\!/\!\mathscr{I}}^{\z}\to\dot{\mathcal{E}}^{\z}$ are not bounded operators. However, the proof of the existence of the direct future wave operator $\boldsymbol{W}_{\mathscr{H}}^{f}$ below will show that
	\begin{align*}
	i_{-\!/\!+}\mathrm{e}^{\mathrm{i}t\dot{H}^{\z}_{\mathscr{H}\!/\!\mathscr{I}}}:\dot{\mathcal{E}}_{\mathscr{H}\!/\!\mathscr{I}}^{\z}\to\dot{\mathcal{E}}^{\z}
	\end{align*}
	are bounded at the limits $t\to\pm\infty$ as extensions of bounded operators defined on the dense subspaces $\mathcal{D}^{\textup{fin},\z}_{\mathscr{H}\!/\!\mathscr{I}}$.
\end{remark}
\begin{proof}[Proof of Theorem \ref{Asymptotic completeness, geometric profiles}]
	We will only show the theorem for the future operators in the $\mathscr{H}$ case. We closely follow the proof of \cite[Theorem 12.3]{GGH17}.
	\paragraph{Existence of the future direct wave operator $\boldsymbol{W}_{\mathscr{H}}^{f}$.} Let $u=u^{\textup{in}}+u^{\textup{out}}\in\mathcal{D}^{\textup{fin},\z}_{\mathscr{H}}$. By Remark \ref{In/out remark},
	\begin{align*}
	(\mathrm{e}^{\mathrm{i}t\dot{H}_{\mathscr{H}}^{\z}}u^{\textup{in}})_{0}(z,x,\omega)&=\mathrm{e}^{\mathrm{i}s\z\int_{x}^{x+t}V(x')\mathrm{d}x'}u_{0}^{\textup{in}}(z,x+t,\omega),\\
	\big(\mathrm{e}^{\mathrm{i}t\dot{H}^{\z}_{\mathscr{H}}}u^{\text{out}}\big)_{0}(z,x,\omega)	&=\mathrm{e}^{\mathrm{i}s\z\int_{x}^{x-t}(V(x')-2V_{-})\mathrm{d}x'}u^{\textup{out}}_{0}(z,x-t,\omega)
	\end{align*}
	and then $\boldsymbol{W}_{\mathscr{H}}^{f}u^{\textup{out}}=0$ because of the support of $i_{-}$. From now on, we write $u=u^{\textup{in}}$. We use Cook's method: a sufficient condition for the limit to exist in $\dot{\mathcal{E}}^{\z}$ is
	\begin{align*}
	\frac{\mathrm{d}}{\mathrm{d}t}\left(\mathrm{e}^{-\mathrm{i}t\dot{H}^{\z}}i_{-}\mathrm{e}^{\mathrm{i}t\dot{H}^{\z}_{\mathscr{H}}}u\right)&=\mathrm{i}\mathrm{e}^{-\mathrm{i}t\dot{H}^{\z}}\big(\dot{H}^{\z}i_{-}-i_{-}\dot{H}^{\z}_{\mathscr{H}}\big)\mathrm{e}^{\mathrm{i}t\dot{H}^{\z}_{\mathscr{H}}}u\in L^{1}(\mathbb{R}^{+}_{t},\mathrm{d}t;\dot{\mathcal{E}}^{\z}).
	\end{align*}
	Put $v=(v_0,v_1):=\mathrm{e}^{\mathrm{i}t\dot{H}^{\z}_{\mathscr{H}}}u$. The above expression makes sense since $v$ is smooth and compactly supported. Recalling the definition of the operators $h^{\z}$ and $k^{\z}$ in the paragraph \ref{The full dynamics} as well as $h^{\z}_{\mathscr{H}}$ and $k^{\z}_{\mathscr{H}}$ in the paragraph \ref{Introduction of the dynamics}, we compute
	\begin{align}
	\label{Cook error term}
	&\dot{H}^{\z}i_{-}-i_{-}\dot{H}^{\z}_{\mathscr{H}}\nonumber\\
	&=\begin{pmatrix}
	0&0\\h^{\z}i_{-}-i_{-}(h^{\z}_{\mathscr{H}}-(k^{\z}_{\mathscr{H}})^{2})&2i_{-}(k^{\z}-k^{\z}_{\mathscr{H}})
	\end{pmatrix}\nonumber\\[1mm]
	&=\begin{pmatrix}
	0&0\\
	-[\partial_{x}^{2},i_{-}]+i_{-}\Big(\mathcal{V}_{\ell}(r)-\mathrm{i}s\z\big(\partial_{x}\mathcal{W}(r)+\mathcal{W}(r)\partial_{x}\big)-2s^{2}\z^{2}V(r)^{2}\mathcal{W}(r)\Big)&-2i_{-}\mathrm{i}s\z \mathcal{W}(r)
	\end{pmatrix}.
	\end{align}
	where
	\begin{align*}
	\mathcal{V}_{\ell}(r)&:=F(r)\left(r^{-1}F'(r)+r^{-2}\sum_{0\leq \ell\leq\ell_{0}}\ell(\ell+1)+m^{2}\z^{2}\right),\qquad\qquad \mathcal{W}(r):=V(r)-V_{-}.
	\end{align*}
	Using the minimal speed of $v_{0}$ (\textit{cf.} Remark \ref{In/out remark}), the uniform boundedness of $\mathrm{e}^{\mathrm{i}t\dot{H}^{\z}}$ (\textit{cf.} Theorem \ref{Uniform boundedness of the evolution}) as well as the exponential decay of $F$ and $V-V_{-}$ (\textit{cf.} \eqref{Decay r-r_- with kappa_pm}), we get:
	\begin{align}
	\label{Rest AS abstract}
	\left\|\frac{\mathrm{d}}{\mathrm{d}t}\left(\mathrm{e}^{-\mathrm{i}t\dot{H}^{\z}}i_{-}\mathrm{e}^{\mathrm{i}t\dot{H}^{\z}_{\mathscr{H}}}u\right)\right\|_{\dot{\mathcal{E}}^{\z}}&\lesssim(1+\ell_{0})^{3}\mathrm{e}^{-2\kappa_{-} t}\big(\|v_{0}\|_{\HH^{\z}}+\|\partial_{x}v_{0}\|_{\HH^{\z}}+\|v_{1}\|_{\HH^{\z}}\big)\nonumber\\
	&\lesssim(1+\ell_{0})^{3}\mathrm{e}^{-2\kappa_{-} t}\big(\|v_{0}\|_{\HH^{\z}}+\|(\partial_{x}-\mathrm{i}s\z (V-V_{-}))v_{0}\|_{\HH^{\z}}\nonumber\\
	&\qquad\qquad\qquad\qquad\qquad\qquad\qquad\ \,+\|v_{1}-\mathrm{i}s\z V_{-}v_{0}\|_{\HH^{\z}}\big)\nonumber\\
	&\lesssim(1+\ell_{0})^{3}\mathrm{e}^{-2\kappa_{-} t}\big(\|v_{0}\|_{\HH^{\z}}+\|v\|_{\dot{\mathcal{E}}^{\z}_{\mathscr{H}}}\big).
	\end{align}
	Here the symbol $\lesssim$ contains no dependence in $\ell_{0}$ (but depends on $\z$). We have $\|v\|_{\dot{\mathcal{E}}^{\z}_{\mathscr{H}}}=\|u\|_{\dot{\mathcal{E}}^{\z}_{\mathscr{H}}}$ and $\|v_{0}(t)\|_{\HH^{\z}}=\|u^{\textup{in}}_{0}(\cdot+t)\|_{\HH^{\z}}=\|u^{\textup{in}}_{0}\|_{\HH^{\z}}$ as translations are unitary on $\HH^{\z}$. In particular, these norms are uniformly bounded in time. This provides a sufficient decay as $t\to+\infty$ in \eqref{Rest AS abstract} and proves the existence of the limit $\boldsymbol{W}_{\mathscr{H}}^{f}$ on $\mathcal{D}^{\textup{fin},\z}_{\mathscr{H}}$.
	
	To extend the existence of the future wave operator to $\dot{\mathcal{E}}_{\mathscr{H}}^{\z}$, we use a density argument: we show that there exists a constant $C>0$ such that for all $u\in\mathcal{D}^{\textup{fin},\z}_{\mathscr{H}}$,
	\begin{align*}
	\|\boldsymbol{W}_{\mathscr{H}}^{f}u\|_{\dot{\mathcal{E}}^{\z}}&\leq C\|u\|_{\dot{\mathcal{E}}^{\z}_{\mathscr{H}}}.
	\end{align*}
	Using again the uniform boundedness of $\mathrm{e}^{\mathrm{i}t\dot{H}^{\z}}$, we can write:
	\begin{align*}
	\|\mathrm{e}^{-\mathrm{i}t\dot{H}^{\z}}i_{-}\mathrm{e}^{\mathrm{i}t\dot{H}^{\z}_{\mathscr{H}}}u\|_{\dot{\mathcal{E}}^{\z}}^{2}&\lesssim\|i_{-}\mathrm{e}^{\mathrm{i}t\dot{H}^{\z}_{\mathscr{H}}}u\|_{\dot{\mathcal{E}}^{\z}}^{2}\\
	&\lesssim\big\langle h^{\z}_{0}i_{-}v_{0},i_{-}v_{0}\big\rangle_{\HH^{\z}}+\|i_{-}(v_{1}-k^{\z}v_{0})\|_{\HH^{\z}}^{2}\\
	&\lesssim\big\langle i_{-}h^{\z}_{0}v_{0},i_{-}v_{0}\big\rangle_{\HH^{\z}}+\big\langle[h^{\z}_{0},i_{-}]v_{0},i_{-}v_{0}\big\rangle_{\HH^{\z}}+\|i_{-}(v_{1}-k^{\z}v_{0})\|_{\HH^{\z}}^{2}\\
	&\lesssim\big\langle i_{-}(h^{\z}_{0}-h^{\z}_{\mathscr{H}})v_{0},i_{-}v_{0}\big\rangle_{\HH^{\z}}+\big\langle h^{\z}_{\mathscr{H}}v_{0},v_{0}\big\rangle_{\HH^{\z}}+\big\langle[h_{0},i_{-}]v_{0},i_{-}v_{0}\big\rangle_{\HH^{\z}}\\
	&+\|i_{-}(k^{\z}_{\mathscr{H}}-k^{\z})v_{0})\|_{\HH^{\z}}^{2}+\|v_{1}-k^{\z}_{\mathscr{H}}v_{0}\|_{\HH^{\z}}^{2}\\
	&\lesssim\|v\|_{\dot{\mathcal{E}}^{\z}_{\mathscr{H}}}^{2}+r(t)\\
	&\lesssim\|u\|_{\dot{\mathcal{E}}^{\z}_{\mathscr{H}}}^{2}+r(t)
	\end{align*}
	where $r(t)=\mathcal{O}_{t\to+\infty}(\mathrm{e}^{-4\kappa_{-}t}\|v_{0}\|^{2}_{\HH^{\z}})$ is a rest similar to the right-hand side of \eqref{Rest AS abstract}. Letting $t\to+\infty$ thus gives the desired result.
	\paragraph{Existence of the future inverse wave operator $\boldsymbol{\Omega}_{\mathscr{H}}^{f}$.} The existence of $\boldsymbol{\Omega}_{\mathscr{H}}^{f}$ necessarily relies on a propagation estimate for the full dynamics $(\mathrm{e}^{\mathrm{i}t\dot{H}^{\z}})_{t\in\mathbb{R}}$. In \cite[Proposition 6.8]{GGH17}, it is shown that for all $\delta>0$ and all $u\in\dot{\mathcal{E}}^{\z}$, we have
	\begin{align}
	\label{Propagation estimate}
	\int_{\mathbb{R}}\big\|w^{\delta}\mathrm{e}^{\mathrm{i}t\dot{H}^{\z}}\chi(\dot{H}^{\z})u\big\|_{\dot{\mathcal{E}}^{\z}}\mathrm{d}t&\lesssim\|u\|_{\dot{\mathcal{E}}^{\z}}
	\end{align}
	where $w(r(x))=\sqrt{(r(x)-r_{-})(r_{+}-r(x))}$ and $\chi\in\mathcal{C}^{\infty}_{\mathrm{c}}(\mathbb{R})$ ($\chi$ must cancel in a neighborhood of the real resonances of $\dot{H}$ in \cite{GGH17}, but \cite[Theorem 3.8]{Be18} shows that no resonance lies on $\mathbb{R}$ for us). The operator $\chi(\dot{H}^{\z})$ is defined with a Helffer-Sj\"{o}strand type formula and is bounded on $\dot{\mathcal{E}}^{\z}$, see \cite[Subsection 5.5]{GGH17}.
	
	To prove that $\boldsymbol{\Omega}_{\mathscr{H}}^{f}$ exists, we show that the sequence $\big(\mathrm{e}^{-\mathrm{i}t\dot{H}_{\mathscr{H}}^{\z}}i_{-}\mathrm{e}^{\mathrm{i}t\dot{H}^{\z}}u\big)_{t>0}$ is Cauchy in $\dot{\mathcal{E}}^{\z}_{\mathscr{H}}$. Let $\varepsilon>0$ and $u\in\dot{\mathcal{E}}^{\z}$. First of all, the uniform boundedness of $\mathrm{e}^{\mathrm{i}t\dot{H}_{\mathscr{H}}^{\z}}$ and $\mathrm{e}^{\mathrm{i}t\dot{H}^{\z}}$ as well as Lemma \ref{Injection lemma} imply that $\mathrm{e}^{-\mathrm{i}t\dot{H}_{\mathscr{H}}^{\z}}i_{-}\mathrm{e}^{\mathrm{i}t\dot{H}^{\z}}\in\mathcal{B}(\dot{\mathcal{E}}^{\z},\dot{\mathcal{E}}_{\mathcal{H}}^{\z})$ for all $t\in\mathbb{R}$. Let us write $u=\sum_{\ell\in\mathbb{N}}u_{\ell}$ where $\omega\mapsto u_{\ell}(z,x,\omega)\in H^{2}(\mathbb{S}^{2},\mathrm{d}\omega)$ and $-\Delta_{\mathbb{S}^{2}}u_{\ell}=\ell(\ell+1)u_{\ell}$ for almost every $(z,x)\in\mathbb{S}^{1}_{z}\times\mathbb{R}_{x}$ (we identify $u_{\ell}$ to an element of $\dot{\mathcal{E}}^{\z}$). Then we have by dominated convergence
	\begin{align*}
		\left\|\mathrm{e}^{-\mathrm{i}t\dot{H}_{\mathscr{H}}^{\z}}i_{-}\mathrm{e}^{\mathrm{i}t\dot{H}^{\z}}u-\mathrm{e}^{-\mathrm{i}t'\dot{H}_{\mathscr{H}}^{\z}}i_{-}\mathrm{e}^{\mathrm{i}t'\dot{H}^{\z}}\sum_{0\leq\ell\leq\ell_{0}}u_{\ell}\right\|_{\dot{\mathcal{E}}^{\z}_{\mathscr{H}}}&=\left\|\mathrm{e}^{-\mathrm{i}t\dot{H}_{\mathscr{H}}^{\z}}i_{-}\mathrm{e}^{\mathrm{i}t\dot{H}^{\z}}\chi(\dot{H}^{\z})\sum_{\ell>\ell_{0}}u_{\ell}\right\|_{\dot{\mathcal{E}}^{\z}_{\mathscr{H}}}\\
		&\leq\sum_{\ell>\ell_{0}}\big\|\mathrm{e}^{-\mathrm{i}t\dot{H}_{\mathscr{H}}^{\z}}i_{-}\mathrm{e}^{\mathrm{i}t\dot{H}^{\z}}u_{\ell}\big\|_{\dot{\mathcal{E}}^{\z}_{\mathscr{H}}}\\
		&\leq C\sum_{\ell>\ell_{0}}\|u_{\ell}\|_{\dot{\mathcal{E}}^{\z}_{\mathscr{H}}}\\
		&<\varepsilon/6
	\end{align*}
	for $\ell_{0}$ large enough as the remainder of a convergent series. Fix such a $\ell_{0}$ and call again $u$ the truncated sum $\sum_{0\leq\ell\leq\ell_{0}}u_{\ell}$. Next,
	\begin{align*}
	\big\|\mathrm{e}^{-\mathrm{i}t\dot{H}_{\mathscr{H}}^{\z}}i_{-}\mathrm{e}^{\mathrm{i}t\dot{H}^{\z}}u-\mathrm{e}^{-\mathrm{i}t'\dot{H}_{\mathscr{H}}^{\z}}i_{-}\mathrm{e}^{\mathrm{i}t'\dot{H}^{\z}}u\big\|_{\dot{\mathcal{E}}^{\z}_{\mathscr{H}}}&\leq\big\|\mathrm{e}^{-\mathrm{i}t\dot{H}_{\mathscr{H}}^{\z}}i_{-}\mathrm{e}^{\mathrm{i}t\dot{H}^{\z}}(1-\chi(\dot{H}^{\z}))u\big\|_{\dot{\mathcal{E}}^{\z}_{\mathscr{H}}}\\
	&+\big\|\mathrm{e}^{-\mathrm{i}t'\dot{H}_{\mathscr{H}}^{\z}}i_{-}\mathrm{e}^{\mathrm{i}t'\dot{H}^{\z}}(1-\chi(\dot{H}^{\z}))u\big\|_{\dot{\mathcal{E}}^{\z}_{\mathscr{H}}}\\
	&+\big\|\mathrm{e}^{-\mathrm{i}t\dot{H}_{\mathscr{H}}^{\z}}i_{-}\mathrm{e}^{\mathrm{i}t\dot{H}^{\z}}\chi(\dot{H}^{\z})u-\mathrm{e}^{-\mathrm{i}t'\dot{H}_{\mathscr{H}}^{\z}}i_{-}\mathrm{e}^{\mathrm{i}t'\dot{H}^{\z}}\chi(\dot{H}^{\z})u\big\|_{\dot{\mathcal{E}}^{\z}_{\mathscr{H}}}.
	\end{align*}
	We will show that the right-hand side above is lesser than $2\varepsilon/3$ for $t,t'$ sufficiently large.
	
	Using again the uniform boundedness of $\mathrm{e}^{\mathrm{i}t\dot{H}_{\mathscr{H}}^{\z}}$ and $\mathrm{e}^{\mathrm{i}t\dot{H}^{\z}}$ as well as Lemma \ref{Injection lemma}, we can write
	\begin{align*}
	\big\|\mathrm{e}^{-\mathrm{i}t\dot{H}_{\mathscr{H}}^{\z}}i_{-}\mathrm{e}^{\mathrm{i}t\dot{H}^{\z}}(1-\chi(\dot{H}^{\z}))u\big\|_{\dot{\mathcal{E}}^{\z}_{\mathscr{H}}}+\big\|\mathrm{e}^{-\mathrm{i}t'\dot{H}_{\mathscr{H}}^{\z}}i_{-}\mathrm{e}^{\mathrm{i}t'\dot{H}^{\z}}(1-\chi(\dot{H}^{\z}))u\big\|_{\dot{\mathcal{E}}^{\z}_{\mathscr{H}}}&\leq 2C\big\|(1-\chi(\dot{H}^{\z}))u\big\|_{\dot{\mathcal{E}}^{\z}}.
	\end{align*}
	By \cite[Proposition 5.11]{GGH17} combined with the fact that $\dot{H}^{\z}$ has no eigenvalue if $s$ is small enough (see the proof of Theorem \ref{Uniform boundedness of the evolution} in the paragraph \ref{Proof of Theorem {Uniform boundedness of the evolution}}),
	\begin{align}
	\label{Strong limit chi(H/L)}
		\textup{s}-\lim_{L\to+\infty}\chi(\dot{H}^{\z}/L)&=1.
	\end{align}
	Fix $\chi$ so that $\big\|(1-\chi(\dot{H}^{\z}))u\big\|_{\dot{\mathcal{E}}^{\z}}<\varepsilon/6C$. Thus
	\begin{align*}
		\big\|\mathrm{e}^{-\mathrm{i}t\dot{H}_{\mathscr{H}}^{\z}}i_{-}\mathrm{e}^{\mathrm{i}t\dot{H}^{\z}}(1-\chi(\dot{H}^{\z}))u\big\|_{\dot{\mathcal{E}}^{\z}_{\mathscr{H}}}+\big\|\mathrm{e}^{-\mathrm{i}t'\dot{H}_{\mathscr{H}}^{\z}}i_{-}\mathrm{e}^{\mathrm{i}t'\dot{H}^{\z}}(1-\chi(\dot{H}^{\z}))u\big\|_{\dot{\mathcal{E}}^{\z}_{\mathscr{H}}}&<\frac{\varepsilon}{3}.
	\end{align*}
	It remains to show that
	\begin{align*}
		\big\|\mathrm{e}^{-\mathrm{i}t\dot{H}_{\mathscr{H}}^{\z}}i_{-}\mathrm{e}^{\mathrm{i}t\dot{H}^{\z}}\chi(\dot{H}^{\z})u-\mathrm{e}^{-\mathrm{i}t'\dot{H}_{\mathscr{H}}^{\z}}i_{-}\mathrm{e}^{\mathrm{i}t'\dot{H}^{\z}}\chi(\dot{H}^{\z})u\big\|_{\dot{\mathcal{E}}^{\z}_{\mathscr{H}}}&<\frac{\varepsilon}{3}
	\end{align*}
	for $t,t'$ large enough.
	
	Pick $\delta\in\left]0,2\right[$ and compute:
	\begin{align*}
	\frac{\mathrm{d}}{\mathrm{d}t}\left(\mathrm{e}^{-\mathrm{i}t\dot{H}_{\mathscr{H}}^{\z}}i_{-}\mathrm{e}^{\mathrm{i}t\dot{H}^{\z}}\chi(\dot{H}^{\z})u\right)&=\mathrm{i}\mathrm{e}^{-\mathrm{i}t\dot{H}_{\mathscr{H}}^{\z}}\left(\dot{H}_{\mathscr{H}}^{\z}i_{-}-i_{-}\dot{H}^{\z}\right)w^{-\delta}w^{\delta}\mathrm{e}^{\mathrm{i}t\dot{H}^{\z}}\chi(\dot{H}^{\z})u
	\end{align*}
	Then \eqref{Cook error term} above shows that for all $v=(v_{0},v_{1})\in\dot{\mathcal{E}}^{\z}$,
	\begin{align*}
	\left\|\left(\dot{H}_{\mathscr{H}}^{\z}i_{-}-i_{-}\dot{H}^{\z}\right)w^{-\delta}v\right\|_{\dot{\mathcal{E}}^{\z}_{\mathscr{H}}}&\lesssim(1+\ell_{0})^{3}\left(\|fv_{0}\|_{\HH^{\z}}+\|gv_{1}\|_{\HH^{\z}}\right)
	\end{align*}
	for some smooth functions $f,g\in\mathcal{O}_{|x|\to+\infty}\big(\mathrm{e}^{(\delta-2\kappa_{-})|x|}\big)$. Using the Hardy type inequality (ii) of \cite[Lemma 9.5]{GGH17}, it follows
	\begin{align*}
	\left\|\left(\dot{H}_{\mathscr{H}}^{\z}i_{-}-i_{-}\dot{H}^{\z}\right)w^{-\delta}v\right\|_{\dot{\mathcal{E}}^{\z}_{\mathscr{H}}}&\lesssim(1+\ell_{0})^{3}\left(\|h_{0}^{1/2}v_{0}\|_{\HH^{\z}}+\|v_{1}-k^{\z}v_{0}\|_{\HH^{\z}}\right)\lesssim(1+\ell_{0})^{3}\|v\|_{\dot{\mathcal{E}}^{\z}}.
	\end{align*}
	With Lemma \ref{Injection lemma}, this gives
	\begin{align*}
	\mathrm{e}^{-\mathrm{i}t\dot{H}_{\mathscr{H}}^{\z}}\left(\dot{H}_{\mathscr{H}}^{\z}i_{-}-i_{-}\dot{H}^{\z}\right)w^{-\delta}&\in\mathcal{B}(\dot{\mathcal{E}}^{\z}_{\mathscr{H}},\dot{\mathcal{E}}^{\z}).
	\end{align*}
	The propagation estimate \eqref{Propagation estimate} then implies that
	\begin{align*}
	\frac{\mathrm{d}}{\mathrm{d}t}\left(\mathrm{e}^{-\mathrm{i}t\dot{H}_{\mathscr{H}}^{\z}}i_{-}\mathrm{e}^{\mathrm{i}t\dot{H}^{\z}}\chi(\dot{H}^{\z})u\right)&\in L^{1}(\mathbb{R},\mathrm{d}t)
	\end{align*}
	whence
	\begin{align*}
	\big\|\mathrm{e}^{-\mathrm{i}t\dot{H}_{\mathscr{H}}^{\z}}i_{-}\mathrm{e}^{\mathrm{i}t\dot{H}^{\z}}\chi(\dot{H}^{\z})u-\mathrm{e}^{-\mathrm{i}t'\dot{H}_{\mathscr{H}}^{\z}}i_{-}\mathrm{e}^{\mathrm{i}t'\dot{H}^{\z}}\chi(\dot{H}^{\z})u\big\|_{\dot{\mathcal{E}}^{\z}_{\mathscr{H}}}&<\varepsilon/3
	\end{align*}
	for $t,t'$ large enough. The proof is complete.
\end{proof}
%
%
%
%
%
\subsubsection{Proof of Proposition \ref{Geo wave op prop}}
\label{Proof of Proposition {{Geo wave op prop}}}
We only show the $f$, $\mathscr{H}$ case:
\begin{align}
\label{Range in}
	\boldsymbol{\Omega}_{\mathscr{H}}^{f}\dot{\mathcal{E}}^{\z}&\subset\Psi_{\mathscr{H}}\big(\dot{\mathcal{H}}^{1,\z}_{\mathscr{H}}\times\{0\}\big),\\
\label{Semi inversion in}
	\boldsymbol{\Omega}_{\mathscr{H}}^{f}\boldsymbol{W}_{\mathscr{H}}^{f}&=\mathds{1}_{\Psi_{\mathscr{H}}(\dot{\mathcal{H}}^{1,\z}_{\mathscr{H}}\times\{0\})}.
\end{align}
\begin{proof}[Proof of \eqref{Range in}]
	We follow the proof of a similar proposition in \cite{Mi}. Let
	\begin{align*}
		\Upsilon_{\mathscr{H}}&:\dot{\mathcal{E}}^{\z}_{\mathscr{H}}\ni u=(u_{0},u_{1})\longmapsto u_{1}-\mathrm{i}L_{\mathscr{H}}u_{0}\in\HH^{\z}.
	\end{align*}
	It defines a continuous operator as
	\begin{align*}
		\|\Upsilon_{\mathscr{H}}u\|_{\HH^{\z}}&=\|u_{1}-\mathrm{i}L_{\mathscr{H}}u_{0}\|_{\HH^{\z}}\leq\|(L_{\mathscr{H}}+\mathrm{i}k_{\mathscr{H}})u_{0}\|_{\HH^{\z}}+\|u_{1}-k_{\mathscr{H}}u_{0}\|_{\HH^{\z}}=\|u\|_{\dot{\mathcal{E}}_{\mathscr{H}}^{\z}}.
	\end{align*}
	Clearly $\ker\Upsilon_{\mathscr{H}}=\Psi_{\mathscr{H}}\big(\dot{\mathcal{H}}^{1,\z}_{\mathscr{H}}\times\{0\}\big)$ so we will show that
	\begin{align*}
		\Upsilon_{\mathscr{H}}\boldsymbol{\Omega}^{f}_{\mathscr{H}}&=0.
	\end{align*}

	We claim that
	\begin{align}
	\label{Claim Upsilon}
		\big\|\Upsilon_{\mathscr{H}}\mathrm{e}^{\mathrm{i}t\dot{H}^{\z}_{\mathscr{H}}}u\big\|_{\HH^{\z}}&=\sqrt{2}\big\|\Upsilon_{\mathscr{H}}u\big\|_{\HH^{\z}}\qquad\qquad\forall u\in\dot{\mathcal{E}}^{\z}_{\mathscr{H}}.
	\end{align}
	To see this, write $\Upsilon_{\mathscr{H}}=\frac{\mathrm{i}}{\sqrt{2}}(L_{+}-L_{\mathscr{H}})\pi_{1}\Psi_{\mathscr{H}}^{-1}$ with $\pi_{1}(u_{0},u_{1}):=u_{1}$. Now \eqref{Claim Upsilon} follows from $\Psi_{\mathscr{H}}^{-1}\mathrm{e}^{\mathrm{i}t\dot{H}_{\mathscr{H}}^{\z}}=\mathrm{e}^{\mathrm{i}t\dot{\mathbb{H}}_{\mathscr{H}}^{\z}}\Psi_{\mathscr{H}}^{-1}$ and the unitarity of $\mathrm{e}^{\mathrm{i}t\dot{\mathbb{H}}_{\mathscr{H}}^{\z}}\Psi_{\mathscr{H}}^{-1}:\dot{\mathcal{E}}_{\mathscr{H}}^{\z}\to\dot{\mathcal{H}}^{1,\z}_{\mathscr{H}}\times\dot{\mathcal{H}}^{1,\z}_{\mathscr{H}}$ and $L_{\mathscr{H}}+\mathrm{i}k_{\mathscr{H}}=-\frac{1}{2}(L_{+}-\mathrm{i}L_{\mathscr{H}}):\dot{\mathcal{H}}^{1,\z}_{\mathscr{H}}\to\HH^{\z}$.

	Let now $u\in\dot{\mathcal{E}}^{\z}$. In view of \eqref{Claim Upsilon}, we are boiled to show that\footnote{Recall that $i_-\in\mathcal{B}\big(\dot{\mathcal{E}}^{\z},\dot{\mathcal{E}}^{\z}_{\mathscr{H}}\big)$, \textit{cf.} Lemma \ref{Injection lemma}.}
	\begin{align}
	\label{To show Upsilon}
		\lim_{t\to+\infty}\big\|\Upsilon_{\mathscr{H}}i_{-}\mathrm{e}^{\mathrm{i}t\dot{H}^{\z}}u\big\|_{\HH^{\z}}&=0.
	\end{align}
	%
	%
	%
	Fix $\varepsilon>0$ and pick $\ell_{0}\in\mathbb{N}$ so that
	\begin{align}
	\label{First error Upsilon}
		\big\|\Upsilon_{\mathscr{H}}i_{-}\mathrm{e}^{\mathrm{i}t\dot{H}^{\z}}u\big\|_{\HH^{\z}}&\leq\left\|\Upsilon_{\mathscr{H}}i_{-}\mathrm{e}^{\mathrm{i}t\dot{H}^{\z}}\sum_{0\leq\ell\leq\ell_{0}}u_{\ell}\right\|_{\HH^{\z}}+\left\|\Upsilon_{\mathscr{H}}i_{-}\mathrm{e}^{\mathrm{i}t\dot{H}^{\z}}\left(u-\sum_{0\leq\ell\leq\ell_{0}}u_{\ell}\right)\right\|_{\HH^{\z}}\nonumber\\
		&\leq\left\|\Upsilon_{\mathscr{H}}i_{-}\mathrm{e}^{\mathrm{i}t\dot{H}^{\z}}\sum_{0\leq\ell\leq\ell_{0}}u_{\ell}\right\|_{\HH^{\z}}+\varepsilon/2
	\end{align}
	where $-\Delta_{\mathbb{S}^{2}}u_{\ell}=\ell(\ell+1)u_{\ell}$ as in the proof of Theorem \ref{Asymptotic completeness, geometric profiles} (we used the continuity of $\Upsilon_{\mathscr{H}}i_{-}\mathrm{e}^{\mathrm{i}t\dot{H}^{\z}}:\dot{\mathcal{E}}^{\z}\to\HH^{\z}$ as well as dominated convergence). We will write again $u$ for $\sum_{0\leq \ell\leq\ell_{0}}u_{\ell}$. Set $v(t):=\mathrm{e}^{\mathrm{i}t\dot{H}^{\z}}u$ and compute:
	\begin{align}
	\label{Eq Upsilon}
		\mathrm{i}(\partial_{t}+L_{+})\Upsilon_{\mathscr{H}}i_{-}\mathrm{e}^{\mathrm{i}t\dot{H}^{\z}}u&=(\partial_{t}+L_{+})(\partial_{t}+L_{\mathscr{H}})i_{-}v_{0}(t)\nonumber\\
		&=\big((\partial_{t}-\mathrm{i}k_{\mathscr{H}})^{2}+h_{\mathscr{H}}\big)i_{-}v_{0}(t)\nonumber\\
		%
		%
		%
		%
		&=\Big([h_{\mathscr{H}},i_{-}]-i_{-}(k_{\mathscr{H}}^{2}-k^{2})+i_{-}(h_{\mathscr{H}}-h)\Big)v_{0}(t)+2i_{-}(k_{\mathscr{H}}-k)v_{1}(t)\nonumber\\
		&=:\Xi(t).
	\end{align}
	This expression makes sense in $\HH^{\z}$: we have
	\begin{align*}
		\|\Xi(t)\|_{\HH^{\z}}&\lesssim\|f_{1}(x)v_{0}\|_{\HH^{\z}}+\|f_{2}(x)\partial_{x}v_{0}\|_{\HH^{\z}}+\|f_{3}(x)v_{1}\|_{\HH^{\z}}\\
		&\lesssim\|g(x)v_{0}\|_{\HH^{\z}}+\|\partial_{x}v_{0}\|_{\HH^{\z}}+\|v_{1}-kv_{0}\|_{\HH^{\z}}
	\end{align*}
	with $|f_{j}(x)|,|g(x)|\leq C_{j}\ell_{0}^{3}\mathrm{e}^{-2\kappa|x|}$, then Hardy type inequality \cite[Lemma 9.5]{GGH17} as well as uniform boundedness of $\mathrm{e}^{\mathrm{i}t\dot{H}^{\z}}$ yield
	\begin{align*}
		\|\Xi(t)\|_{\HH^{\z}}&\lesssim\|\tilde{g}(x)w^{\delta}v(t)\|_{\dot{\mathcal{E}}^{\z}}\lesssim\|\tilde{g}(x)w^{\delta}u\|_{\dot{\mathcal{E}}^{\z}}
	\end{align*}
	with $\delta\in\left]0,2\kappa\right[$ and $|\tilde{g}(x)|\leq \tilde{C}\ell_{0}^{3}\mathrm{e}^{-2\kappa|x|-\delta}$. By \cite[Proposition 6.8]{GGH17}, we have $\Xi\in L^{1}(\mathbb{R}_{t},\HH^{\z})$. From \eqref{Eq Upsilon} and the fact that $L_{+}$ generates a strongly continuous group on $\HH^{\z}$, we then deduce the following Kirchhoff type formula
	\begin{align*}
		\Upsilon_{\mathscr{H}}i_{-}\mathrm{e}^{\mathrm{i}t\dot{H}^{\z}}u&=\mathrm{i}\mathrm{e}^{-tL_{+}}\Upsilon_{\mathscr{H}}i_{-}u+\mathrm{i}\int_{0}^{t}\mathrm{e}^{-(t-t')L_{+}}\Xi(t')\mathrm{d}t'
	\end{align*}
	which we can rewritten as
	\begin{align}
	\label{Error Kirchhoff}
	\Upsilon_{\mathscr{H}}i_{-}\mathrm{e}^{-\mathrm{i}t\dot{H}^{\z}}u&=\mathrm{i}\mathrm{e}^{-tL_{+}}\left(\Upsilon_{\mathscr{H}}i_{-}u+\mathrm{i}\int_{0}^{+\infty}\mathrm{e}^{t'L_{+}}\Xi(t')\mathrm{d}t'\right)+o_{t\to+\infty}(1)
	\end{align}
	in the $L^{1}(\mathbb{R}_{t},\HH^{\z})$ sense. Finally, write $i_{-}=j_{-}i_{-}$ so that
	\begin{align}
	\label{Error Kirchhoff 2}
		\Upsilon_{\mathscr{H}}i_{-}\mathrm{e}^{\mathrm{i}t\dot{H}^{\z}}u&=\mathrm{i}j_{-}'i_{-}(\mathrm{e}^{\mathrm{i}t\dot{H}^{\z}}u)_{0}+j_{-}\Upsilon_{\mathscr{H}}i_{-}\mathrm{e}^{\mathrm{i}t\dot{H}^{\z}}u.
	\end{align}
	Since $j_{-}'$ is supported near $0$, \cite[Proposition 6.5]{GGH17} shows that the first term above goes to 0 as $t\to+\infty$; the second term above also falls off at the limit using \eqref{Error Kirchhoff} since $j_{-}\mathrm{e}^{-tL_{+}}\phi\to 0$ as $t\to+\infty$ 
	for any $\phi\in\HH^{\z}$. This shows that the expression in \eqref{Error Kirchhoff 2} is smaller than $\varepsilon/2$ for $t\gg0$; back into \eqref{First error Upsilon}, this gives \eqref{To show Upsilon}.
\end{proof}
\begin{proof}[Proof of \eqref{Semi inversion in}]
	Notice that we can define $\boldsymbol{\Omega}_{\mathscr{H}}^{f}$ using $j_{-}$ instead of $i_{-}$ in part 2. of Theorem \ref{Asymptotic completeness, geometric profiles}; this will immediately cancel mixed terms $j_{-}i_{+}$ below\footnote{Otherwise, we have to involve a propagation estimate to make $i_{-}i_{+}\mathrm{e}^{\mathrm{i}t\dot{H}_{\mathcal{H}}^{\z}}$ vanish at the limit $t\to+\infty$.}.
	
	Let first $u\in\Psi_{\mathscr{H}}(\dot{\mathcal{H}}^{1,\z}_{\mathscr{H}}\times\{0\})\cap\mathcal{D}^{\textup{fin},\z}_{\mathscr{H}}$ so that $\boldsymbol{W}_{\mathscr{H}}^{f}u$ is the limit of $\mathrm{e}^{-\mathrm{i}t\dot{H}^{\z}}i_{-}\mathrm{e}^{\mathrm{i}t\dot{H}^{\z}_{\mathcal{H}}}u$ as $t\to+\infty$.	Since $u_{1}=\mathrm{i}L_{\mathscr{H}}$, we obtain as in Remark \ref{In/out remark}
	\begin{align}
	\label{Action on in}
	\big(\mathrm{e}^{\mathrm{i}t\dot{H}^{\z}_{\mathscr{H}}}u\big)_{0}(x)	&=\mathrm{e}^{\mathrm{i}s\z\int_{x}^{x+t}V(x')\mathrm{d}x'}u_{0}(x+t)=\mathrm{e}^{-tL_{\mathscr{H}}}u_{0}(x)
	\end{align}
	where we omit the dependence in $(z,\omega)\in\mathbb{S}^{1}\times\mathbb{S}^{2}$. Set then $\widetilde{u}:=\mathrm{e}^{\mathrm{i}t\dot{H}^{\z}_{\mathscr{H}}}u$; integrating by parts in Kirchhoff formula \eqref{Kirchhoff type formula 2 simplified}, we get:
	\begin{align*}
	&\big(\mathrm{e}^{-\mathrm{i}t\dot{H}_{\mathscr{H}}^{\z}}j_{-}\widetilde{u}\big)_{0}(x)\\
	&=\frac{\mathrm{e}^{-\mathrm{i}s\z tV_{-}}}{2}\left(\sum_{\pm}\mathrm{e}^{\mathrm{i}s\z\int_{x}^{x\pm t}(V(x')-V_{-})\mathrm{d}x'}j_{-}(x\pm t)\widetilde{u}_{0}(x\pm t)\right.\\
	&\qquad\qquad\quad\ \left.+\mathrm{i}\int_{x+t}^{x-t}\mathrm{e}^{\mathrm{i}s\z\int_{x}^{y}(V(x')-V_{-})\mathrm{d}x'}j_{-}(y)\Big(\widetilde{u}_{1}-s\z V_{-}\widetilde{u}_{0}\Big)(y)\mathrm{d}y\right)\\
	%
	%
	%
	%
	&=j_{-}(x-t)u_{0}(x)-\frac{\mathrm{e}^{-\mathrm{i}s\z tV_{-}}}{2}\int_{x+t}^{x-t}\mathrm{e}^{\mathrm{i}s\z\int_{x}^{y}(V(x')-V_{-})\mathrm{d}x'}\mathrm{e}^{\mathrm{i}s\z\int_{y}^{y+t}V(x')\mathrm{d}x'}j_{-}'(y)u_{0}(y+t)\mathrm{d}y.
	\end{align*}
	Since $j_{-}(x-t)\to 1$ as $t\to+\infty$, $j_{-}(x-t)u_{0}(x)=u_{0}(x)$ for $t\gg0$. The integral term vanishes for large $t$ because $j'$ is supported near $0$ whereas $\mathrm{Supp\,}u_{0}(x+t)$ leaves any neighborhood of 0. This means that $\big(\mathrm{e}^{-\mathrm{i}t\dot{H}_{\mathscr{H}}^{\z}}j_{-}\mathrm{e}^{\mathrm{i}t\dot{H}^{\z}_{\mathscr{H}}}u\big)_{0}=u_{0}$ in finite time and then $\big(\mathrm{e}^{-\mathrm{i}t\dot{H}_{\mathscr{H}}^{\z}}j_{-}\mathrm{e}^{\mathrm{i}t\dot{H}^{\z}_{\mathscr{H}}}u\big)_{1}=\mathrm{i}L_{\mathscr{H}}u_{0}$ using \eqref{Action on in}.
	
	By Theorem \ref{Asymptotic completeness, geometric profiles}, $\boldsymbol{W}_{\mathscr{H}}^{f}$ and $\boldsymbol{\Omega}_{\mathscr{H}}^{f}$ are bounded operators so that
	\begin{align*}
		s-\lim_{t\to+\infty}\left(\mathrm{e}^{-\mathrm{i}t\dot{H}^{\z}_{\mathscr{H}}}j_{-}\mathrm{e}^{\mathrm{i}t\dot{H}^{\z}}\right)\left(\mathrm{e}^{-\mathrm{i}t\dot{H}_{\mathscr{H}}^{\z}}i_{-}\mathrm{e}^{\mathrm{i}t\dot{H}^{\z}}\right)&=\boldsymbol{\Omega}_{\mathscr{H}}^{f}\boldsymbol{W}_{\mathscr{H}}^{f}.
	\end{align*}
	The above computations show that the left-hand side above is $\mathds{1}_{\Psi_{\mathscr{H}}(\dot{\mathcal{H}}^{1,\z}_{\mathscr{H}}\times\{0\})}$ (first proved on a dense subspace then extended to $\Psi_{\mathscr{H}}(\dot{\mathcal{H}}^{1,\z}_{\mathscr{H}}\times\{0\})$ by continuity) which entails \eqref{Semi inversion in}.
\end{proof}
%
%
%
%
%
%
%
%
%
\section{Geometric interpretation}
\label{Geometric interpretation}
We provide the geometric interpretation of the scattering associated to the dynamics $\dot{H}_{\mathscr{H}\!/\!\mathscr{I}}$ of Subsection \ref{Comparison dynamics}. We will show how the inverse wave operators of Theorem \ref{Asymptotic completeness, geometric profiles} are related to traces onto the horizons $\mathscr{H}$ and $\mathscr{I}$. We adapt \cite{HaNi04} which deals with Dirac equation in Kerr spacetime.

This section is organized as follows: in Subsection \ref{Energy spaces on the horizons}, we define energy spaces on horizons using the principal null geodesics; we construct in Subsection \ref{The full wave operators} full wave and inverse wave operators using Theorem \ref{Asymptotic completeness, geometric profiles}; Subsection \ref{Inversion of the full wave operators} then shows that they are indeed inverse; in Subsection \ref{Traces on the energy spaces}, we extend the trace operators from smooth compactly supported data to energy spaces as bounded invertible operators; finally, we solve an abstract Goursat problem in Subsection \ref{Solution to the Goursat problem}.
%
%
%
%
\subsection{Energy spaces on the horizons}
\label{Energy spaces on the horizons}
We define in this Subsection the energy spaces on the horizons obtained by transport of the principal null geodesics. This will allow us to define traces in Subsection \ref{Traces on the energy spaces} and extend them as abstract operators acting on energy spaces.

First, we explicit the correspondence between horizons and the initial data slice $\Sigma_{0}$ using principal null geodesics. Recall from Subsection \ref{Principal null geodesics} the extended-star $(t^{\star},z^{\star},\omega^{\star})$ and star extended $(^{\star}t,\,\!^{\star}z,\,\!^{\star}\omega)$ coordinates which describe the horizons:
\begin{align*}
	t^{\star}&=t+T(r),\qquad\qquad ^{\star}t=t-T(r),\\
	z^{\star}&=z+Z(r),\qquad\qquad^{\star}z=z-Z(r),\\
	\omega^{\star}&=\,\!^{\star}\omega=\omega
\end{align*}
with $T=x$ and
\begin{align*}
	T'(r)&=F(r)^{-1},\qquad\qquad Z'(r)=-sV(r)F(r)^{-1}.
\end{align*}
Then
\begin{align*}
	\mathscr{H}^{+}&=\mathbb{R}_{t^{\star}}\times\mathbb{S}^{1}_{z^{\star}}\times\{-r_{-}\}_{r}\times\mathbb{S}^{2}_{\omega},\qquad\qquad\mathscr{I}^{+}=\mathbb{R}_{^{\star}\!t}\times\mathbb{S}^{1}_{^{\star}\!z}\times\{r_{+}\}_{r}\times\mathbb{S}^{2}_{\omega}\\
	%
	%
	\mathscr{H}^{-}&=\mathbb{R}_{^{\star}\!t}\times\mathbb{S}^{1}_{^{\star}\!z}\times\{r_{-}\}_{r}\times\mathbb{S}^{2}_{\omega},\qquad\qquad\mathscr{I}^{-}=\mathbb{R}_{t^{\star}}\times\mathbb{S}^{1}_{z^{\star}}\times\{-r_{+}\}_{r}\times\mathbb{S}^{2}_{\omega}.
\end{align*}
Let $(t_{0},z_0,r_0,\omega_0)\in\M$. Since $t^{\star}$ and $z^{\star}$ (respectively $^{\star}t$ and $^{\star}z$) are constant along $\gamma_{\textup{in}}$ (respectively along $\gamma_{\textup{out}}$), we find
\begin{align*}
	\lim_{r\to r_{-}}\gamma_{\textup{in}}(r)&=(t_{0}+T(r_{0}),z_{0}+Z(r_{0}),-r_{-},\omega_{0})\in\mathscr{H}^{+},\\
	\lim_{r\to r_{+}}\gamma_{\textup{out}}(r)&=\big(t_{0}-T(r_{0}),z-Z(r_{0}),r_{+},\omega_{0}\big)\in\mathscr{I}^{+},\\
	\lim_{r\to r_{-}}\gamma_{\textup{out}}(r)&=\big(t_{0}-T(r_{0}),z-Z(r_{0}),r_{-},\omega_{0}\big)\in\mathscr{H}^{-},\\
	\lim_{r\to r_{+}}\gamma_{\textup{in}}(r)&=\big(t_{0}+T(r_{0}),z+Z(r_{0}),-r_{+},\omega_{0}\big)\in\mathscr{I}^{-}
\end{align*}
for the end points of the principal null geodesics intersecting $(t_{0},z_0,r_0,\omega_0)$ at $r=r_{0}$ (the explicit expressions of $T(r)=x(r)$ and $Z(r)$ are given in \eqref{t^*} and \eqref{z^*}). As $T,Z:\left]r_-,r_+\right[\to\mathbb{R}$ are smooth diffeomorphisms, the applications $\mathscr{F}_{\mathscr{H}}^{\pm}:\mathscr{H}^{\pm}\to\Sigma_{0}$ and $\mathscr{F}_{\mathscr{I}}^{\pm}:\mathscr{I}^{\pm}\to\Sigma_{0}$ defined by
\begin{align*}
	\mathscr{F}_{\mathscr{H}}^{-\!/\!+}\left(\lim_{r\to r_{-}}\gamma_{\textup{out/in}}(r)(p_{0})\right)&:=p_{0},\qquad\qquad\mathscr{F}_{\mathscr{I}}^{-\!/\!+}\left(\lim_{r\to r_{+}}\gamma_{\textup{in/out}}(r)(p_{0})\right):=p_{0}
\end{align*}
for all $p_{0}\in\Sigma_{0}$ are well-defined diffeomorphisms which identify end points on the future/past horizons to the initial point on $\Sigma_{0}$.

Let us turn to the definition of the energy spaces on the horizons. We define the \textit{asymptotic future/past energy spaces} $\dot{\mathscr{E}}_{\pm}^{\z}:=((\mathscr{F}_{\mathscr{H}}^{\pm})^{-1})^{*}\dot{\HH}^{1}_{\mathscr{H}}\times((\mathscr{F}_{\mathscr{I}}^{\pm})^{-1})^{*}\dot{\HH}^{1}_{\mathscr{I}}$ and their restrictions $\dot{\mathscr{E}}^{\z}_{\pm}$ to $\ker(\mathrm{i}\partial_{z}+\z)$ endowed with the norms
\begin{align*}
	\|(\xi,\zeta)\|_{\dot{\mathscr{E}}_{\pm}}^{2}&:=\big\|(L_{\mathscr{H}}+\mathrm{i}k_{\mathscr{H}})\big((\mathscr{F}^{\pm}_{\mathscr{H}})^{-1}\big)^{*}\xi\big\|_{\mathcal{H}}^{2}+\big\|(L_{\mathscr{I}}+\mathrm{i}k_{\mathscr{I}})\big((\mathscr{F}^{\pm}_{\mathscr{I}})^{-1}\big)^{*}\zeta\big\|_{\mathcal{H}}^{2}.
\end{align*}
Using the coordinates recalled in Subsection \ref{Energy spaces on the horizons}, we can write
\begin{align*}
\partial_{x}&=\partial_{t^{\star}}-sV(r)\partial_{z^{\star}},\qquad\quad\partial_{z}=\partial_{z^{\star}},\\
\partial_{x}&=-\partial_{^{\star}\!t}+sV(r)\partial_{^{\star}\!z},\qquad\quad\partial_{z}=\partial_{^{\star}\!z}.
\end{align*}
Using $L_{\mathscr{H}}=-\partial_{x}-sV\partial_{z}$ and $L_{\mathscr{I}}=\partial_{x}-sV\partial_{z}$, we explicitly get:
\begin{align*}
	\big\|(L_{\mathscr{H}}+\mathrm{i}k_{\mathscr{H}})\big((\mathscr{F}^{+}_{\mathscr{H}})^{-1}\big)^{*}\xi\big\|_{\mathcal{H}}^{2}&=\int_{\mathbb{R}_{t^{\star}}\times\mathbb{S}^{1}_{z^{\star}\!}\times\mathbb{S}^{2}_{\omega}}\Big(\partial_{t^{\star}}\xi-sV_{-}\partial_{z^{\star}}\xi\Big)^{2}(t^{\star},z^{\star},\omega)\mathrm{d}t^{\star}\mathrm{d}z^{\star}\mathrm{d}\omega,\\
	\big\|(L_{\mathscr{I}}+\mathrm{i}k_{\mathscr{I}})\big((\mathscr{F}^{+}_{\mathscr{i}})^{-1}\big)^{*}\zeta\big\|_{\mathcal{H}}^{2}&=\int_{\mathbb{R}_{^{\star}\!t}\times\mathbb{S}^{1}_{\,\!^{\star}\!z\!}\times\mathbb{S}^{2}_{\omega}}\Big(\partial_{^{\star}\;\!\!t}\zeta-sV_{+}\partial_{^{\star}\!z}\zeta\Big)^{2}(^{\star}t,\,\!^{\star}z,\omega)\mathrm{d}^{\star}t\mathrm{d}^{\star}z\mathrm{d}\omega.
\end{align*}
Similar formulas hold on $\mathscr{H}^{-}$ and $\mathscr{I}^{-}$.
\begin{remark}
\label{Remark energies horizons}
	The asymptotic energies on the horizons are nothing but the flux of the Killing generators $X_{-}:=\partial_{t^{\star}}-sV_{-}\partial_{z^{\star}}$, $X_{+}:=\partial_{^{\star}\;\!\!t}-sV_{+}\partial_{^{\star}\!z}$ defined in \eqref{Killing generators} through the corresponding horizon.
\end{remark}
%
%
%
%
%
%
\subsection{The full wave operators}
\label{The full wave operators}
The operators in Theorem \ref{Asymptotic completeness, geometric profiles} are not inverse despite their name because the cut-offs $i_{\pm}$ cancel outgoing/incoming data. In this Subsection, we construct full wave and inverse wave operators which encode scattering in both the ends of the spacetime.

Let $\Pi_{0}:\dot{\mathcal{H}}^{1,\z}_{\mathscr{H}}\times\dot{\mathcal{H}}^{1,\z}_{\mathscr{I}}\ni(u_0,u_1)\mapsto(u_0,0)\in\dot{\mathcal{H}}^{1,\z}_{\mathscr{H}}$, $\Pi_{1}:\dot{\mathcal{H}}^{1,\z}_{\mathscr{H}}\times\dot{\mathcal{H}}^{1,\z}_{\mathscr{I}}\ni(u_0,u_1)\mapsto(0,u_1)\in\dot{\mathcal{H}}^{1,\z}_{\mathscr{I}}$ and $\widehat{\Pi}_{0}:\dot{\mathcal{E}}^{\z}_{\mathscr{H}}\ni(u_{0},u_{1})\mapsto(u_{0},0)\in\dot{\mathcal{H}}^{1,\z}_{\mathscr{H}}\times\{0\}$, $\widehat{\Pi}_{1}:\dot{\mathcal{E}}^{\z}_{\mathscr{I}}\ni(u_{0},u_{1})\mapsto(0,u_{0})\in\{0\}\times\dot{\mathcal{H}}^{1,\z}_{\mathscr{I}}$. For all $t\in\mathbb{R}$, we define the following operators:
\begin{align}
\label{W(t)}
	W(t)&:=\sqrt{2}\mathrm{e}^{-\mathrm{i}t\dot{H}}i_{-}\Psi_{\mathscr{H}}\mathrm{e}^{\mathrm{i}t\dot{\mathbb{H}}_{\mathscr{H}}}\Pi_{0}+\sqrt{2}\mathrm{e}^{-\mathrm{i}t\dot{H}}i_{+}\Psi_{\mathscr{I}}\mathrm{e}^{\mathrm{i}t\dot{\mathbb{H}}_{\mathscr{I}}}\Pi_{1},\\
\label{Omega(t)}
	\Omega(t)&:=\mathrm{e}^{-\mathrm{i}t\dot{\mathbb{H}}_{\mathscr{H}}}\widehat{\Pi}_{0}j_{-}\mathrm{e}^{\mathrm{i}t\dot{H}^{\z}}+\mathrm{e}^{-\mathrm{i}t\dot{\mathbb{H}}_{\mathscr{I}}}\widehat{\Pi}_{1}j_{+}\mathrm{e}^{\mathrm{i}t\dot{H}^{\z}}.
	%
\end{align}
Recall that $j_{-\!/\!+}\in\mathcal{B}\big(\dot{\mathcal{E}}^{\z},\dot{\mathcal{E}}_{\mathscr{H}\!/\!\mathscr{I}}^{\z}\big)$ by Lemma \ref{Injection lemma} so that \eqref{Omega(t)} makes sense.
\begin{lemma}
	\label{Extension W(t)}
	Let $\z\in\mathbb{Z}\setminus\{0\}$. There exists $s_{0}>0$ such that for all $s\in\left]-s_{0},s_{0}\right[$, the following holds:
	\begin{enumerate}
		\item For all $u=(u_{0},u_{1})$ such that $(u_{0},0)\in\Psi_{\mathscr{H}}^{-1}\big(\mathcal{D}^{\textup{fin},\z}_{\mathscr{H}}\big)$ and $(0,u_{1})\in\Psi_{\mathscr{I}}^{-1}\big(\mathcal{D}^{\textup{fin},\z}_{\mathscr{I}}\big)$, the limits
		\begin{align*}
			\mathbb{W}^{\pm}u&:=\lim_{t\to\pm\infty}W(t)u=\boldsymbol{W}_{\mathscr{H}}^{f/p}\Psi_{\mathscr{H}}\Pi_{0}u+\boldsymbol{W}_{\mathscr{I}}^{f/p}\Psi_{\mathscr{I}}\Pi_{1}u
		\end{align*}
		exist in $\dot{\mathcal{E}}^{\z}$. The operators $\mathbb{W}^{\pm}$ extend to bounded operators $\boldsymbol{W}_{\pm}\in\mathcal{B}\big(\dot{\mathcal{H}}^{1,\z}_{\mathscr{H}}\times\dot{\mathcal{H}}^{1,\z}_{\mathscr{I}},\dot{\mathcal{E}}^{\z}\big)$.\\
		We call $\mathbb{W}^{\pm}$ the \textup{full future/past wave operators}.
		\item The following strong limits exist:
		\begin{align*}
			\OO^{\pm}&:=\textup{s}-\lim_{t\to\pm\infty}\Omega(t)=\Psi_{\mathscr{H}}^{-1}\boldsymbol{\Omega}^{f/p}_{\mathscr{H}}+\Psi_{\mathscr{I}}^{-1}\boldsymbol{\Omega}^{f/p}_{\mathscr{I}}\in\mathcal{B}\big(\dot{\mathcal{E}}^{\z},\dot{\mathcal{H}}^{1,\z}_{\mathscr{H}}\times\dot{\mathcal{H}}^{1,\z}_{\mathscr{I}}\big).
		\end{align*}
		The operators $\OO^{\pm}$ are the \textup{full future/past inverse wave operators}.
	\end{enumerate}
\end{lemma}
\begin{proof}
\begin{enumerate}
	\item Let $\pi_{j}:(u_0,u_1)\mapsto u_j$ be the projection onto the $j$-th component, $j\in\{0,1\}$. The existence of the strong limits on $\pi_{0}\left(\Psi_{\mathscr{H}}^{-1}\big(\mathcal{D}^{\textup{fin},\z}_{\mathscr{H}}\big)\right)\times\pi_{1}\left(\Psi_{\mathscr{I}}^{-1}\big(\mathcal{D}^{\textup{fin},\z}_{\mathscr{I}}\big)\right)$ follows from part 1. of Theorem \ref{Asymptotic completeness, geometric profiles} as $\Psi_{\mathscr{H}\!/\!\mathscr{I}}\mathrm{e}^{\mathrm{i}t\dot{\mathbb{H}}_{\mathscr{H}\!/\!\mathscr{I}}}=\mathrm{e}^{\mathrm{i}t\dot{H}_{\mathscr{H}\!/\!\mathscr{I}}}\Psi_{\mathscr{H}\!/\!\mathscr{I}}$. To define the operator on any $u\in\dot{\mathcal{H}}^{1,\z}_{\mathscr{H}}\times\dot{\mathcal{H}}^{1,\z}_{\mathscr{I}}$, it suffices to observe that
	\begin{align}
	\label{Density HxH}
	\dot{\mathcal{H}}^{1,\z}_{\mathscr{H}}\times\dot{\mathcal{H}}^{1,\z}_{\mathscr{I}}&=\overline{\pi_{0}\left(\Psi_{\mathscr{H}}^{-1}\big(\mathcal{D}^{\textup{fin},\z}_{\mathscr{H}}\big)\right)\times\pi_{1}\left(\Psi_{\mathscr{I}}^{-1}\big(\mathcal{D}^{\textup{fin},\z}_{\mathscr{I}}\big)\right)}^{_{\|\cdot\|_{\dot{\mathcal{H}}^{1,\z}_{\mathscr{H}}\times\dot{\mathcal{H}}^{1,\z}_{\mathscr{I}}}}}.
	\end{align}
	This follows from the facts that $\mathcal{D}^{\textup{fin},\z}_{\mathscr{H}\!/\!\mathscr{I}}$ are dense in $\dot{\mathcal{E}}_{\mathscr{H}\!/\!\mathscr{I}}^{\z}$ (\textit{cf.} Lemma \ref{Density in/out}), that $\Psi_{\mathscr{H}\!/\!\mathscr{I}}$ are homeomorphisms (\textit{cf.} Lemma \ref{Transformation Psi}) and that projections are continuous with respect to the product topology. We then conclude using that $\boldsymbol{W}_{\mathscr{H}}^{f/p}$ and $\boldsymbol{W}_{\mathscr{I}}^{f/p}$ have continuous extensions by part 1. of Theorem \ref{Asymptotic completeness, geometric profiles}.
	\item Let us write
	\begin{align*}
	\Omega(t)&=\mathrm{e}^{-\mathrm{i}t\dot{\mathbb{H}}^{\z}_{\mathscr{H}}}\widehat{\Pi}_{0}\mathrm{e}^{\mathrm{i}t\dot{H}^{\z}_{\mathscr{H}}}\left(\mathrm{e}^{-\mathrm{i}t\dot{H}^{\z}_{\mathscr{H}}}j_{-}\mathrm{e}^{\mathrm{i}t\dot{H}^{\z}}\right)+\mathrm{e}^{-\mathrm{i}t\dot{\mathbb{H}}^{\z}_{\mathscr{I}}}\widehat{\Pi}_{1}\mathrm{e}^{\mathrm{i}t\dot{H}^{\z}_{\mathscr{I}}}\left(\mathrm{e}^{-\mathrm{i}t\dot{H}^{\z}_{\mathscr{I}}}j_{+}\mathrm{e}^{\mathrm{i}t\dot{H}^{\z}}\right)
	\end{align*}
	where
	\begin{align*}
	\mathrm{e}^{-\mathrm{i}t\dot{\mathbb{H}}^{\z}_{\mathscr{H}}}\widehat{\Pi}_{0}\mathrm{e}^{\mathrm{i}t\dot{H}^{\z}_{\mathscr{H}}}&\in\mathcal{B}\big(\dot{\mathcal{E}}^{\z},\dot{\mathcal{H}}^{1,\z}_{\mathscr{H}}\times\dot{\mathcal{H}}^{1,\z}_{\mathscr{H}}\big),\qquad\qquad\mathrm{e}^{-\mathrm{i}t\dot{\mathbb{H}}^{\z}_{\mathscr{I}}}\widehat{\Pi}_{1}\mathrm{e}^{\mathrm{i}t\dot{H}^{\z}_{\mathscr{I}}}\in\mathcal{B}\big(\dot{\mathcal{E}}^{\z},\dot{\mathcal{H}}^{1,\z}_{\mathscr{I}}\times\dot{\mathcal{H}}^{1,\z}_{\mathscr{I}}\big)
	\end{align*}
	uniformly in $t\in\mathbb{R}$. Theorem \ref{Asymptotic completeness, geometric profiles} then implies that
	\begin{align*}
	\Omega(t)\begin{pmatrix}
	\phi_{0}\\\phi_{1}
	\end{pmatrix}&=\mathrm{e}^{-\mathrm{i}t\dot{\mathbb{H}}^{\z}_{\mathscr{H}}}\widehat{\Pi}_{0}\mathrm{e}^{\mathrm{i}t\dot{H}^{\z}_{\mathscr{H}}}\boldsymbol{\Omega}^{f}_{\mathscr{H}}\begin{pmatrix}
	\phi_{0}\\\phi_{1}
	\end{pmatrix}+\mathrm{e}^{-\mathrm{i}t\dot{\mathbb{H}}^{\z}_{\mathscr{I}}}\widehat{\Pi}_{1}\mathrm{e}^{\mathrm{i}t\dot{H}^{\z}_{\mathscr{I}}}\boldsymbol{\Omega}^{f}_{\mathscr{I}}\begin{pmatrix}
	\phi_{0}\\\phi_{1}
	\end{pmatrix}+o_{t\to+\infty}(1).
	\end{align*}
	Next, Proposition \ref{Geo wave op prop} shows that $\boldsymbol{\Omega}^{f}_{\mathscr{H}}\dot{\mathcal{E}}^{\z}\subset\Psi_{\mathscr{H}}\big(\dot{\mathcal{H}}^{1,\z}_{\mathscr{H}}\times\{0\}\big)$ and $\boldsymbol{\Omega}^{f}_{\mathscr{I}}\dot{\mathcal{E}}^{\z}\subset\Psi_{\mathscr{I}}\big(\{0\}\times\dot{\mathcal{H}}^{1,\z}_{\mathscr{I}}\big)$. Since $\Psi_{\mathscr{H}}\widehat{\Pi}_{0}$ projects onto states $u$ satisfying $u_1=\mathrm{i}L_{\mathscr{H}}u_0$, it comes:
	\begin{align*}
		\mathrm{e}^{-\mathrm{i}t\dot{\mathbb{H}}^{\z}_{\mathscr{H}}}\widehat{\Pi}_{0}\mathrm{e}^{\mathrm{i}t\dot{H}^{\z}_{\mathscr{H}}}\boldsymbol{\Omega}_{\mathscr{H}}^{f}\begin{pmatrix}
		\phi_{0}\\\phi_{1}
		\end{pmatrix}&=\mathrm{e}^{-\mathrm{i}t\dot{\mathbb{H}}^{\z}_{\mathscr{H}}}\widehat{\Pi}_{0}\mathrm{e}^{\mathrm{i}t\dot{H}^{\z}_{\mathscr{H}}}\Psi_H\widehat{\Pi}_{0}\Psi_{\mathscr{H}}^{-1}\boldsymbol{\Omega}_{\mathscr{H}}^{f}\begin{pmatrix}
		\phi_{0}\\\phi_{1}
		\end{pmatrix}\\
		&=\mathrm{e}^{-\mathrm{i}t\dot{\mathbb{H}}^{\z}_{\mathscr{H}}}\widehat{\Pi}_{0}\Psi_{\mathscr{H}}\mathrm{e}^{\mathrm{i}t\dot{\mathbb{H}}^{\z}_H}\widehat{\Pi}_{0}\Psi_{\mathscr{H}}^{-1}\boldsymbol{\Omega}_{\mathscr{H}}^{f}\begin{pmatrix}
		\phi_{0}\\\phi_{1}
		\end{pmatrix}\\
		&=\begin{pmatrix}
		\mathrm{e}^{tL_{\mathscr{H}}^{\z}}&0\\0&\mathrm{e}^{tL_{\mathscr{H}}^{\z}}
		\end{pmatrix}\widehat{\Pi}_{0}\Psi_{\mathscr{H}}\begin{pmatrix}
		\mathrm{e}^{-tL_{\mathscr{H}}^{\z}}&0\\0&\mathrm{e}^{-tL_{\mathscr{H}}^{\z}}
		\end{pmatrix}\widehat{\Pi}_{0}\Psi_{\mathscr{H}}^{-1}\boldsymbol{\Omega}_{\mathscr{H}}^{f}\begin{pmatrix}
		\phi_{0}\\\phi_{1}
	\end{pmatrix}\\
		&=\widehat{\Pi}_{0}\Psi_{\mathscr{H}}\widehat{\Pi}_{0}\Psi_{\mathscr{H}}^{-1}\boldsymbol{\Omega}_{\mathscr{H}}^{f}\begin{pmatrix}
		\phi_{0}\\\phi_{1}
		\end{pmatrix}\\
		&=\Psi_{\mathscr{H}}^{-1}\boldsymbol{\Omega}_{\mathscr{H}}^{f}\begin{pmatrix}
		\phi_{0}\\\phi_{1}
		\end{pmatrix}.
	\end{align*}
	Similarly, we have
	\begin{align*}
		\Psi_{\mathscr{I}}^{-1}\mathrm{e}^{-\mathrm{i}t\dot{H}^{\z}_{\mathscr{I}}}\Psi_{\mathscr{I}}\widehat{\Pi}_{1}\mathrm{e}^{\mathrm{i}t\dot{H}^{\z}_{\mathscr{I}}}\boldsymbol{\Omega}^{f}_{\mathscr{I}}\begin{pmatrix}
		\phi_{0}\\\phi_{1}
		\end{pmatrix}&=\Psi_{\mathscr{I}}^{-1}\boldsymbol{\Omega}^{f}_{\mathscr{I}}\begin{pmatrix}
		\phi_{0}\\\phi_{1}
		\end{pmatrix}
	\end{align*}
	%
	%
	%
	%
	%
	%
	%
	%
	%
	whence finally:
	\begin{align*}
	\Omega(t)\begin{pmatrix}
	\phi_{0}\\\phi_{1}
	\end{pmatrix}&=\Psi_{\mathscr{H}}^{-1}\boldsymbol{\Omega}^{f}_{\mathscr{H}}\begin{pmatrix}
	\phi_{0}\\\phi_{1}
	\end{pmatrix}+\Psi_{\mathscr{I}}^{-1}\boldsymbol{\Omega}^{f}_{\mathscr{I}}\begin{pmatrix}
	\phi_{0}\\\phi_{1}
	\end{pmatrix}+o_{t\to+\infty}(1).
	\end{align*}
\end{enumerate}
\end{proof}
\begin{remark}
\label{Remark kernels}
	We have
	\begin{align}
	\label{Kernels}
		\ker\OO^{\pm}&=\ker\boldsymbol{\Omega}^{f/p}_{\mathscr{H}}\cap\ker\boldsymbol{\Omega}^{f/p}_{\mathscr{I}}.
	\end{align}
	Indeed, let $u=(u_{0},u_{1})\in\dot{\mathcal{E}}^{\z}$.	Using that $\Psi_{\mathscr{H}\!/\!\mathscr{I}}^{-1}$ are isometries by Lemma \ref{Transformation Psi}, we can write
	\begin{align}
	\label{Injectivity inequality}
	\left\|\OO^{\pm}u\right\|_{\dot{\mathcal{H}}^{1,\z}_{\mathscr{H}}\times\dot{\mathcal{H}}^{1,\z}_{\mathscr{I}}}^{2}&=\left\|\big(\boldsymbol{\Omega}^{f/p}_{\mathscr{H}}u\big)_{0}\right\|_{\dot{\mathcal{H}}^{1,\z}_{\mathscr{H}}}^{2}+\left\|\big(\boldsymbol{\Omega}^{f/p}_{\mathscr{I}}u\big)_{0}\right\|_{\dot{\mathcal{H}}^{1,\z}_{\mathscr{I}}}^{2}\leq\left\|\boldsymbol{\Omega}^{f/p}_{\mathscr{H}}u\right\|_{\dot{\mathcal{E}}^{\z}_{\mathscr{H}}}^{2}+\left\|\boldsymbol{\Omega}^{f/p}_{\mathscr{I}}u\right\|_{\dot{\mathcal{E}}^{\z}_{\mathscr{I}}}^{2}
	\end{align}
	so that $\ker\OO^{\pm}\supset\ker\boldsymbol{\Omega}^{f/p}_{\mathscr{H}}\cap\ker\boldsymbol{\Omega}^{f/p}_{\mathscr{I}}$. Since $\boldsymbol{\Omega}_{\mathscr{H}}^{f/p}\dot{\mathcal{E}}^{\z}\subset\Psi_{\mathscr{H}}\big(\dot{\mathcal{H}}^{1,\z}_{\mathscr{H}}\times\{0\}\big)$ and $\boldsymbol{\Omega}_{\mathscr{I}}^{f/p}\dot{\mathcal{E}}^{\z}\subset\Psi_{\mathscr{I}}\big(\{0\}\times\dot{\mathcal{H}}^{1,\z}_{\mathscr{I}}\big)$ by Proposition \ref{Geo wave op prop},  $\big(\boldsymbol{\Omega}^{f/p}_{\mathscr{H}\!/\!\mathscr{I}}u\big)_{0}=0$ entails $\big(\boldsymbol{\Omega}^{f/p}_{\mathscr{H}\!/\!\mathscr{I}}u\big)_{1}=0$. This means that the left-hand side in \eqref{Injectivity inequality} vanishes if and only if the right-hand side does, and \eqref{Kernels} follows.
	
	We will show in the proof of Theorem \ref{Inversion of full wave operators} that both the sets in \eqref{Kernels} are actually trivial.
\end{remark}
%
%
%
%
%
\subsection{Inversion of the full wave operators}
\label{Inversion of the full wave operators}
In this Subsection, we show that the full wave operators and the full inverse wave operators are indeed inverses in the energy spaces.
\begin{lemma}
	\label{Difference norms}
	Let $\z\in\mathbb{Z}\setminus\{0\}$ and pick\footnote{Recall as in the proof of Theorem \ref{Asymptotic completeness, geometric profiles} that there is no restriction on the support of $\chi$ because we assume $s$ small enough so that \cite[Theorem 3.8]{Be18} implies that there is no resonance on $\mathbb{R}$; in the general case, $\chi$ must cancel in a neighborhood of the real resonances.} $\chi\in\mathcal{C}^{\infty}_{\mathrm{c}}(\mathbb{R})$. For all $u\in\dot{\mathcal{E}}^{\z}$,
	\begin{align*}
	\lim_{|t|\to+\infty}\left|\big\|i_{-\!/\!+}\mathrm{e}^{\mathrm{i}t\dot{H}^{\z}}\chi(\dot{H}^{\z})u\big\|_{\dot{\mathcal{E}}^{\z}}-\big\|i_{-\!/\!+}\mathrm{e}^{\mathrm{i}t\dot{H}^{\z}}\chi(\dot{H}^{\z})u\big\|_{\dot{\mathcal{E}}_{\mathscr{H}\!/\!\mathscr{I}}^{\z}}\right|&=0.
	\end{align*}
\end{lemma}
\begin{proof}
	We only treat the $-,\mathscr{H}$ case. As in the proof of the existence of the future inverse wave operator in Theorem \ref{Asymptotic completeness, geometric profiles} (\textit{cf.} Subsection \ref{Proof of Theorem {Asymptotic completeness, geometric profiles}}), we can assume that $-\Delta_{\mathbb{S}^{2}}u=\sum_{0\leq \ell\leq\ell_{0}}\ell(\ell+1)u$. Set $v:=\mathrm{e}^{\mathrm{i}t\dot{H}^{\z}}u\in\dot{\mathcal{E}}^{\z}$. Then
	\begin{align*}
	\left|\|i_{-}v\|_{\dot{\mathcal{E}}^{\z}_{\mathscr{H}}}^{2}-\|i_{-}v\|_{\dot{\mathcal{E}}^{\z}}^{2}\right|&\leq\left|\big\langle(h_{\mathscr{H}}-h_{0})i_{-}v_{0},i_{-}v_{0}\big\rangle_{\HH^{\z}}\right|+\left|\big\langle i_{-}(v_{1}-\mathrm{i}s\z V_{-}v_{0}),i_{-}\mathrm{i}(V-V_{-})v_{0}\big\rangle_{\HH^{\z}}\right|\\
	&+\left|\big\langle i_{-}\mathrm{i}(V-V_{-})v_{0},i_{-}(v_{1}-\mathrm{i}s\z Vv_{0})\big\rangle_{\HH^{\z}}\right|\\[1.5mm]
	&\lesssim\left|\big\langle(f(x)+(V-V_{-})\partial_{x})i_{-}v_{0},i_{-}v_{0}\big\rangle_{\HH^{\z}}\right|+\big\|v_{1}-\mathrm{i}s\z V_{-}v_{0}\big\|_{\HH^{\z}}\big\|i_{-}(V-V_{-})v_{0}\big\|_{\HH^{\z}}\\
	&+\big\|i_{-}(V-V_{-})v_{0}\|_{\HH^{\z}}\big\|v_{1}-\mathrm{i}s\z Vv_{0}\big\|_{\HH^{\z}}\\[1.5mm]
	&\lesssim\left|\big\langle g(x)i_{-}v_{0},i_{-}v_{0}\big\rangle_{\HH^{\z}}\right|+\|\partial_{x}v_{0}\|_{\HH^{\z}}\|i_{-}(V-V_{-})v_{0}\big\|_{\HH^{\z}}\\
	&+\big\|i_{-}(V-V_{-})v_{0}\|_{\HH^{\z}}\big\|v_{1}-\mathrm{i}s\z Vv_{0}\big\|_{\HH^{\z}}
	\end{align*}
	with $|f(x)|,|g(x)|\lesssim(1+\ell_{0}^{3})\mathrm{e}^{-2\kappa|x|}$ and $\kappa:=\min\{\kappa_{-},|\kappa_{+}|\}$ (as in \eqref{Rest AS abstract}, $\lesssim$ is $\ell_{0}$-independent but depends on $\z$). It follows
	\begin{align*}
	\left|\|i_{-}v\|_{\dot{\mathcal{E}}^{\z}_{\mathscr{H}}}^{2}-\|i_{-}v\|_{\dot{\mathcal{E}}^{\z}}^{2}\right|&\lesssim\left|\big\langle w^{2}v_{0},v_{0}\big\rangle_{\HH^{\z}}\right|+\|v\|_{\dot{\mathcal{E}}^{\z}}\|w^{2}v_{0}\big\|_{\HH^{\z}}
	\end{align*}
	where $w(r(x))=\sqrt{(r(x)-r_{-})(r_{+}-r(x))}=\mathcal{O}_{|x|\to+\infty}(\mathrm{e}^{-2\kappa|x|})$. The $\HH^{\z}$ norms make sense thanks to Hardy type inequality \cite[Lemma 9.5]{GGH17}:
	\begin{align*}
	\left|\|i_{-}v\|_{\dot{\mathcal{E}}^{\z}_{\mathscr{H}}}^{2}-\|i_{-}v\|_{\dot{\mathcal{E}}^{\z}}^{2}\right|&\lesssim\|h_{0}^{1/2}w^{1/2}v_{0}\|_{\HH^{\z}}^{2}+\|v\|_{\dot{\mathcal{E}^{\z}}}\|h_{0}^{1/2}w^{1/2}v_{0}\big\|_{\HH^{\z}}\\
	&\lesssim\|w^{1/2}v\|_{\dot{\mathcal{E}}^{\z}}^{2}+\|v\|_{\dot{\mathcal{E}}^{\z}}\|w^{1/2}v\|_{\dot{\mathcal{E}}^{\z}}\\
	&\lesssim\|w^{1/2}v\|_{\dot{\mathcal{E}}^{\z}}.
	\end{align*}
	Fix now $\varepsilon>0$ arbitrarily small and pick $(\phi_{n})_{\in\mathbb{R}}$ a sequence of smooth compactly supported functions such that $\phi_{n}\to u$ in $\dot{\mathcal{E}}^{\z}$ as $n\to+\infty$. We have
	\begin{align*}
	\|w^{1/2}\mathrm{e}^{\mathrm{i}t\dot{H}^{\z}}u\|_{\dot{\mathcal{E}}^{\z}}&\lesssim\|w^{1/2}\mathrm{e}^{\mathrm{i}t\dot{H}^{\z}}(\phi_{n}-u)\|_{\dot{\mathcal{E}}^{\z}}+\|w^{1/2}\mathrm{e}^{\mathrm{i}t\dot{H}^{\z}}\phi_{n}\|_{\dot{\mathcal{E}}^{\z}}.
	\end{align*}
	Fix $N\gg0$ so that
	\begin{align*}
	\|w^{1/2}\mathrm{e}^{\mathrm{i}t\dot{H}^{\z}}(\phi_{N}-u)\|_{\dot{\mathcal{E}}^{\z}}&<\varepsilon/2
	\end{align*}
	then apply \cite[Proposition 6.7]{GGH17} to get
	\begin{align*}
		\big\|w^{1/2}\mathrm{e}^{\mathrm{i}t\dot{H}^{\z}}\chi(\dot{H}^{\z})\phi_{N}\big\|_{\dot{\mathcal{E}}^{\z}}&\leq\big\|w^{1/2}\mathrm{e}^{\mathrm{i}t\dot{H}^{\z}}\chi(\dot{H}^{\z})w^{1/2}\big\|_{\mathcal{B}(\dot{\mathcal{E}}^{\z})}\big\|w^{-1/2}\phi_{N}\big\|_{\dot{\mathcal{E}}^{\z}}\leq C_{N}\langle t\rangle^{-1}
	\end{align*}
	with $C_{N}>0$ depending on the support of $\phi_{N}$. Choose $t\gg0$ (depending on $N$) so that $C_{N}\langle t\rangle^{-1}<\varepsilon/2$. Hence, for $t$ large enough, we have
	\begin{align*}
	\left|\big\|i_{-}\mathrm{e}^{\mathrm{i}t\dot{H}^{\z}}\chi(\dot{H}^{\z})u\big\|_{\dot{\mathcal{E}}^{\z}}-\big\|i_{-}\mathrm{e}^{\mathrm{i}t\dot{H}^{\z}}\chi(\dot{H}^{\z})u\big\|_{\dot{\mathcal{E}}_{\mathscr{H}}^{\z}}\right|&<C\varepsilon
	\end{align*}
	for some $C>0$ (independent of $\phi_{N}$). This completes the proof.
\end{proof}
\begin{theorem}[Inversion of the full wave operators]
	\label{Inversion of full wave operators}
	Let $\z\in\mathbb{Z}\setminus\{0\}$. There exists $s_{0}>0$ such that for all $s\in\left]-s_{0},s_{0}\right[$,
	\begin{align}
	\label{OmegaW}
		\OO^{\pm}\mathbb{W}^{\pm}&=\mathds{1}_{\dot{\mathcal{H}}_{\mathscr{H}}^{1,\z}\times\dot{\mathcal{H}}_{\mathscr{I}}^{1,\z}},\\
	\label{WOmega}
		\mathbb{W}^{\pm}\OO^{\pm}&=\mathds{1}_{\dot{\mathcal{E}}^{\z}}.
	\end{align}
\end{theorem}
\begin{proof}
	Let us show \eqref{OmegaW}. Since $\sqrt{2}\Psi_{\mathscr{H}}\Pi_{0}$ (respectively $\sqrt{2}\Psi_{\mathscr{I}}\Pi_{1}$) is the projection onto $\Psi_{\mathscr{H}}\big(\dot{\mathcal{H}}^{1,\z}_{\mathscr{H}}\times\{0\}\big)$ (respectively onto $\Psi_{\mathscr{H}}\big(\{0\}\times\dot{\mathcal{H}}^{1,\z}_{\mathscr{I}}\big)$), this identity directly follows from Proposition \ref{Geo wave op prop}.

Let us show \eqref{WOmega}. It is sufficient to show that $\OO^{\pm}$ is one-to-one, since then the right-inverse is also a left-inverse\footnote{Indeed, assume that $AB=\mathds{1}$ and $A$ has been shown to be one-to-one. Then $A=(AB)A$ and the injectivity allows us to simplify the equation on its left, that is $\mathds{1}=BA$.}. By Remark \ref{Remark kernels}, it is sufficient to show that
	\begin{align}
	\label{Sufficient condition kernel}
		\ker\boldsymbol{\Omega}^{f/p}_{\mathscr{H}\!/\!\mathscr{I}}&=\{0\}.
	\end{align}
	Let $u\in\dot{\mathcal{E}}^{\z}$ such that
	\begin{align}
	\label{Assumption injectivity}
		\boldsymbol{\Omega}^{f/p}_{\mathscr{H}}u&=\boldsymbol{\Omega}^{f/p}_{\mathscr{I}}u=0.
	\end{align}
	Pick $\varepsilon>0$, $t\in\mathbb{R}$ and $\chi\in\mathcal{C}^{\infty}_{\mathrm{c}}(\mathbb{R})$ as in Lemma \ref{Difference norms} then write:
	\begin{align}
	\label{Injectivity first inequality}
		\|u\|_{\dot{\mathcal{E}}^{\z}}&\lesssim\big\|\chi(\dot{H}^{\z})u\big\|_{\dot{\mathcal{E}}^{\z}}+\big\|(1-\chi(\dot{H}^{\z}))u\big\|_{\dot{\mathcal{E}}^{\z}},
	\end{align}
	\begin{align*}
		\big\|\chi(\dot{H}^{\z})u\big\|_{\dot{\mathcal{E}}^{\z}}&\lesssim\big\|\mathrm{e}^{\mathrm{i}t\dot{H}^{\z}}\chi(\dot{H}^{\z})u\big\|_{\dot{\mathcal{E}}^{\z}}\\
		&\lesssim\big\|(i_{-}^{2}+i_{+}^{2})\mathrm{e}^{\mathrm{i}t\dot{H}^{\z}}\chi(\dot{H}^{\z})u\big\|_{\dot{\mathcal{E}}^{\z}}\\
		&\lesssim\big\|i_{-}\mathrm{e}^{\mathrm{i}t\dot{H}^{\z}}\chi(\dot{H}^{\z})u\big\|_{\dot{\mathcal{E}}^{\z}}+\big\|i_{+}\mathrm{e}^{\mathrm{i}t\dot{H}^{\z}}\chi(\dot{H}^{\z})u\big\|_{\dot{\mathcal{E}}^{\z}}.
	\end{align*}
	We have used above the boundedness on $\dot{\mathcal{E}}^{\z}$ of $\mathrm{e}^{\mathrm{i}t\dot{H}^{\z}}$ (\textit{cf.} Theorem \ref{Uniform boundedness of the evolution}) and $i_{\pm}$ (\textit{cf.} \cite[Lemma 5.4]{GGH17}). Next, using Lemma \ref{Difference norms} as well as the unitarity of $\mathrm{e}^{\mathrm{i}t\dot{H}_{\mathscr{H}\!/\!\mathscr{I}}^{\z}}$, we get for $|t|\gg 0$:
	\begin{align}
	\label{Injection inequality 2}
		\big\|\chi(\dot{H}^{\z})u\big\|_{\dot{\mathcal{E}}^{\z}}&\lesssim\big\|i_{-}\mathrm{e}^{\mathrm{i}t\dot{H}^{\z}}\chi(\dot{H}^{\z})u\big\|_{\dot{\mathcal{E}}_{\mathscr{H}}^{\z}}+\big\|i_{+}\mathrm{e}^{\mathrm{i}t\dot{H}^{\z}}\chi(\dot{H}^{\z})u\big\|_{\dot{\mathcal{E}}_{\mathscr{I}}^{\z}}+\frac{\varepsilon}{6}\nonumber\\
		&\lesssim\big\|\mathrm{e}^{-\mathrm{i}t\dot{H}_{\mathscr{H}}^{\z}}i_{-}\mathrm{e}^{\mathrm{i}t\dot{H}^{\z}}\chi(\dot{H}^{\z})u\big\|_{\dot{\mathcal{E}}_{\mathscr{H}}^{\z}}+\big\|\mathrm{e}^{-\mathrm{i}t\dot{H}_{\mathscr{I}}^{\z}}i_{+}\mathrm{e}^{\mathrm{i}t\dot{H}^{\z}}\chi(\dot{H}^{\z})u\big\|_{\dot{\mathcal{E}}_{\mathscr{I}}^{\z}}+\frac{\varepsilon}{6}\nonumber\\
		&\lesssim\big\|\mathrm{e}^{-\mathrm{i}t\dot{H}_{\mathscr{H}}^{\z}}i_{-}\mathrm{e}^{\mathrm{i}t\dot{H}^{\z}}u\big\|_{\dot{\mathcal{E}}_{\mathscr{H}}^{\z}}+\big\|\mathrm{e}^{-\mathrm{i}t\dot{H}_{\mathscr{I}}^{\z}}i_{+}\mathrm{e}^{\mathrm{i}t\dot{H}^{\z}}u\big\|_{\dot{\mathcal{E}}_{\mathscr{I}}^{\z}}\nonumber\\
		&+\|\mathrm{e}^{-\mathrm{i}t\dot{H}_{\mathscr{H}}^{\z}}i_{-}\mathrm{e}^{\mathrm{i}t\dot{H}^{\z}}(1-\chi(\dot{H}^{\z}))u\big\|_{\dot{\mathcal{E}}_{\mathscr{H}}^{\z}}+\big\|\mathrm{e}^{-\mathrm{i}t\dot{H}_{\mathscr{I}}^{\z}}i_{+}\mathrm{e}^{\mathrm{i}t\dot{H}^{\z}}(1-\chi(\dot{H}^{\z}))u\big\|_{\dot{\mathcal{E}}_{\mathscr{I}}^{\z}}+\frac{\varepsilon}{6}.
	\end{align}
	By part 2. of Theorem \ref{Asymptotic completeness, geometric profiles}, we have for $t$ sufficiently large:
	\begin{align*}
		\|\mathrm{e}^{-\mathrm{i}t\dot{H}_{\mathscr{H}}^{\z}}i_{-}\mathrm{e}^{\mathrm{i}t\dot{H}^{\z}}u\big\|_{\dot{\mathcal{E}}_{\mathscr{H}}^{\z}}&\leq\big\|\boldsymbol{\Omega}^{f/p}_{\mathscr{H}}u\big\|_{\dot{\mathcal{E}}_{\mathscr{H}}^{\z}}+\frac{\varepsilon}{6},\qquad\qquad\big\|\mathrm{e}^{-\mathrm{i}t\dot{H}_{\mathscr{I}}^{\z}}i_{+}\mathrm{e}^{\mathrm{i}t\dot{H}^{\z}}u\big\|_{\dot{\mathcal{E}}_{\mathscr{I}}^{\z}}\leq\big\|\boldsymbol{\Omega}^{f/p}_{\mathscr{I}}u\big\|_{\dot{\mathcal{E}}_{\mathscr{I}}^{\z}}+\frac{\varepsilon}{6}.
	\end{align*}
	Furthermore,
	\begin{align*}
		\|\mathrm{e}^{-\mathrm{i}t\dot{H}_{\mathscr{H}\!/\!\mathscr{I}}^{\z}}i_{-\!/\!+}\mathrm{e}^{\mathrm{i}t\dot{H}^{\z}}(1-\chi(\dot{H}^{\z}))u\big\|_{\dot{\mathcal{E}}_{\mathscr{H}\!/\!\mathscr{I}}^{\z}}&\lesssim\|(1-\chi(\dot{H}^{\z}))u\big\|_{\dot{\mathcal{E}}^{\z}}
	\end{align*}
	by uniform boundedness of the evolutions as well as boundedness of $i_{-\!/\!+}:\dot{\mathcal{E}}^{\z}\to\dot{\mathcal{E}}^{\z}_{\mathscr{H}\!/\!\mathscr{I}}$ (\textit{cf.} Lemma \ref{Injection lemma}). Back into \eqref{Injection inequality 2}, we obtain:
	\begin{align}
	\label{Injection inequality 3}
		\big\|\chi(\dot{H}^{\z})u\big\|_{\dot{\mathcal{E}}^{\z}}&\lesssim\big\|\boldsymbol{\Omega}^{f/p}_{\mathscr{H}}u\big\|_{\dot{\mathcal{E}}_{\mathscr{H}}^{\z}}+\big\|\boldsymbol{\Omega}^{f/p}_{\mathscr{I}}u\big\|_{\dot{\mathcal{E}}_{\mathscr{I}}^{\z}}+2\|(1-\chi(\dot{H}^{\z}))u\big\|_{\dot{\mathcal{E}}^{\z}}+\frac{\varepsilon}{2}.
	\end{align}
	Plugging \eqref{Injection inequality 3} into \eqref{Injectivity first inequality} and letting the support of $\chi$ widespread enough in order to use \eqref{Strong limit chi(H/L)}, we obtain with assumption \eqref{Assumption injectivity}:
	\begin{align*}
		\big\|u\big\|_{\dot{\mathcal{E}}^{\z}}&\lesssim\big\|\boldsymbol{\Omega}^{f/p}_{\mathscr{H}}u\big\|_{\dot{\mathcal{E}}_{\mathscr{H}}^{\z}}+\big\|\boldsymbol{\Omega}^{f/p}_{\mathscr{I}}u\big\|_{\dot{\mathcal{E}}_{\mathscr{I}}^{\z}}+\varepsilon\lesssim\varepsilon.
	\end{align*}
	As $\varepsilon$ was arbitrary, we have shown \eqref{Sufficient condition kernel}. This completes the proof.
\end{proof}
%
%
%
%
%
%
\subsection{Traces on the energy spaces}
\label{Traces on the energy spaces}
Let $(\phi_{0},\phi_{1})\in\mathcal{C}^{\infty}_{\mathrm{c}}(\Sigma_{0})\times\mathcal{C}^{\infty}_{\mathrm{c}}(\Sigma_{0})$. By Leray's theorem (\textit{cf.} \cite{Le53}), there exists an unique solution $\phi\in\mathcal{C}^{\infty}(\M)$ of
\begin{align}
\label{IVP Trace}
\begin{cases}
\Box_{\g}\phi=0\\
\phi_{\vert\Sigma_{0}}=\phi_{0}\\
(-\mathrm{i}\partial_{t}\phi)_{\vert\Sigma_{0}}=\phi_{1}
\end{cases}.
\end{align}
Moreover, $\phi$ extends to a smooth function $\hat{\phi}\in\mathcal{C}^{\infty}(\MM)$. In particular, $\hat{\phi}$ has traces $(\xi^{\pm},\zeta^{\pm})\in\mathcal{C}^{\infty}(\mathscr{H}^{\pm})\times\mathcal{C}^{\infty}(\mathscr{I}^{\pm})$. The \textit{future/past trace operators} are then defined on smooth compactly supported data by
\begin{align}
\label{Trace operators def}
\mathscr{T}^{\pm}&:\mathcal{C}^{\infty}_{\mathrm{c}}(\Sigma_{0})\times\mathcal{C}^{\infty}_{\mathrm{c}}(\Sigma_{0})\ni(\phi_{0},\phi_{1})\longmapsto(\xi^{\pm},\zeta^{\pm})\in\mathcal{C}^{\infty}(\mathscr{H}^{\pm})\times\mathcal{C}^{\infty}(\mathscr{I}^{\pm}).
\end{align}
The purpose of this Subsection is to extend the traces on the asymptotic energy spaces $\dot{\mathscr{E}}_{\pm}^{\z}$. To do so, we will use completeness of wave operators.

Using the identification diffeomorphisms of Subsection \ref{Energy spaces on the horizons}, we first link traces on the horizons to the operators $\OO^{\pm}$:
\begin{lemma}[Pointwise traces]
	\label{Link with inverse wave operators}
	Define the isometries
	\begin{align*}
	\mathcal{U}^{\pm}&:=\begin{pmatrix}
	(\mathscr{F}_{\mathscr{H}}^{\pm})^{*}&0\\0&(\mathscr{F}_{\mathscr{I}}^{\pm})^{*}
	\end{pmatrix}\in\mathcal{B}\big(\dot{\mathcal{H}}^{1}_{\mathscr{H}}\times\dot{\mathcal{H}}^{1}_{\mathscr{I}},\dot{\mathscr{E}}_{\pm}\big).
	\end{align*}
	For all $\phi=(\phi_{0},\phi_{1})\in\mathcal{C}^{\infty}_{\mathrm{c}}(\Sigma_{0})\times\mathcal{C}^{\infty}_{\mathrm{c}}(\Sigma_{0})$,
	\begin{align*}
	\mathscr{T}^{\pm}\begin{pmatrix}
	\phi_{0}\\\phi_{1}
	\end{pmatrix}&=\mathcal{U}^{\pm}\OO^{\pm}\begin{pmatrix}
	\phi_{0}\\\phi_{1}
	\end{pmatrix}.
	\end{align*}
\end{lemma}
\begin{proof}
	We only treat the $+$ case. Let $\phi=(\phi_{0},\phi_{1})\in\mathcal{C}^{\infty}_{\mathrm{c}}(\Sigma_{0})\times\mathcal{C}^{\infty}_{\mathrm{c}}(\Sigma_{0})$ and set $(v_{0}(t),v_{1}(t)):=\mathrm{e}^{\mathrm{i}t\dot{H}^{\z}}(\phi_{0},\phi_{1})$. The operators $(\mathscr{F}_{\mathscr{H}\!/\!\mathscr{I}}^{+})^{*}\mathrm{e}^{-\mathrm{i}t\dot{\mathbb{H}}_{\mathscr{H}\!/\!\mathscr{I}}}$ carry data onto the future horizons along principal null geodesics, so we have
	\begin{align*}
	\Omega(t)\begin{pmatrix}
	\phi_{0}\\\phi_{1}
	\end{pmatrix}&=\begin{pmatrix}
	(\mathscr{F}_{\mathscr{H}}^{+})^{*}(j_{-}v_{0}(t))\circ\gamma_{\textup{in}}(-t)\\
	(\mathscr{F}_{\mathscr{I}}^{+})^{*}(j_{+}v_{0}(t))\circ\gamma_{\textup{out}}(-t)
	\end{pmatrix}\\
	&=\begin{pmatrix}
	(v_{0})_{\vert\mathscr{H}^{+}}\\(v_{0})_{\vert\mathscr{I}^{+}}
	\end{pmatrix}\\
	&=\mathscr{T}^{+}\begin{pmatrix}
	\phi_{0}\\\phi_{1}
	\end{pmatrix}
	\end{align*}
	for $t\gg 0$ since $i_{-\!/\!+}\equiv 1$ near $\mathscr{H}\!/\!\mathscr{I}$.
\end{proof}
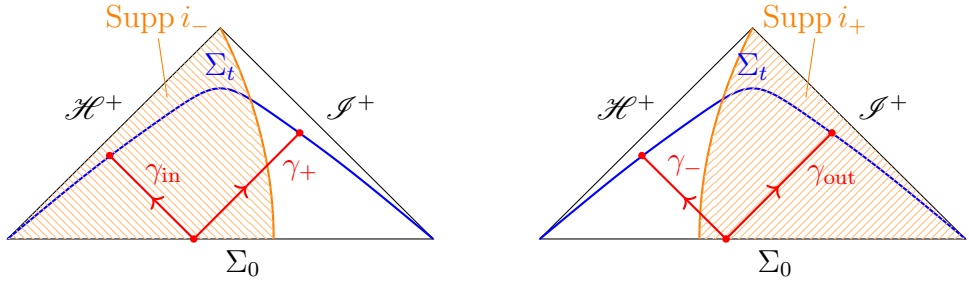
\begin{figure}[!h]
	\centering
	\captionsetup{justification=centering,margin=1.25cm}
	\begin{tikzpicture}[scale=0.70]
	\draw[](-4,0)--(0,4)--(4,0)--cycle;
	\draw[orange](-1.175,4.16)node[]{$\mathrm{Supp\,}i_{-}$};
	\draw[orange, -](-1.165,3.8)--(-1.0,2.68);
	%
	%
	\begin{scope} [shift={(0.0,0.0)}]
	\draw [blue, thick] plot [rotate  = 0.0, domain = -4.0:0.0, samples = 120] (\x,{0.84*(4-sqrt(2)*sqrt(ln(exp(\x*\x/2)+0.2*\x*\x+0.2)))});
	\draw [blue, thick] plot [rotate  = 0.0, domain = 0.0:4.0, samples = 120] (\x,{0.84*(4-sqrt(2)*sqrt(ln(exp(\x*\x/2)+0.2*\x*\x+0.2)))});
	\end{scope}
	%
	%
	\begin{scope} [shift={(0.0,0.0)}]
	\draw [orange, thick] plot [cm={cos(-90) ,-sin(-90) ,sin(-90) ,cos(-90) ,(0.0, 0.0)}, domain = 0.0:4.0, samples = 80] (\x, {(\x*\x-16)/16});
	\end{scope}
	\path[pattern = north west lines, pattern color = orange!50] plot [cm={cos(-90) ,-sin(-90) ,sin(-90) ,cos(-90) ,(0.0, 0.0)}, domain = 0.0:4.0, samples = 80] (\x, {(\x*\x-16)/16}) -- (-4,0) -- cycle;
	\draw[red](-1.07,1.235)node[]{$\gamma_{\textup{in}}$};
	\draw[red](1.5,1.30)node[]{$\gamma_{+}$};
	\draw[blue](0.0,2.8)node[above]{$\Sigma_{t}$};
	\draw(0.425,0.0)node[below]{$\Sigma_{0}$};
	\draw(-1.6,2.5)node[left]{$\mathscr{H}^{+}$};
	\draw(1.77,2.5)node[right]{$\mathscr{I}^{+}$};
	\begin{scope}[very thick,decoration={
		markings,
		mark=at position 0.5 with {\arrow{>}}}
	] 
	\draw[red, thick, postaction={decorate}](-0.5,0)--(-2.08,1.58);
	\draw[red, thick, postaction={decorate}](-0.5,0)--(1.49,2.01);
	\end{scope}
	\filldraw[red, thick](-0.5,0)circle[x radius=0.05, y radius=0.05];
	\filldraw[red, thick](-2.08,1.58)circle[x radius=0.05, y radius=0.05];
	\filldraw[red, thick](1.49,2.01)circle[x radius=0.05, y radius=0.05];
	%
	%
	%
	%
	\draw[](6,0)--(10,4)--(14,0)--cycle;
	\draw[orange](1.175++10.0,4.16)node[]{$\mathrm{Supp\,}i_{+}$};
	\draw[orange, -](1.165++10.0,3.8)--(1.0+10.0,2.68);
	%
	%
	\begin{scope} [shift={(10.0,0.0)}]
	\draw [blue, thick] plot [rotate  = 0.0, domain = -4.0:0.0, samples = 120] (\x,{0.84*(4-sqrt(2)*sqrt(ln(exp(\x*\x/2)+0.2*\x*\x+0.2)))});
	\draw [blue, thick] plot [rotate  = 0.0, domain = 0.0:4.0, samples = 120] (\x,{0.84*(4-sqrt(2)*sqrt(ln(exp(\x*\x/2)+0.2*\x*\x+0.2)))});
	\end{scope}
	%
	%
	\begin{scope} [shift={(10.0,0.0)}]
	\draw [orange, thick] plot [cm={cos(-90) ,-sin(-90) ,sin(-90) ,cos(-90) ,(0.0, 0.0)}, domain = 0.0:4.0, samples = 80] (\x, {-(\x*\x-16)/16});
	\end{scope}
	\path[pattern = north east lines, pattern color = orange!50] plot [cm={cos(-90) ,-sin(-90) ,sin(-90) ,cos(-90) ,(10.0, 0.0)}, domain = 0.0:4.0, samples = 80] (\x, {-(\x*\x-16)/16}) -- (4+10.0,0) -- cycle;
	\draw[red](-1.07+9.8,1.385)node[]{$\gamma_{-}$};
	\draw[red](1.5+10.0,1.2)node[]{$\gamma_{\textup{out}}$};
	\draw[blue](0.0+10.0,2.8)node[above]{$\Sigma_{t}$};
	\draw(0.425+10.0,0.0)node[below]{$\Sigma_{0}$};
	\draw(-1.6+10.0,2.5)node[left]{$\mathscr{H}^{+}$};
	\draw(1.77+10.0,2.5)node[right]{$\mathscr{I}^{+}$};
	\begin{scope}[very thick,decoration={
		markings,
		mark=at position 0.5 with {\arrow{>}}}
	] 
	\draw[red, thick, postaction={decorate}](-0.5+10.0,0)--(-2.08+10.0,1.58);
	\draw[red, thick, postaction={decorate}](-0.5+10.0,0)--(1.49+10.0,2.01);
	\end{scope}
	\filldraw[red, thick](-0.5+10.0,0)circle[x radius=0.05, y radius=0.05];
	\filldraw[red, thick](-2.08+10.0,1.58)circle[x radius=0.05, y radius=0.05];
	\filldraw[red, thick](1.49+10.0,2.01)circle[x radius=0.05, y radius=0.05];
	\end{tikzpicture}
	\caption{\label{Explanation Lemma smooth}Transports from $\Sigma_{0}$ onto $\Sigma_{t}$ along the principal null geodesics $\gamma_{\textup{in/out}}$ and the curves $\gamma_{+\!/\!-}$. Data reaching horizons are carried only by $\gamma_{\textup{in/out}}$.}
\end{figure}
Combining Lemma \ref{Link with inverse wave operators} with Theorem \ref{Inversion of full wave operators}, we obtain:
\begin{theorem}[Extension of the traces]
	\label{Extension of the traces}
	Let $\z\in\mathbb{Z}\setminus\{0\}$. There exists $s_{0}>0$ such that for all $s\in\left]-s_{0},s_{0}\right[$, the traces extend to energy spaces as bounded invertible operators:
	\begin{align*}
		\mathscr{T}^{\pm}&=\mathcal{U}^{\pm}\OO^{\pm}\in\mathcal{B}\big(\dot{\mathcal{E}}^{\z},\dot{\mathscr{E}}_{\pm}^{\z}\big),\\
		(\mathscr{T}^{\pm})^{-1}&=\mathbb{W}^{\pm}(\mathcal{U}^{\pm})^{-1}\in\mathcal{B}\big(\dot{\mathscr{E}}_{\pm}^{\z},\dot{\mathcal{E}}^{\z}\big).
	\end{align*}
\end{theorem}
\begin{remark}
\label{Geometric interpretation of full wave operators}
	The regularity of elements in the energy spaces does \textup{a priori} not ensure the existence of the traces for general solutions of the extended wave equation. Theorem \ref{Extension of the traces} shows that they exist thanks to the completeness of the wave operators.
	
	Theorem \ref{Extension of the traces} also provides the geometric interpretation of the full wave operators as inverses of the traces on horizons. Both are linked by the transformations $\mathcal{U}^{\pm}$ which identify points on horizons and $\Sigma_{0}$ via transport along principal null geodesics.
\end{remark}
%
%
%
%
%
\subsection{Solution of the Goursat problem}
\label{Solution to the Goursat problem}
The Goursat problem consists in an inverse problem on the global outer space $(\MM,\g)$ constructed in Subsection \ref{Crossing rings}. Given boundary data $(\xi^{\pm},\zeta^{\pm})\in\mathcal{C}^{\infty}_{\mathrm{c}}(\mathscr{H}^{\pm})\times\mathcal{C}^{\infty}_{\mathrm{c}}(\mathscr{I}^{\pm})$, we are asked to find $\phi:\MM\to\mathbb{C}^{2}$ solving the wave equation \eqref{Extended wave equation} and such that
\begin{align}
\label{Boundary data}
\phi_{\vert\mathscr{H}^{\pm}}&=\xi^{\pm},\qquad\qquad\phi_{\vert\mathscr{I}^{\pm}}=\zeta^{\pm}.
\end{align}
The Goursat problem is linked to the trace operators as being the inverse procedure of taking the trace of a solution of equation \eqref{Extended wave equation}. The analytic scattering theory solves this problem by constructing the inverse wave operators. See the paper of Nicolas \cite[Remark 1.1 \& Section 4]{Ni15} for some discussions about the different points of view of the scattering.

Theorem \ref{Extension of the traces} allows us to solve the following abstract Goursat problem:
\begin{theorem}
	\label{Abstract Goursat problem}
	Let $\z\in\mathbb{Z}\setminus\{0\}$. There exists $s_{0}>0$ such that for all $s\in\left]-s_{0},s_{0}\right[$ the following property: there exist homeomorphisms
	\begin{align*}
	\mathbb{T}^{\pm}&:\dot{\mathcal{E}}^{\z}\longrightarrow\dot{\mathscr{E}}^{\z}_{\pm}
	\end{align*}
	solving the Goursat problem \eqref{Boundary data} in the energy spaces, that is, for all $(\xi^{\pm},\zeta^{\pm})\in\dot{\mathscr{E}}^{\z}_{\pm}$, there exists an unique $\phi\in\mathcal{C}^{0}(\mathbb{R}_{t};\dot{\mathcal{E}}^{\z})$ solving the wave equation on $(\M,\g)$ with initial data $\phi(0)=(\phi_{0},\phi_{1})$ such that
	\begin{align*}
	(\xi^{\pm},\zeta^{\pm})&=\mathbb{T}^{\pm}(\phi_{0},\phi_{1}).
	\end{align*}
\end{theorem}
\begin{remark}
	In the standard case of the wave equation on De Sitter-Schwarzschild spacetime \cite{Ni15}, the traces extend as unitary operators between energy spaces defined on $\Sigma_0$ and on the horizons. Here, we only have bounded extensions because of the superradiance. In particular, we have the following control of the energies:
	\begin{align*}
		\frac{1}{C}\|\phi(0)\|_{\dot{\mathcal{E}}^{\z}}^{2}&\leq\|(\xi^{\pm},\zeta^{\pm})\|_{\dot{\mathscr{E}}^{\z}_{\pm}}^{2}\leq C\|\phi(0)\|_{\dot{\mathcal{E}}^{\z}}^{2}
	\end{align*}
	for some constant $C>0$.
\end{remark}
%
%
%
%
%
%
%
%
%
%
%
%
%
%
%
%
%
%
%
%
%
%
%
%
%
%
%
%
%
\bibliographystyle{abbrv}

\footnotesize
\begin{thebibliography}{30}
	%
	\bibitem[Ba97]{Ba97}
	Bachelot, A.: \emph{Scattering of scalar fields by spherical gravitational collapse}, J. Math. Pures Appl. (9)\textbf{76}, 155–210 (1997)
	%
	\bibitem[Ba99]{Ba99}
	Bachelot, A.: \emph{The Hawking effect}, Ann. Inst. H. Poincar\'{e} Phys. Th\'{e}or. \textbf{70}, 41–99 (1999)
	%
	\bibitem[Ba00]{Ba00}
	Bachelot., A.: \emph{Creation of fermions at the charged black-hole horizon}, Ann. Henri	Poincar\'{e}, 1(6):1043–1095 (2000)
	%
	\bibitem[Ba04]{Ba}
	Bachelot, A.: \textit{Superradiance and scattering of the charged Klein–Gordon field by a step-like electrostatic potential}, J. Math. Pures et Appl., \textbf{83}, 1179-1239 (2004)
	%
	\bibitem[BM93]{BaMo93}
	Bachelot, A., Motet-Bachelot, A.: \textit{Les r\'esonances d’un trou noir de Schwarzschild}, Ann. Inst. H. Poincar\'e, Phys. Th\'eor. \textbf{59}(1), 3–68 (1993)
	%
	\bibitem[Be19]{Be18}
	Besset, N.: \textit{Decay of the Local Energy for the Charged Klein-Gordon Equation in the Exterior De Sitter-Reissner-Nordstr\"{o}m Spacetime}, arXiv:1812.09390 [math-ph] (2019)
	%
	\bibitem[BH08]{BoHa08}
	Bony, J.-F., H\"{a}fner, D.: \textit{Decay and non-decay of the local energy for the wave equation on the De Sitter-Schwarzschild metric}, Comm. Math. Phys., \textbf{282}(3), 697-719 (2008)
	%
	\bibitem[DRSR18]{DaRoSR18}
	Dafermos, M., Rodnianski, I., Shlapentokh-Rothman, Y.: \textit{A scattering theory for the	wave equation on Kerr black hole exteriors}, Ann. Sci. ENS \textbf{51}(2), 371-486 (2018)
	%
	\bibitem[DH72]{DH}
	Duistermaat J., H\"{o}rmander L.: \textit{Fourier integral operators II}, Acta	Mathematica, \textbf{128}(1), 183–269 (1972)
	%
	\bibitem[E03]{El03}
	Elvang, H.: \textit{A Charged rotating black ring}, Phys.Rev. D68 (2003)
	%
	\bibitem[ER02]{EmRe02}
	Emparan, R., Reall, H. S.: \textit{A Rotating Black Ring Solution in Five Dimensions}, Phys. Rev. Lett. \textbf{88}:101101 (2002)
	%
	\bibitem[ER08]{EmRe08}
	Emparan, R., Reall, H. S.: \textit{A Rotating Black Ring Solution in Five Dimensions}, Living Rev. Rel. \textbf{11}:6 (2008)	
	%
	\bibitem[GGH13]{GGH13}
	Georgescu, V., Gérard, C., H\"{a}fner, D.: \textit{Boundary values of resolvents of self-adjoint operators in Krein spaces}, J. Funct. Anal. \textbf{265}, 3245–3304 (2013)
	%
	\bibitem[GGH15]{GGH15}
	Georgescu, V., Gérard, C., H\"{a}fner, D.: \textit{Resolvent and propagation estimates for Klein–Gordon equations with non-positive energy}, J. Spectr. Theory \textbf{5}, 113–192 (2015)
	%
	\bibitem[GGH17]{GGH17}
	Georgescu, V., G\'{e}rard, C. and H\"{a}fner, D.: \textit{Asymptotic completeness for superradiant Klein–Gordon equations and applications to the De Sitter–Kerr metric}, J. Eur. Math. Soc. 19, 2171–2244 (2017)
	%
	\bibitem[Ge12]{Ge12}
	G\'{e}rard, C.: \textit{Scattering theory for Klein–Gordon equations with non-positive energy}, Ann. Henri Poincar\'{e} \textbf{13}, 883–941 (2012)
	%
	\bibitem[Gu04]{Gu04}
	Guillarmou, C.: \textit{Meromorphic Properties of the Resolvent on Asymptotically Hyperbolic Manifolds}, Duke Math. J. \textbf{129}(1), 1-37 (2005).
	%
	\bibitem[H01]{Ha02}
	H\"{a}fner, D.: \textit{Compl\'{e}tude asymptotique pour l’\'{e}quation des ondes dans une classe d’espaces-temps stationnaires et asymptotiquement plats}, Ann. Inst. Fourier \textbf{51}, 779–833 (2001)
	%
	\bibitem[H03]{Ha03}
	H\"{a}fner, D.: \textit{Sur la th\'{e}orie de la diffusion pour l’\'{e}quation de Klein-Gordon dans la m\'{e}trique de Kerr}, Dissertationes Mathematicae \textbf{421} (2003)
	%
	\bibitem[H09]{Ha09}
	H\"{a}fner, D.: \textit{Creation of fermions by rotating charged black holes}, M\'{e}moires de la SMF \textbf{117}, 158 pp. (2009)
	%
	\bibitem[HN04]{HaNi04}
	H\"{a}fner, D., Nicolas, J.-P.: \textit{Scattering of massless Dirac fields by a Kerr black hole}, Rev. in Math. Phys. Vol. 16, No. 1, 29-123 (2004)
	%
	\bibitem[Hi18]{Hi18}
	Hintz, P.: \textit{Non-linear Stability of the Kerr–Newman–de Sitter Family of Charged Black Holes}, Annals of PDE \textbf{4} (2016)
	%
	\bibitem[Ka21]{Ka21}
	Kaluza, T.: \textit{Zum Unitätsproblem in der Physik}, Sitzungsber. Preuss. Akad. Wiss. Berlin. (Math. Phys.), 966–972 (1921)
	%
	\bibitem[Kl26]{Kl26}
	Klein, O.: \textit{Quantentheorie und fünfdimensionale Relativitätstheorie}, Zeitschrift für Physik A. \textbf{37}(12), 895–906 (1926)
	%
	\bibitem[L53]{Le53}
	Leray, J.: \textit{Hyperbolic differential equations}, The Institute for Advanced Study, Princeton, N. J., 240 pages (1953).
	%
	\bibitem[Mi]{Mi}
	Millet, P.: \textit{The Goursat problem at the horizons for the Klein-Gordon equation on the De Sitter Kerr metric}, in preparation
	%
	\bibitem[MM87]{MaMe87}
	Mazzeo, R., Melrose, R.: \textit{Meromorphic extension of the resolvent on complete spaces with asymptotically
	constant negative curvature}, J. Funct. Anal. \textbf{75}(2), 260–310 (1987)
	%
	%
	%
	\bibitem[N15]{Ni15}
	Nicolas, J.-P.: \textit{Conformal scattering on the Schwarzschild metric}, Ann. Institut Fourier
	Tome \textbf{66}, No. 3, 1175-1216 (2016)
	%
	\bibitem[ON95]{ON}
	O'Neill, B.: \textit{The Geometry of Kerr Black Holes}, A.K. Peters, Wellesley (1995)
	%
	\bibitem[RS80]{ReSi}
	Reed, M., Simon, B.: \textit{Methods of Modern Mathematical Physics Vol. I}, New York Academic Press, (1980)
	%
	\bibitem[W84]{Wa}
	Wald, R. M.: \textit{General Relativity}, The University of Chicago Press (1984)
	%
	%
\end{thebibliography}
%

%
%
%
%
%
%
%
%
%
%
%
%
%
%
%
%
%
%
\end{document}